\documentclass[12pt, letterpaper,table]{amsart}
\usepackage[margin = 2.5cm]{geometry}
\usepackage{
amsmath,
amssymb,
amsthm,
color,
verbatim,
ytableau,
bm,
booktabs,
mathrsfs,
colortbl,
thmtools,
thm-restate,
enumitem,
mathtools,
bbm,
shuffle,
comment,
float
}
\usepackage[
  separate-uncertainty = true,
  multi-part-units = repeat
]{siunitx}
\usepackage{tikz}
\usetikzlibrary{arrows.meta, decorations.markings}
\usetikzlibrary{matrix,arrows}

\newtheorem{innercustomthm}{Theorem}
\newenvironment{customthm}[1]
  {\renewcommand\theinnercustomthm{#1}\innercustomthm}
  {\endinnercustomthm}

\definecolor{gold}{RGB}{255,207,64}

\usepackage{hyperref}
\hypersetup{
    colorlinks=true,
    linkcolor=magenta,
    filecolor=blue,
    urlcolor=teal,
    citecolor=teal
}
\usepackage[nameinlink]{cleveref}

\newcommand{\C}{\mathbb{C}}

\newcommand{\Z}{\mathbb{Z}}

\newcommand{\boldc}{{\mathbf{c}}}

\newcommand{\bolde}{{\mathbf{e}}}

\newcommand{\hatomega}{\widehat{\omega}}

\newcommand{\Hom}{{\mathrm{Hom}}}

\newcommand{\sgn}{{\mathrm{sgn}}}

\newcommand{\pp}{\mathrm{PP}}
\newcommand{\trip}{\mathsf{trip}}
\newcommand{\tripstrand}{\mathsf{trip}}
\newcommand{\sep}{\mathsf{sep}}

\newcommand{\barone}{\overline{1}}
\newcommand{\bartwo}{\overline{2}}
\newcommand{\barthree}{\overline{3}}
\newcommand{\barfour}{\overline{4}}

\newcommand{\spp}{\mathrm{SPP}}
\newcommand{\cspp}{\mathrm{CSPP}}
\newcommand{\scpp}{\mathrm{SCPP}}
\newcommand{\tspp}{\mathrm{TSPP}}
\newcommand{\tsscpp}{\mathrm{TSSCPP}}

\newcommand{\slr}{\mathfrak{sl}_r}
\renewcommand{\sl}{\mathfrak{sl}}

\newcommand{\subs}{\subseteq}

\renewcommand{\b}{\partial}

\newtheorem*{main-theorem}{Main Theorem}
\newtheorem{thm}{Theorem}[section]
\newtheorem{lemma}[thm]{Lemma}
\newtheorem{prop}[thm]{Proposition}
\newtheorem{cor}[thm]{Corollary}

\theoremstyle{definition}
\newtheorem{defn}[thm]{Definition}

\newtheorem{example}[thm]{Example}
\newtheorem{remark}[thm]{Remark}
\newtheorem{algorithm}[thm]{Algorithm}

\begin{document}

\title[Webification of Symmetry classes of plane partitions]{Webification of Symmetry classes of plane partitions}

\author{Ashleigh Adams-Jessmore}
\email{ashleigh.adams@ndsu.edu}
\address{School of Mathematics\\North Dakota State University\\Fargo, ND 58102}

\author{Jessica Striker}
\email{jessica.striker@ndsu.edu}
\address{School of Mathematics\\North Dakota State University\\Fargo, ND 58102}



\begin{abstract}
    Webs are graphical objects that give a tangible, combinatorial way to compute and classify tensor invariants.
    Recently, Gaetz, Pechenik, Pfannerer, Striker, and Swanson found a rotation-invariant web basis for $\mathrm{SL}_4$, as well as its quantum deformation $U_q(\sl_4)$, and a bijection between move equivalence classes of $U_q(\sl_4)$-webs and fluctuating tableaux such that web rotation corresponds to tableau promotion.
    They also found a bijection between the set of plane partitions in an $a\times b\times c$ box and a benzene move equivalence class of $U_q(\sl_4)$-webs by determining the corresponding oscillating tableau. In this paper, we similarly find the oscillating tableaux corresponding to plane partitions in certain symmetry classes. We furthermore give a combinatorial map, defined on the span of the invariants of webs arising from certain symmetry classes, from $U_q(\sl_4)$ invariants to $U_q(\sl_r)$ invariants for $r=2,3$.
\end{abstract}

\maketitle

\section{Introduction}
$U_q(\slr)$-webs form a powerful graphical calculus for computing $U_q(\slr)$ invariants. Combinatorially, $U_q(\slr)$-webs can be interpreted as bi-colored planar graphs on a disk, called \emph{plabic graphs} \cite{hopkins2022promotion}.
Postnikov \cite{postnikov2006total} introduced a tool called \emph{trip permutations}, in which one walks along the edges of a
plabic graph, taking a right at every black vertex and a left at every white vertex. This definition was extended in \cite{gaetz2023rotation}, wherein they define a collection of walks along the edges of an \emph{hourglass plabic graph} (Definition~\ref{def:hourglass-plabic-graph}). The $i$th trip, denoted $\trip_i,$ is the walk which takes the $i$th right at every black vertex and the $i$th left at every white vertex (see Definition \ref{def:trip permutations on hgpgs}).

With this generalized definition of trip permutations, \cite{gaetz2023rotation} define the \emph{separation label} of each edge, with values in $\{1,\ldots, r\}$ (Definition \ref{def: separation-labelings}) and use the separation labels of the boundary edges (the \emph{boundary word}) to
construct a unique $r$-row \emph{fluctuating tableau}  \cite[Definition~2.1]{gaetz2023promotion}.

The authors showed that fluctuating tableaux index equivalence classes of \emph{contracted}, \emph{fully-reduced}, \emph{top} hourglass plabic graphs.
More importantly, they show that these hourglass plabic graphs form a rotation-invariant web-basis for the space of $U_q(\sl_4)$-invariants \cite[Theorem~A]{gaetz2023rotation}.
The combinatorial heart of this result lies in \cite[Theorem~B]{gaetz2023rotation}, where the authors show that there is a bijection between $4$-row rectangular fluctuating tableaux and equivalence classes of $U_q(\sl_4)$-webs. Moreover, they showed that this bijection relates trip permutations on the hourglass plabic graph to a certain dynamical action on fluctuating tableaux, called promotion.

In a surprising observation, \cite{gaetz2023rotation} also found a $U_q(\sl_4)$-web equivalence class in bijection with the set of \emph{plane partitions} in an $a\times b\times c$ box.
In this paper, we study \emph{symmetry classes} of plane partitions from the perspective of webs. We focus on the following symmetry classes: symmetric ($\spp$), cyclically symmetric ($\cspp$), totally symmetric ($\tspp$), and totally symmetric self-complementary ($\tsscpp$).
A plane partition in each symmetry class is determined by its restriction to a \emph{fundamental domain}. We give more explicit descriptions of these in Section \ref{subsec: symmetry-classes-of-pp} and give an example in Figure \ref{fig:symmetry-classes-example}.

\begin{figure*}[ht]
    \centering
    \includegraphics[scale=0.3]{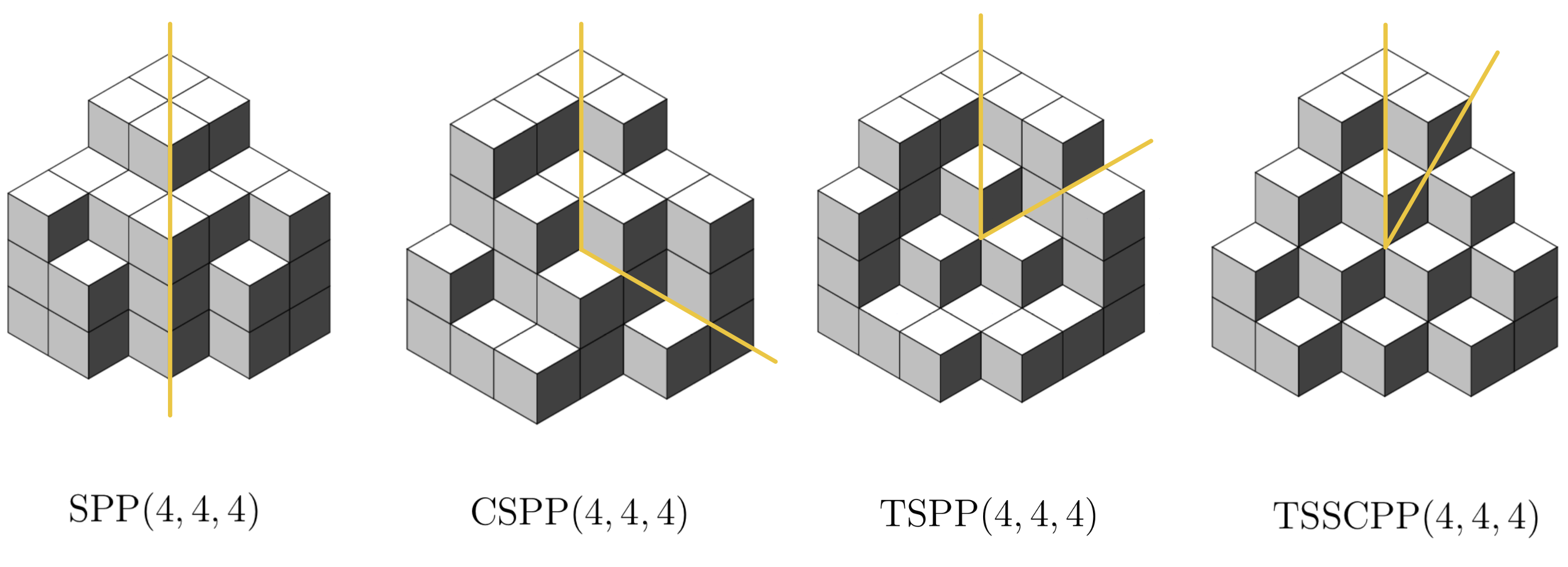}
    \caption{
    Examples of plane partitions inside a $4\times 4\times 4$-box, with
    the symmetry class listed below it and
    its corresponding fundamental domain.
    See Section \ref{subsec: symmetry-classes-of-pp} for details.}
    \label{fig:symmetry-classes-example}
\end{figure*}

The two main objects of interest in this paper are $U_q(\sl_4)$-webs (viewed as hourglass plabic graphs) and plane partitions.
The webs that are in bijection with plane partitions are certain webs where each face is a hexagon and the hourglass edges give a perfect matching.
A \emph{benzene move} is a local move on a hexagonal face whose edges alternate between hourglass edges and simple edges. It exchanges the hourglass edges and the simple edges of the face and leaves the rest of the web unchanged. (See Figure \ref{fig:benzene-move}.)
In the bijection with plane partitions, a benzene move on the web results in the addition or removal of a box in the plane partition.
We say that two webs reside in the same benzene equivalence class if they differ only by a sequence of benzene moves.
Benzene moves do not change the separation labels of the boundary edges.
Thus, each equivalence class corresponds to a single boundary word.

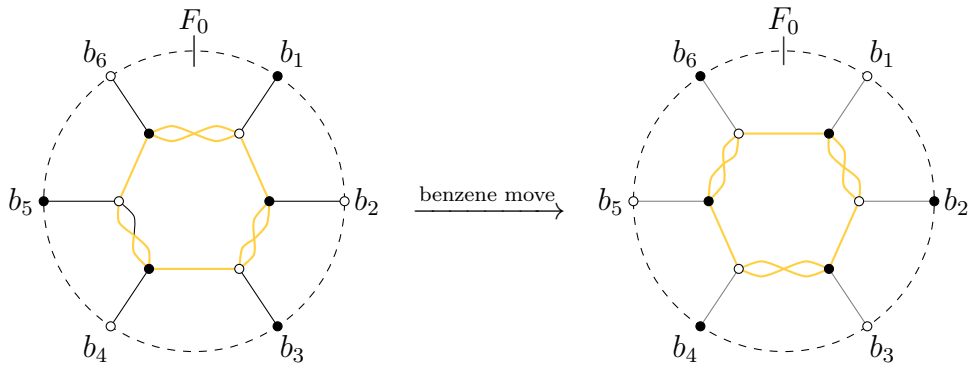
\begin{figure}
\centering
\begin{tikzpicture}[scale = 0.4]
\begin{scope}[xshift=-9.8cm]
\draw[dashed] (0,0) circle (5cm);

\node at (0,6) {\small $F_0$};
\node at (0,5) {$|$};

\draw[color=gold,thick] (1.5,2.25)  -- (2.5,0); 
\draw[color=gold,thick] (1.5,-2.25) -- (-1.5,-2.25); 
\draw[color=gold,thick] (-2.5,0)    -- (-1.5,2.25); 

\draw (2.7735,4.16025)    -- (1.5,2.25);
\draw (-2.7735, -4.16025) -- (-1.5,-2.25);
\draw (-2.7735, 4.16025)  -- (-1.5,2.25);
\draw (2.7735, -4.16025)  -- (1.5,-2.25);
\draw (2.5,0)             -- (5,0);
\draw (-2.5,0)            -- (-5,0);

\node[xshift = .2cm, yshift= .3cm]   at (2.7735,4.16025)   {$b_1$};
\node[xshift = .3cm]                 at (5,0)              {$b_2$};
\node[xshift = .2cm, yshift= -.3cm]  at (2.7735,-4.16025)  {$b_3$};
\node[xshift = -.2cm, yshift= -.3cm] at (-2.7735,-4.16025) {$b_4$};
\node[xshift = -.3cm]                at (-5,0)             {$b_5$};
\node[xshift = -.2cm, yshift= .3cm]  at (-2.7735,4.16025)  {$b_6$};

\draw[color=gold,thick] plot [smooth] coordinates {(-1.5,2.25) (-.75,2.5) (0,2.25)};
\draw[color=gold,thick] plot [smooth] coordinates {(0,2.25) (.75,2.5) (1.5,2.25)};
\draw[color=gold,thick] plot [smooth] coordinates {(-1.5,2.25) (-.75,2) (0,2.25)};
\draw[color=gold,thick] plot [smooth] coordinates {(0,2.25) (.75,2) (1.5,2.25)};

\draw[color=gold,thick] plot [smooth] coordinates {(2.5,0) (2.42,-0.762) (2,-1.25)};
\draw[color=gold,thick] plot [smooth] coordinates {(2,-1.26) (1.98,-1.79) (1.5,-2.25)};
\draw[color=gold,thick] plot [smooth] coordinates {(2.5,0) (2,-.6) (2,-1.25)};
\draw[color=gold,thick] plot [smooth] coordinates {(2,-1.26) (1.6,-1.56) (1.5,-2.25)};

\draw[color=gold,thick] plot [smooth] coordinates { (-1.5,-2.25) (-1.52,-1.6) (-2,-1.125) };
\draw[color=gold,thick] plot [smooth] coordinates { (-2,-1.125) (-1.999,-.5) (-2.5,0)};
\draw[color=gold,thick] plot [smooth] coordinates { (-1.5,-2.25) (-2,-1.7) (-2,-1.125) };
\draw[color=gold,thick] plot [smooth] coordinates { (-2,-1.125) (-2.5,-.69) (-2.5,0)};

\draw[fill=white] (1.5,2.25)   circle (0.15cm);
\draw[fill=black] (2.5,0)      circle (0.15cm);
\draw[fill=white] (1.5,-2.25)  circle (0.15cm);
\draw[fill=black] (-1.5,-2.25) circle (0.15cm);
\draw[fill=white] (-2.5,0)     circle (0.15cm);
\draw[fill=black] (-1.5,2.25)  circle (0.15cm);

\draw[fill=black] (2.7735,4.16025)   circle (0.15cm);
\draw[fill=white] (5,0)              circle (0.15cm);
\draw[fill=black] (2.7735,-4.16025)  circle (0.15cm);
\draw[fill=white] (-2.7735,-4.16025) circle (0.15cm);
\draw[fill=black] (-5,0)             circle (0.15cm);
\draw[fill=white] (-2.7735,4.16025)  circle (0.15cm);

\end{scope}
\begin{scope}
    \node[] at (0,0) {$\xrightarrow{\text{benzene move}}$};
\end{scope}
\begin{scope}[xshift=9.8cm,rotate=180]
    \draw[dashed] (0,0) circle (5cm);

\node at (0,-6) {$F_0$};
\node at (0,-5) {\small $|$};

\draw[color=gold,thick] (1.5,2.25) -- (2.5,0); 
\draw[color=gold,thick] (1.5,-2.25) -- (-1.5,-2.25); 
\draw[color=gold,thick] (-2.5,0) -- (-1.5,2.25); 

\draw[color=gray] (2.7735,4.16025) -- (1.5,2.25);
\draw[color=gray] (-2.7735, -4.16025) -- (-1.5,-2.25);
\draw[color=gray] (-2.7735, 4.16025) -- (-1.5,2.25);
\draw[color=gray] (2.7735, -4.16025) -- (1.5,-2.25);
\draw[color=gray] (2.5,0) -- (5,0);
\draw[color=gray] (-2.5,0) -- (-5,0);

\node[xshift = -.2cm, yshift= -.3cm] at (2.7735,4.16025)   {$b_4$};
\node[xshift = -.3cm]                at (5,0)              {$b_5$};
\node[xshift = -.2cm, yshift= .3cm]  at (2.7735,-4.16025)  {$b_6$};
\node[xshift = .2cm, yshift= .3cm]   at (-2.7735,-4.16025) {$b_1$};
\node[xshift = .3cm]                 at (-5,0)             {$b_2$};
\node[xshift = .2cm, yshift= -.3cm]  at (-2.7735,4.16025)  {$b_3$};

\draw[color=gold,thick] plot [smooth] coordinates {(-1.5,2.25) (-.75,2.5) (0,2.25)};
\draw[color=gold,thick] plot [smooth] coordinates {(0,2.25) (.75,2.5) (1.5,2.25)};
\draw[color=gold,thick] plot [smooth] coordinates {(-1.5,2.25) (-.75,2) (0,2.25)};
\draw[color=gold,thick] plot [smooth] coordinates {(0,2.25) (.75,2) (1.5,2.25)};

\draw[color=gold,thick] plot [smooth] coordinates {(2.5,0) (2.42,-0.762) (2,-1.25)};
\draw[color=gold,thick] plot [smooth] coordinates {(2,-1.26) (1.98,-1.79) (1.5,-2.25)};
\draw[color=gold,thick] plot [smooth] coordinates {(2.5,0) (2,-.6) (2,-1.25)};
\draw[color=gold,thick] plot [smooth] coordinates {(2,-1.26) (1.6,-1.56) (1.5,-2.25)};

\draw[color=gold,thick] plot [smooth] coordinates { (-1.5,-2.25) (-1.52,-1.6) (-2,-1.125) };
\draw[color=gold,thick] plot [smooth] coordinates { (-2,-1.125) (-1.999,-.5) (-2.5,0)};
\draw[color=gold,thick] plot [smooth] coordinates { (-1.5,-2.25) (-2,-1.7) (-2,-1.125) };
\draw[color=gold,thick] plot [smooth] coordinates { (-2,-1.125) (-2.5,-.69) (-2.5,0)};

\draw[fill=black] (1.5,2.25)   circle (0.15cm);
\draw[fill=white] (2.5,0)      circle (0.15cm);
\draw[fill=black] (1.5,-2.25)  circle (0.15cm);
\draw[fill=white] (-1.5,-2.25) circle (0.15cm);
\draw[fill=black] (-2.5,0)     circle (0.15cm);
\draw[fill=white] (-1.5,2.25)  circle (0.15cm);

\draw[fill=white] (2.7735,4.16025)   circle (0.15cm);
\draw[fill=black] (5,0)              circle (0.15cm);
\draw[fill=white] (2.7735,-4.16025)  circle (0.15cm);
\draw[fill=black] (-2.7735,-4.16025) circle (0.15cm);
\draw[fill=white] (-5,0)             circle (0.15cm);
\draw[fill=black] (-2.7735,4.16025)  circle (0.15cm);

\end{scope}
\end{tikzpicture}
    \caption{On the left we have an hourglass plabic graph, and on the right is the resulting hourglass plabic graph after the application of a benzene move. The boundary vertices and their colors are unchanged.}
    \label{fig:benzene-move}
\end{figure}

A description of the boundary word for the set of plane partitions in an $a\times b\times c$ box was given in \cite{gaetz2023rotation}, restated below in Theorem \ref{thm:general-lattice-words}. In this paper, we determine in Theorem \ref{thm: lattice-words-symmetry-classes} the analogous boundary words for certain symmetry classes of plane partitions.

\begin{thm}[\protect{\cite[Prop.~5.5]{gaetz2023rotation}}]
\label{thm:general-lattice-words}
Let $a\geq b.$ The $U_q(\sl_4)$-web equivalence class given by contracted, fully-reduced hourglass plabic graphs with boundary word
\[
\begin{array}{ll}
    \pp(a,b,c): & 1^a~\Bar{4}^c~2^b~\Bar{1}^{a-b}~\Bar{2}^b~4^c~\Bar{1}^b
\end{array}
\]
is in bijection with the set $\pp(a,b,c).$
\end{thm}

In Section \ref{subsec: proof of lattice words of symmetry classes}, we prove our first main theorem. We give a less precise statement of the theorem below:
\begin{thm}
\label{thm: lattice-words-symmetry-classes}
    The boundary words of $U_q(\sl_4)$-webs corresponding to plane partitions of select symmetry classes are as follows:
\[
\begin{array}{ll}
     \spp(a,a,c):  & 1^a~\Bar{4}^c~2^a~\{4^c,(34)^a\} \\
     \cspp(a,a,a): & 1^a~\Bar{4}^a~\{(\overline{31})^m,\Bar{1}^{a-m}\}~\{4^{a-m},(34)^m\} \\
     \tspp(a,a,a): & 1^a~\{(23)^{a-m},2^m\}~\{4^{a-m},(34)^m\} \\
     \tsscpp(a,a,a), a = 2d: & 1^d~\underbrace{2~\Bar{4}~2~\cdots\Bar{4}~2}_{2d-1}~\{4^{d-1},(34)^{d-1}\}~(34)
\end{array}
\]
where $\{\cdots\}$ is a multiset from which one builds a subword using every letter in the multiset with the multiplicity given by the exponent. Each of these multisets has particular restrictions on the possible subwords given by the letters in the multiset.  The list of the restrictions can be found in Section \ref{subsec: proof of lattice words of symmetry classes}.
\end{thm}

\begin{figure}[ht]
    \centering
    \includegraphics[width=0.45\linewidth]{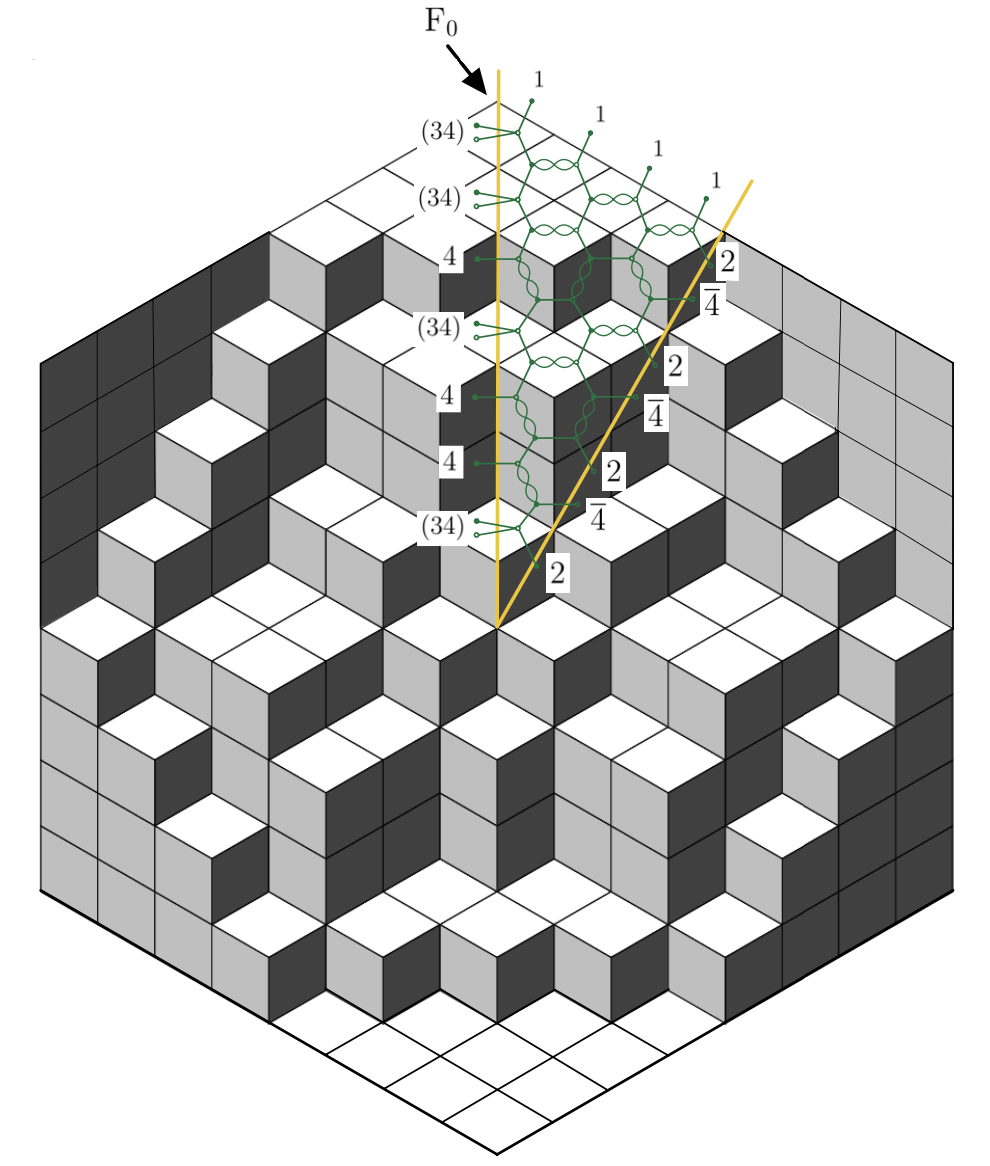}
    \caption{A totally symmetric self-complementary plane partition in an $8\times 8\times 8$ box. The $U_q(\sl_4)$-web in bijection with this plane partition is drawn on top of the plane partition, restricted to the fundamental domain of the symmetry class. We have labeled the top most face (\emph{base face}) as $F_0$. The boundary word corresponding to this web (and thus also this plane partition) is:
    \[1~1~1~1~2~\barfour~2~\barfour~2~\barfour~2~(34)~4~4~(34)~4~(34)~(34).\]}
    \label{fig:TSSCPP-web-boundary-word}
\end{figure}

An example of Theorem \ref{thm: lattice-words-symmetry-classes} can be seen in Figure \ref{fig:TSSCPP-web-boundary-word}.
We then use these lattice words to construct a combinatorial map in Algorithm \ref{alg:lattice-word-to-smaller web} that sends a $U_q(\sl_4)$-invariant to a $U_q(\sl_2)$ invariant or $U_q(\sl_3)$ invariant, depending on the symmetry class of the plane partition whose fundamental domain is in bijection with the $U_q(\sl_4)$-web.
The algorithm leads to our second main result, Theorem \ref{thm:big invariant to small invariant}, in which we define a map
\begin{equation*}
        \mathrm{Hom}_{U_q(\sl_4)}
        \left(
        \bigwedge\nolimits_q^{\underline{c}(\omega)}V_q,\C(q)
        \right)
        \to
        \mathrm{Hom}_{U_q(\slr)}\left(
        \bigwedge\nolimits_q^{\underline{c}(\hat{\omega})}V_q,\C(q)
        \right)
\end{equation*}
given by $[W]_q\mapsto [\widehat{W}]_q,$ where $W$ and $\widehat{W}$ are webs with lattice words $\omega$ and $\hatomega,$ respectively, via Algorithm \ref{alg:lattice-word-to-smaller web}.
This map is defined for a web corresponding to a symmetric, totally symmetric, or totally symmetric self complementary plane partition.
In the second case, $r=3,$ otherwise $r=2.$
More precisely, the map is defined on the span of the invariants $[W]_q$ of these webs, and it is injective.

We do not know whether this map has a representation-theoretic interpretation, for example as a restriction from $U_q(\sl_4)$ to $U_q(\sl_2)$ or $U_q(\sl_3)$, and we leave this question to future work.
In this way, we are able to use the new combinatorics of hourglass plabic graphs to give maps between spaces of invariants.
In particular, since $U_q(\sl_2)$-webs can be viewed as non-crossing perfect matchings, the algorithm says that the boundary words of the $U_q(\sl_4)$-webs coming from the sets $\spp$ and $\tsscpp$ are in bijection with sets of decorated non-crossing matchings.

In forthcoming work, we will show how a dynamical action on tableaux called \emph{promotion}  maps through the map from Theorem \ref{thm:big invariant to small invariant} onto the smaller invariant.

This paper is organized as follows:
We begin in Section \ref{sec:background} with some basic background on $U_q(\sl_r)$-webs and their interpretation as quantum invariants. We then review the description of $U_q(\sl_4)$-webs as hourglass plabic graphs from \cite{gaetz2023rotation} and the associated combinatorial tools, including oscillating tableaux and lattice words. We  also review the symmetry classes of plane partitions that are the focus of this paper.
In Section \ref{sec: Lattice words of symmetry classes}, we describe the correspondence between hourglass plabic graphs and plane partitions.
We also define vital notation for the proof of Theorem \ref{thm: lattice-words-symmetry-classes}.
In Section \ref{subsec: proof of lattice words of symmetry classes}, we give the proof of our first main result, Theorem \ref{thm: lattice-words-symmetry-classes}.
We also count the lattice words of each of the four symmetry classes discussed.
In Section \ref{sec:A combinatorial map between invariant spaces} we use the new lattice words given in our first main result to construct Algorithm \ref{alg:lattice-word-to-smaller web}.
This algorithm allows us to prove our second main result, Theorem \ref{thm:big invariant to small invariant}. In this section we also give two large examples of Algorithm \ref{alg:lattice-word-to-smaller web} and Theorem \ref{thm:big invariant to small invariant}. Our first example shows mapping a $U_q(\sl_4)$-web resulting from a $\tsscpp$ to a $U_q(\sl_2)$-web (a non-crossing matching), while our second example is a little more involved and shows a $U_q(\sl_4)$-web resulting from a $\tspp$ to a $U_q(\sl_3)$-web.

\section*{Acknowledgments}
  Striker was supported by  NSF grant DMS-2247089 and Simons Foundation gift MP-TSM-00002802.  Adams was supported by NSF grant DMS-2247089 and NSF Graduate Fellowship Grant No. 2034612. We would like to thank Jesse Kim, Oliver Pechenik, and Joshua Swanson for helpful conversations.

\section{Background}
\label{sec:background}
This section is organized as follows. In Subsection \ref{subsec:sl4-webs} we give the definition of $U_q(\slr)$-webs as used in this paper and relate it to representation theory. In Subsection \ref{subsec: computing an invariant from a web} we describe how a web in our convention gives an invariant. In Subsection \ref{subsec: web-bases}, we give a short background on web bases for invariant spaces.
In Subsection \ref{subsec:hourglass-plabic-graphs}, we move into the combinatorics of webs. We begin discussing the main results of \cite{gaetz2023rotation} by describing webs as hourglass plabic graphs and assigning labels to the edges, called separation labels.
We continue this discussion in Subsection \ref{subsec:oscillating-tableaux}, where we show how \cite{gaetz2023rotation} used the  separation labeling to relate webs to oscillating tableaux.
We conclude the background section with Subsection \ref{subsec: symmetry-classes-of-pp}, in which we review the basics of plane partitions and give the definitions of symmetry classes of plane partitions.

\subsection{\texorpdfstring{$U_q(\slr)$}{Uq(slr)}-webs and their representation theory}
\label{subsec:sl4-webs}
Understanding the space of invariants of $\mathrm{SL}_r(\C)$, and its natural quantum deformation $U_q(\slr),$ has been of interest in numerous fields. A way to classify these spaces categorically via generators and relations has been developed over the past few decades in the form of certain diagrams called $\mathrm{SL}_r$-webs ($U_q(\slr)$-webs, resp.), or just ``webs" when the group and $r$ are understood.
The graphical calculus formed by webs gives us a tangible, combinatorial way to compute and classify tensor invariants.
Since the work in this paper also holds in the quantum group setting, we will define everything in terms of $U_q(\slr).$ For those uninterested in quantum deformations, one can naively disregard the $q$ term by setting $q = 1$ and replacing $U_q(\slr)$ with $\mathrm{SL_r}(\C).$

Consider the group $G = \mathrm{SL_r}(\C).$  Let $V=\C^r$ be the defining representation for $G$ and take $V({\alpha_{c}}) = \bigwedge^{c}V,$ for weight $\alpha_c.$ We denote the dual by $\bigwedge^{-c}V =(\bigwedge^{c}V)^*.$ For $\underline{c} = (c_1,\ldots,c_n),$ $c_i\in\Z,$ the subspace of $\mathrm{SL}_r(\C)$-invariant multilinear forms is
\begin{align*}
    \mathsf{Inv}_{G}
    \left(
    \bigwedge\nolimits^{\underline{c}}V
    \right)
    & :=
    \Hom_{G}
    \left(
    \bigwedge\nolimits^{\underline{c}}V,\C
    \right) \\
    & :=
    \Hom_{G}
    \left(
    \bigwedge\nolimits^{c_1}V\otimes\cdots\otimes\bigwedge\nolimits^{c_n}V,\C
    \right).
\end{align*}
A web $W$ gives an element of this space $[W],$ motivating the study of webs as a tool to characterize the space of (quantum) invariants. If, as in this paper, one would rather consider the quantum group, $U_q(\slr),$ replacing $G$ with $U_q(\slr)$ deforms $\mathsf{Inv}_{\mathsf{SL}_r}(\bigwedge\nolimits^{\underline{c}}V),$ thereby giving those quantum-deformed linear functionals on $ \bigwedge\nolimits^{\underline{c}}_qV_q$ that are invariant under $U_q(\slr).$

We warn the reader that there are different conventions for constructing webs. Each convention has valuable uses. The work here builds on \cite{gaetz2023rotation}, which uses the convention developed by \cite{fraser2019dimers}.
The convention developed by \cite{morrison2007diagrammatic,cautis2014webs} consists of webs that look more like oriented ``ladders." It is within this convention that one can more easily compute the actual morphisms in the free monoidal category of webs, which are suppressed within the convention used here.
In order to convert from the convention used in this paper to that of \cite{cautis2014webs}, we refer the reader to \cite[Section~2.1]{gaetz2023rotation} for a tactile conversion and to \cite{poudel2022comparison} for a more theoretical comparison.

\begin{defn}[\protect{\cite[Definition~3.2]{fraser2019dimers}}]
\label{def:tensor-diagram}
A $U_q(\sl_r)$-\emph{web} $W$ is a finite planar graph on a disk such that
\begin{itemize}
    \item every vertex of $W$ is colored black or white, and every edge of $W$ connects a black vertex to a white vertex,
    \item every boundary vertex lies on the boundary of the disk and is single valent,
    \item every edge has a weight in $\{1,\ldots,r-1\},$
    \item the weights of the edges incident to an internal vertex sum to $r.$
\end{itemize}
The graph $W$ need not be connected. The \emph{type} of $W$ is $\underline{c}=(c_1,\ldots,c_n),$ where $|c_i|$ is the weight of the edge incident to the boundary vertex $b_i,$ and $c_i$ is positive if $b_i$ is black and negative if $b_i$ is white.
\end{defn}

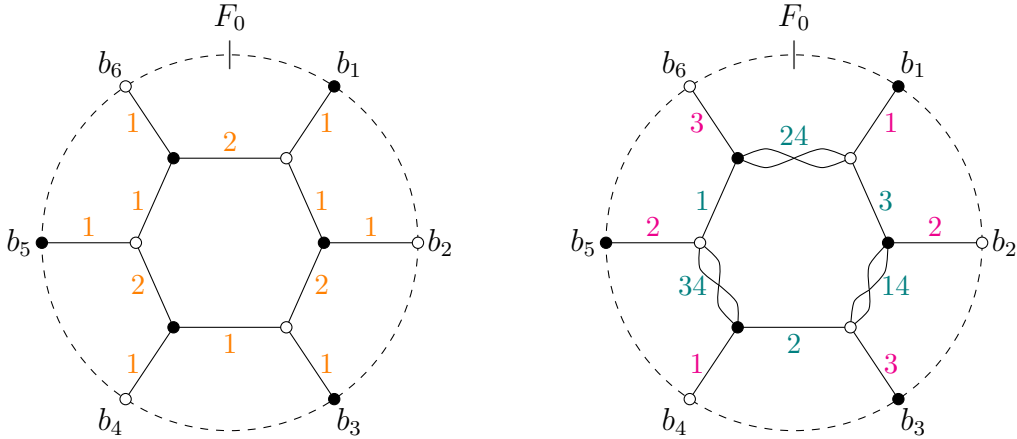
\begin{figure}[ht]
\centering
\begin{tikzpicture}[scale = 0.5]
\begin{scope}[xshift = -7.5cm]
\draw[dashed] (0,0) circle (5cm);

\node at (0,6) {$F_0$};
\node at (0,5) {\small $|$};

\draw (1.5,2.25) -- (2.5,0)
                 -- (1.5,-2.25)
                 -- (-1.5,-2.25)
                 -- (-2.5,0)
                 -- (-1.5,2.25)
                 -- cycle;
\draw (2.7735,4.16025)    -- (1.5,2.25);
\draw (-2.7735, -4.16025) -- (-1.5,-2.25);
\draw (-2.7735, 4.16025)  -- (-1.5,2.25);
\draw (2.7735, -4.16025)  -- (1.5,-2.25);
\draw (2.5,0)             -- (5,0);
\draw (-2.5,0)            -- (-5,0);

\draw[fill=white] (1.5,2.25)   circle (0.15cm);
\draw[fill=black] (2.5,0)      circle (0.15cm);
\draw[fill=white] (1.5,-2.25)  circle (0.15cm);
\draw[fill=black] (-1.5,-2.25) circle (0.15cm);
\draw[fill=white] (-2.5,0)     circle (0.15cm);
\draw[fill=black] (-1.5,2.25)  circle (0.15cm);

\draw[fill=black] (2.7735,4.16025)   circle (0.15cm);
\draw[fill=white] (5,0)              circle (0.15cm);
\draw[fill=black] (2.7735,-4.16025)  circle (0.15cm);
\draw[fill=white] (-2.7735,-4.16025) circle (0.15cm);
\draw[fill=black] (-5,0)             circle (0.15cm);
\draw[fill=white] (-2.7735,4.16025)  circle (0.15cm);

\node[xshift = .2cm, yshift= .3cm]   at (2.7735,4.16025)   {$b_1$};
\node[xshift = .3cm]                 at (5,0)              {$b_2$};
\node[xshift = .2cm, yshift= -.3cm]  at (2.7735,-4.16025)  {$b_3$};
\node[xshift = -.2cm, yshift= -.3cm] at (-2.7735,-4.16025) {$b_4$};
\node[xshift = -.3cm]                at (-5,0)             {$b_5$};
\node[xshift = -.2cm, yshift= .3cm]  at (-2.7735,4.16025)  {$b_6$};

\node[xshift = 0.22cm] at (2.13675,3.205125)    {\small \textcolor{orange}{$1$}};
\node[yshift = 0.22cm] at (3.75,0)              {\small \textcolor{orange}{$1$}};
\node[xshift = 0.22cm] at (2.13675,-3.205125)   {\small \textcolor{orange}{$1$}};
\node[yshift = 0.22cm] at (-3.75,0)             {\small \textcolor{orange}{$1$}};
\node[xshift = -0.22cm] at (-2.13675,-3.205125) {\small \textcolor{orange}{$1$}};
\node[xshift = -0.22cm] at (-2.13675,3.205125)  {\small \textcolor{orange}{$1$}};

\node[xshift = 0.22cm]  at (2,1.125)   {\small \textcolor{orange}{$1$}};
\node[xshift = 0.22cm]  at (2,-1.125)  {\small \textcolor{orange}{$2$}};
\node[yshift = -0.22cm] at (0,-2.25)   {\small \textcolor{orange}{$1$}};
\node[xshift = -0.22cm] at (-2,-1.125) {\small \textcolor{orange}{$2$}};
\node[xshift = -0.22cm] at (-2,1.125)  {\small \textcolor{orange}{$1$}};
\node[yshift = 0.22cm]  at (0,2.25)    {\small \textcolor{orange}{$2$}};

\end{scope}
\begin{scope}[xshift = 7.5cm]
\draw[dashed] (0,0) circle (5cm);

\node at (0,6) {$F_0$};
\node at (0,5) {\small $|$};

\draw (1.5,2.25)  -- (2.5,0); 
\draw (1.5,-2.25) -- (-1.5,-2.25); 
\draw (-2.5,0)    -- (-1.5,2.25); 

\draw (2.7735,4.16025)    -- (1.5,2.25);
\draw (-2.7735, -4.16025) -- (-1.5,-2.25);
\draw (-2.7735, 4.16025)  -- (-1.5,2.25);
\draw (2.7735, -4.16025)  -- (1.5,-2.25);
\draw (2.5,0)             -- (5,0);
\draw (-2.5,0)            -- (-5,0);

\node[xshift = .2cm, yshift= .3cm]  at (2.7735,4.16025)    {$b_1$};
\node[xshift = .3cm]                 at (5,0)              {$b_2$};
\node[xshift = .2cm, yshift= -.3cm]  at (2.7735,-4.16025)  {$b_3$};
\node[xshift = -.2cm, yshift= -.3cm] at (-2.7735,-4.16025) {$b_4$};
\node[xshift = -.3cm]                at (-5,0)             {$b_5$};
\node[xshift = -.2cm, yshift= .3cm]  at (-2.7735,4.16025)  {$b_6$};

\node[xshift = 0.22cm]  at (2.13675,3.205125)    {\small \textcolor{magenta}{$1$}};
\node[yshift = 0.22cm]  at (3.75,0)              {\small \textcolor{magenta}{$2$}};
\node[xshift = 0.22cm]  at (2.13675,-3.205125)   {\small \textcolor{magenta}{$3$}};
\node[xshift = -0.22cm] at (-2.13675,-3.205125)  {\small \textcolor{magenta}{$1$}};
\node[yshift = 0.22cm]  at (-3.75,0)             {\small \textcolor{magenta}{$2$}};
\node[xshift = -0.22cm] at (-2.13675,3.205125)   {\small \textcolor{magenta}{$3$}};

\node[xshift = 0.22cm]  at (2,1.125)   {\small \textcolor{teal}{$3$}};
\node[yshift = -0.22cm] at (0,-2.25)   {\small \textcolor{teal}{$2$}};
\node[xshift = -0.22cm] at (-2,1.125)  {\small \textcolor{teal}{$1$}};

\node[yshift = 0.3cm] at (0,2.25)     {\small \textcolor{teal}{$24$}};
\draw plot [smooth] coordinates {(-1.5,2.25) (-.75,2.5) (0,2.25)};
\draw plot [smooth] coordinates {(0,2.25) (.75,2.5) (1.5,2.25)};
\draw plot [smooth] coordinates {(-1.5,2.25) (-.75,2) (0,2.25)};
\draw plot [smooth] coordinates {(0,2.25) (.75,2) (1.5,2.25)};

\node[xshift = 0.35cm,yshift=-.2] at (2,-1.125)   {\small \textcolor{teal}{$14$}};
\draw plot [smooth] coordinates {(2.5,0) (2.42,-0.762) (2,-1.25)};
\draw plot [smooth] coordinates {(2,-1.26) (1.98,-1.79) (1.5,-2.25)};
\draw plot [smooth] coordinates {(2.5,0) (2,-.6) (2,-1.25)};
\draw plot [smooth] coordinates {(2,-1.26) (1.6,-1.56) (1.5,-2.25)};

\node[xshift = -0.35cm,yshift=-.2] at (-2,-1.125) {\small \textcolor{teal}{$34$}};
\draw plot [smooth] coordinates { (-1.5,-2.25) (-1.52,-1.6) (-2,-1.125) };
\draw plot [smooth] coordinates { (-2,-1.125) (-1.999,-.5) (-2.5,0)};
\draw plot [smooth] coordinates { (-1.5,-2.25) (-2,-1.7) (-2,-1.125) };
\draw plot [smooth] coordinates { (-2,-1.125) (-2.5,-.69) (-2.5,0)};

\draw[fill=white] (1.5,2.25)   circle (0.15cm);
\draw[fill=black] (2.5,0)      circle (0.15cm);
\draw[fill=white] (1.5,-2.25)  circle (0.15cm);
\draw[fill=black] (-1.5,-2.25) circle (0.15cm);
\draw[fill=white] (-2.5,0)     circle (0.15cm);
\draw[fill=black] (-1.5,2.25)  circle (0.15cm);

\draw[fill=black] (2.7735,4.16025)   circle (0.15cm);
\draw[fill=white] (5,0)              circle (0.15cm);
\draw[fill=black] (2.7735,-4.16025)  circle (0.15cm);
\draw[fill=white] (-2.7735,-4.16025) circle (0.15cm);
\draw[fill=black] (-5,0)             circle (0.15cm);
\draw[fill=white] (-2.7735,4.16025)  circle (0.15cm);

\end{scope}
\end{tikzpicture}
    \caption{On the left is a $U_q(\sl_4)$-web of type
    $
    \underline{c} = (
    \textcolor{orange}{1},
    \textcolor{orange}{-1},
    \textcolor{orange}{1},
    \textcolor{orange}{-1},
    \textcolor{orange}{1},
    \textcolor{orange}{-1}
    ).
    $
    with labeled edge weights. On the right is the same web as an hourglass plabic graph and given a proper coloring.
   }
    \label{fig:web-to-hourglass-web}
\end{figure}

\begin{example}
     In Figure \ref{fig:web-to-hourglass-web}, on the left we have a $U_q(\sl_4)$-web of type
    $
    \underline{c} = (
    \textcolor{orange}{1},
    \textcolor{orange}{-1},
    \textcolor{orange}{1},
    \textcolor{orange}{-1},
    \textcolor{orange}{1},
    \textcolor{orange}{-1}
    )
    $
    with boundary vertices $b_1,\ldots,b_6.$
    The internal edges on the hexagon alternate between weight $1$ and weight $2.$ On the right, we translate the web into an hourglass plabic graph with base face $F_0$ á la \cite{gaetz2023rotation}. We have also given the hourglass plabic graph a proper edge coloring with colors $\{1,2,3,4\}.$
\end{example}

A web corresponding to an invariant in
$
\Hom_{U_q(\slr)}
\left(
\bigwedge\nolimits_q^{\underline{c}}V_q,\C(q)
\right)
$
has boundary with vertices $(b_1,\ldots,b_n)$ and type $\underline{c} = (c_1,\ldots,c_n)$ where each vertex $b_i$ is incident to a single edge $e_i$ with edge weight $|c_i|,$ and the sign of $c_i$ records the color of $b_i.$ (From this point forward, we use weight to mean an edge weight in a graph.)

 \subsection{Computing an invariant from a web}
 \label{subsec: computing an invariant from a web}
 A web $W$ allows us to calculate a specific invariant $[W]_q$ in the polynomial ring $\C(q)[x_{i,S}],$ $S\subs[r].$ Consider a $U_q(\slr)$-web $W$ with boundary vertices $b_1,\ldots,b_n$. Define on the edges $E(W)$ a proper coloring
 \[
 \kappa: E(W)\to 2^{[r]},
 \]
 $e\mapsto \kappa(e)$ such that $|\kappa(e)|$ is the weight of $e$ (i.e., if $e$ has weight $\alpha,$ the number of colors attached to $e$ is $\alpha$).
 A coloring $\kappa$ is proper if for any internal vertex $v\in V(W)$ and distinct incident edges $e_i(v),e_j(v),$ the sets $\kappa(e_i(v))$ and $\kappa(e_j(v))$ are disjoint, i.e., the labels on the edges incident to an internal vertex $v$ partition the set $[r].$

For each coloring we obtain a monomial in $[W]_q,$ given by
\[
            \sgn_\kappa(W)
            q^{\mathsf{wgt}_\kappa(W)}
            \prod_{b_i\text{ black}}
            x_{
            i,\kappa(e(b_i))
            }
            \prod_{b_i \text{ white}}
            y_{
            \kappa(e(b_i)),i
            },
\]
where $\sgn_\kappa(W)$ denotes the sign of $W$ and $\mathsf{wgt}_\kappa(W)$ denotes the $q$-weight. Computing this $q$-statistic is a bit involved, as it is dependent on the particular coloring of $W.$
Currently, one must move to the convention of \cite{cautis2014webs} and use the natural $U_q(\slr)$-module maps that arise from coproducts, duals, evaluation, coevaluation, and the identity as given in \cite{cautis2014webs} in order to compute $\mathsf{wgt}_\kappa(W)$.
We refer the reader to \cite{gaetz2023rotation} for a discussion on the convention for computing $\sgn_\kappa(W)$.

Thus,
\begin{equation}
            [W]_q =
            \sum_{\text{colorings $\kappa$}}
            \left(
            \sgn_\kappa(W)
            q^{\mathsf{wgt}_\kappa(W)}
            \prod_{b_i\text{ black}}
            x_{
            i,\kappa(e(b_i))
            }
            \prod_{b_i \text{ white}}
            y_{
            \kappa(e(b_i)),i
            }
            \right).
\end{equation}

In this paper, we will focus on the case that $r=4.$

\subsection{Web bases}
\label{subsec: web-bases}
Although there are many well-known bases for the space
$
\Hom_{\mathrm{G}}
\left(
\bigwedge\nolimits^{\underline{c}}V,\C
\right),
$
it is natural to ask for a basis given by webs since each element can be described via webs. For $\mathrm{SL}_2,$ Kuperberg described a web-basis via non-crossing perfect matchings of points on a disk, which comes from the Temperley--Lieb algebra. (More on this in Remark \ref{rmk: a2-spiders}.) For $\mathrm{SL}_3$, Kuperberg developed a basis of non-elliptic $\mathrm{SL}_3$-webs \cite{kuperberg}. Both of these bases have the property of being \emph{rotation-invariant}, which means that the basis is sent to the corresponding basis under the cyclic shift
    \[
    \Hom_{U_q(\sl_r)}
    \left(
    \bigwedge\nolimits_q^{(c_1,c_2,\ldots,c_n)}V_q,\C(q)
    \right)
    \to
    \Hom_{U_q(\sl_r)}
    \left(
    \bigwedge\nolimits_q^{(c_2,\ldots,c_n,c_1)}V_q,\C(q)
    \right).
    \]

For $\mathrm{SL}_4,$ Kim gave a conjectural basis \cite{kim2003graphical}, which was extended to general $r$ by Morrison in his thesis \cite{morrison2007diagrammatic}. Cautis, Kamnitzer, and Morrison \cite{cautis2014webs} resolved the conjecture, giving a complete description of the space of invariants via webs with a full list of generators and relations. They further extend the results to the quantum group, $U_q(\slr).$
Their basis, which has nice properties of its own, is not rotation-invariant.
Fraser, Lam, and Le give an alternate basis \cite{fraser2019dimers}. There are other web-bases for simple Lie groups of higher rank, but most of them require arbitrary choices and none of them are rotation-invariant. For example, for $\mathrm{SL}_r,$ Westbury \cite{westbury2012web}  gave a basis of leading terms when representations are restricted to $c_i\in\{1,r-1\},$ which was extended by Fontaine, Kamnitzer, and Kuperberg \cite{fontaine2013buildings}.
 Elias \cite{elias2015light} gave a basis for certain quantum deformations of these groups, while  Hagemeyer \cite{hagemeyer2018spiders} gave a basis in his thesis for the quantum group, $U_q(\sl_4).$
A $U_q(\sl_4)$ basis exhibiting the rotation-invariance property was recently given by Gaetz, Pechenik, Pfannerer, Striker, and Swanson \cite{gaetz2023rotation}.

\begin{thm}[\protect{\cite[Theorem A]{gaetz2023rotation}}]
\label{thm:web-basis-GPPSS}
    The collection
        \[
        \{[W]_q~:~W \text{ is a top fully reduced hourglass plabic graph of type }\underline{c} \}
        \]
        is a rotation-invariant basis for
        $
        \Hom_{U_q(\sl_4)}
        \left(
        \bigwedge\nolimits_q^{\underline{c}}V_q,\C(q)
        \right)
        $.
\end{thm}
Here \emph{contracted}, \emph{fully reduced}, and \emph{top} are as in \cite[Definitions~3.6, 3.9, and~6.7]{gaetz2023rotation}. The next section describes the web convention used in  the proof of this theorem. In the case of $r\geq 5,$ a rotation-invariant basis remains elusive.

\subsection{Hourglass plabic graphs}
\label{subsec:hourglass-plabic-graphs}
Hopkins and Rubey observed that $U_q(\sl_3)$-webs can be viewed as plabic graphs \cite{hopkins2022promotion}. This purely combinatorial method of analyzing webs motivates most of the work in this paper. Just as in the case of webs, plabic graphs can similarly be given edge weights. The motivation behind the object of interest in this section is to interpret these edge weights as special multi-edges called hourglass edges.

\begin{defn}[\protect{\cite[Definition~11.5]{postnikov2006total}}]
    A \emph{plabic graph} is a bicolored, planar, undirected graph embedded in a disk. The boundary vertices, labeled clockwise as $b_1,b_2,\ldots,$ are univalent, and their edges have edge weight one.
\end{defn}

The definition of plabic graphs was extended by \cite{gaetz2023rotation} in their diagrammatic interpretation of $U_q(\sl_4)$-webs. Their definition includes that all internal vertices are $4$-valent and that if there is an edge $e$ with weight $2$ and vertices $v_1,v_2$ in the Fraser-Lam-Le convention, $e$ is replaced by two edges $e_1,e_2$ both with vertices $v_1,v_2.$ Furthermore, this edge is twisted so that $e_1,e_2$ cross exactly once to form an edge called an \emph{hourglass edge.} We call edges with edge weight $1,$ \emph{simple edges}. We note that the authors of \cite{gaetz2023rotation} give a more general definition of hourglass plabic graphs, but this more general notion is not necessary for our purposes.

\begin{defn}[\protect{\cite[Definition~3.1]{gaetz2023rotation}}]
\label{def:hourglass-plabic-graph}
    A properly bicolored plabic graph with edge weights in $\{1,2\}$ is called an \emph{hourglass plabic graph} if every internal vertex has degree $4$ (counting each hourglass edge twice), every edge with weight $2$ is replaced with an hourglass edge and simple edges remain unchanged.
\end{defn}

We consider such graphs up to planar isotopy. A way to understand hourglass plabic graphs is through a permutation called a trip. Given a plabic graph $W$ with $n$ boundary vertices, a trip permutation $\sigma\in S_n$ is determined by walking along edges of the graph via the rules of the road (Definition \ref{def:trip permutations on hgpgs}) and recording the walk's departure vertex $b_i$ and arrival vertex $b_j$ by setting $\sigma(i)=j.$ These were originally introduced by Postnikov \cite{postnikov2006total} and were generalized to higher trips by \cite{gaetz2023rotation} as in the following definition.
\begin{defn}[\protect{\cite[Definition~2.12]{gaetz2023rotation}}]
\label{def:trip permutations on hgpgs}
Given an hourglass plabic graph with boundary vertices $b_1,\ldots,b_n$ and $1\leq a\leq 3$,
we define a path called the \emph{$a$-th trip}, denoted \emph{$\trip_a(b_i)$}, that begins on the unique edge incident to a boundary vertex $b_i$  and follows the rules of the road:
    \begin{itemize}
        \item $\trip_1$ takes a right at black vertices and a left at white vertices
        \item $\trip_2$ takes the second right at white and black vertices.
        \item $\trip_3$ takes a left at black vertices and a right at white vertices.
    \end{itemize}
    The path $\trip_a(b_i)$  ends at another unique edge incident to a boundary vertex $b_{j}$. The permutation on $\{1,\ldots,n\}$ given by the function $\trip_a(G)(i)=j$ whenever $\trip_a(b_i)$ ends at $b_j$ is called the \emph{$a$-th trip permutation}.
\end{defn}

Note that by construction $\trip_1$ and $\trip_3$ are inverses of each other and $\trip_2$ is an involution. Also, every simple edge $s$ (not just those edges incident to a boundary vertex), has exactly three trips passing through it from black to white. These are exactly $\trip_1$, $\trip_2$, and $\trip_3$ for that edge. Therefore, it makes sense to consider $\tripstrand_i(s)$ for an arbitrary simple edge $s.$ The trips allow us to compute and associate to each edge a value in $\{1,2,3,4\}$, as in the definition below. Note that
although this does result in a proper coloring, and thus gives a monomial in the corresponding invariant, this coloring should be thought about as something altogether different.

\begin{defn}[\protect{\cite[Definition~4.8]{gaetz2023rotation}}]
\label{def: separation-labelings}
    Given an hourglass plabic graph with boundary vertices $b_1,\ldots,b_n,$ call the face between the edge incident to $b_n$ and the edge incident to $b_1$ the \emph{base face} and label it $F_0.$ Let $\sep(e)$ be the \emph{separation labeling} of the edge $e,$ defined as follows:
    \begin{itemize}
        \item Let $s$ be a simple edge and let $F(s)$ be the face to the right of $s$ when traveling from the black to white vertex. Then
        \[
        \sep(s) = \#\{i~|~\tripstrand_i(s) \text{ separates } F(s) \text { from } F_0\}+1.
        \]
        \item Let $h$ be an hourglass edge containing a vertex $v$, then
        \[
        \sep(h) = \{1,2,3,4\}\setminus\bigcup_{\substack{e\ni v\\ h\neq e}}\sep(e).
        \]
    \end{itemize}
\end{defn}

\begin{example}
\label{ex: separation labels and lattice word}
    The separation labels (Definition \ref{def: separation-labelings}) for the web in Figure \ref{fig:web-to-hourglass-web} are given in Figure \ref{fig:sep-labels-lattice-word-tableau}. They are different from the proper coloring shown on the right of Figure \ref{fig:web-to-hourglass-web}.
\end{example}

\cite[Proposition~4.9]{gaetz2023rotation} shows that the separation labeling is well defined and proper for contracted, fully-reduced hourglass plabic graphs. In this paper, the webs considered are naturally contracted and fully-reduced, so we can use the separation labeling without referring to these properties. The separation labeling was instrumental in proving a bijection from equivalence classes of hourglass plabic graphs to fluctuating tableaux. This bijection was pivotal in the proof of Theorem \ref{thm:general-lattice-words}. We discuss a special type of fluctuating tableaux, called oscillating tableaux, in the next section. More on fluctuating tableaux can be found in \cite{gaetz2023promotion}.

\subsection{Oscillating tableaux and their lattice words}
\label{subsec:oscillating-tableaux}
Fluctuating tableaux are a type of tableau studied in \cite{gaetz2023promotion}, which are defined via a sequence of positive and negative values. See Figure \ref{fig:oscillating-tab} for an example. These are of unique importance to the topic of $U_q(\sl_4)$-webs, since they played an important role in the proof of Theorem \ref{thm:web-basis-GPPSS}.
Every fluctuating tableau is defined uniquely via a Yamanouchi lattice word.
In this paper, we use only a certain type of fluctuating tableau called an oscillating tableau. For more on fluctuating tableaux see \cite[Definition 2.1]{gaetz2023promotion}. In this section we give formal definitions of oscillating tableaux, of lattice words, and we also describe explicitly their connection to webs.

\begin{defn}[\protect{\cite[Definition~2.1]{gaetz2023promotion}}]
\label{def:oscillating-tab}
    An $r$-row \emph{oscillating tableau}
    \[
    T = \lambda^0 = \varnothing \xrightarrow{c_1}\lambda^1\xrightarrow{c_2}\cdots\xrightarrow{c_n}\lambda^n
    \]
    of type $\boldc=(c_1,\ldots,c_n)\in\{\pm 1\}^n$ is a sequence of $r$-row generalized partitions, that is, weakly decreasing vectors $\lambda=(\lambda_1\geq\cdots\geq\lambda_r)\in\Z^r,$ such that $\lambda^i-\lambda^{i-1}=c_i\bolde_j$ for some row $1\leq j\leq r.$ When $c_i = 1,$ we add a box to row $j,$ whereas when $c_i = -1,$ we remove a box from row $j.$ In the pictorial representation of an oscillating tableau we represent $\lambda^0$ as single vertical line corresponding to the empty tableau. We often write $-i$ as $\Bar{i}.$
\end{defn}

\begin{figure}[htbp]
\begin{tikzpicture}[scale=0.25, every node/.style={inner sep=0pt, outer sep=0pt}]
    \def\n{7} 
    \def\spacing{2.0cm} 
    \def\center{-0.5*(\n-1)*\spacing} 

    \begin{scope}[scale=3.5, transform shape, xshift=\center + 0*\spacing]
        \node (n) {
            \begin{ytableau}
                \none & \none \\
                \none & \none \\
                \none & \none \\
                \none & \none \\
            \end{ytableau}
        };
        \draw[line width=.7mm,black] ([xshift=.65cm]n.south west)--([xshift=.65cm]n.north west);
    \end{scope}

    \begin{scope}[scale=3.5, transform shape, xshift=\center + 1*\spacing]
        \node (n) {
            \begin{ytableau}
                \none & \scriptstyle 1 \\
                \none & \none \\
                \none & \none \\
                \none & \none \\
            \end{ytableau}
        };
        \draw[line width=.7mm,black] ([xshift=.65cm]n.south west)--([xshift=.65cm]n.north west);
    \end{scope}

    \begin{scope}[scale=3.5, transform shape, xshift=\center + 2*\spacing]
        \node (n) {
            \begin{ytableau}
                \none & \scriptstyle 1 \\
                \none & \none \\
                \none & \none \\
                \scriptstyle \Bar{2} & \none \\
            \end{ytableau}
        };
        \draw[line width=.7mm,black] ([xshift=.65cm]n.south west)--([xshift=.65cm]n.north west);
    \end{scope}

    \begin{scope}[scale=3.5, transform shape, xshift=\center + 3*\spacing]
        \node (n) {
            \begin{ytableau}
                \none & \scriptstyle 1 \\
                \none & \scriptstyle 3 \\
                \none & \none \\
                \scriptstyle \Bar{2} & \none \\
            \end{ytableau}
        };
        \draw[line width=.7mm,black] ([xshift=.65cm]n.south west)--([xshift=.65cm]n.north west);
    \end{scope}

    \begin{scope}[scale=3.5, transform shape, xshift=\center + 4*\spacing]
        \node (n) {
            \begin{ytableau}
                \none & \scriptstyle 1 \\
                \none & \scriptstyle 3\,\Bar{4} \\
                \none & \none \\
                \scriptstyle \Bar{2} & \none \\
            \end{ytableau}
        };
        \draw[line width=.7mm,black] ([xshift=.65cm]n.south west)--([xshift=.65cm]n.north west);
    \end{scope}

    \begin{scope}[scale=3.5, transform shape, xshift=\center + 5*\spacing]
        \node (n) {
            \begin{ytableau}
                \none & \scriptstyle 1 \\
                \none & \scriptstyle 3\,\Bar{4} \\
                \none & \none \\
                \scriptstyle \Bar{2}\,5 & \none \\
            \end{ytableau}
        };
        \draw[line width=.7mm,black] ([xshift=.65cm]n.south west)--([xshift=.65cm]n.north west);
    \end{scope}

    \begin{scope}[scale=3.5, transform shape, xshift=\center + 6*\spacing]
        \node (n) {
            \begin{ytableau}
                \none & \scriptstyle 1\,\Bar{6} \\
                \none & \scriptstyle 3\,\Bar{4} \\
                \none & \none \\
                \scriptstyle \Bar{2}\,5 & \none \\
            \end{ytableau}
        };
        \draw[line width=.7mm,black] ([xshift=.65cm]n.south west)--([xshift=.65cm]n.north west);
    \end{scope}
\end{tikzpicture}

\vspace{0.65cm}

\begin{tikzpicture}
\begin{scope}
    \node at (0,0) {
    $T =
    \protect\underbrace{0000}_{\lambda^0}
    \xrightarrow{e_1}
    \protect\underbrace{1000}_{\lambda^1}
    \xrightarrow{-e_4}
    \protect\underbrace{100\barone}_{\lambda^2}
    \xrightarrow{e_2}
    \protect\underbrace{110\barone}_{\lambda^3}
    \xrightarrow{-e_2}
    \protect\underbrace{100\barone}_{\lambda^4}
    \xrightarrow{e_4}
    \protect\underbrace{1000}_{\lambda^5}
    \xrightarrow{-e_1}
    \protect\underbrace{0000}_{\lambda^6}$
    };
\end{scope}
\end{tikzpicture}

\caption{
Here we give an example of Definition \ref{def:oscillating-tab} by building an oscillating tableau on the word $\omega = 1~\Bar{4}~2~\Bar{2}~4~\Bar{1}.$ The pictorial visualization on top represents the oscillating tableau $T$ given underneath.
}
\label{fig:oscillating-tab}
\end{figure}
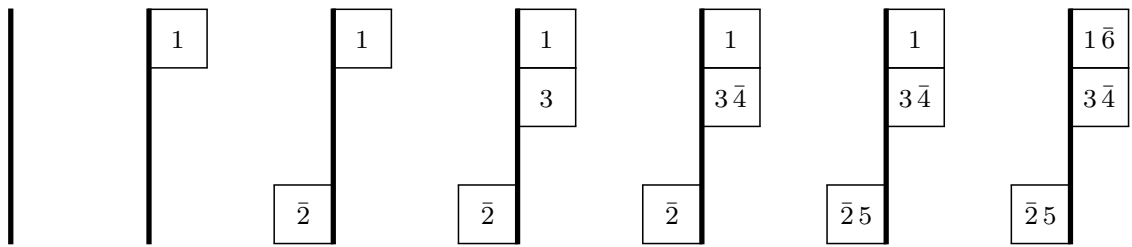

\begin{defn}
\label{def:lattice-word}
    A \emph{lattice word} $\omega = \omega(T) = \omega_1,\ldots,\omega_n\in\{\pm 1,\ldots, \pm r\}$ associated to an $r$-row oscillating tableau $T$ of type $(c_1,\ldots,c_n)$ has $\lambda^i = \lambda^{i-1}+\bolde_{\omega_i},$ where we use the convention $\bolde_{\bar{j}}=-\bolde_j.$ (See Figure \ref{fig:oscillating-tab}.)
\end{defn}

All such lattice words carry a special property called Yamanouchi.
\begin{defn}
    A word $w$ on an alphabet $\{\pm 1,\ldots, \pm r\}$ is called \emph{Yamanouchi} if in any initial subword $w_1\cdots w_k,$ the number of $i$s minus the number of $\overline{i}$s is greater than or equal to the number of $(i+1)$s minus the number of $(\overline{i+1})$s for all $1\leq i<r.$ We say that $w$ is \emph{rectangular} if, in the whole word $w,$ the number of $i$s minus the number of $\overline{i}$s is the same for all $1\leq i\leq r,$ that is, if all rows of the final partition $\lambda^n$ are equal. The boundary words of the webs in this paper are rectangular Yamanouchi words.
\end{defn}

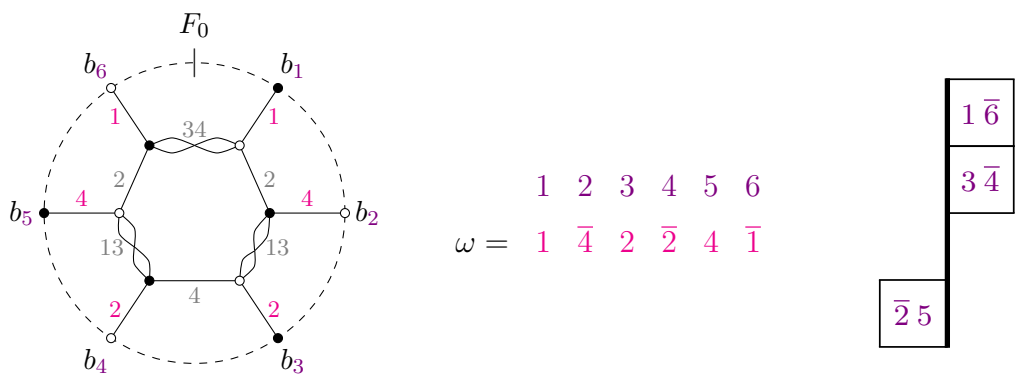
\begin{figure}[ht]
\centering
\begin{tikzpicture}[scale = 0.4]

\draw[dashed] (0,0) circle (5cm);

\node[yshift=.1cm] at (0,6) {\small $F_0$};
\node at (0,5) {\small $|$};

\draw (1.5,2.25) -- (2.5,0); 
\draw (1.5,-2.25) -- (-1.5,-2.25); 
\draw (-2.5,0) -- (-1.5,2.25); 

\draw (2.7735,4.16025) -- (1.5,2.25);
\draw (-2.7735, -4.16025) -- (-1.5,-2.25);
\draw (-2.7735, 4.16025) -- (-1.5,2.25);
\draw (2.7735, -4.16025) -- (1.5,-2.25);
\draw (2.5,0) -- (5,0);
\draw (-2.5,0) -- (-5,0);

\node[xshift = .2cm, yshift= .3cm]   at (2.7735,4.16025)
{\small $b_{\textcolor{violet}{1}}$};
\node[xshift = .3cm]                 at (5,0)
{\small $b_{\textcolor{violet}{2}}$};
\node[xshift = .2cm, yshift= -.3cm]  at (2.7735,-4.16025)
{\small $b_{\textcolor{violet}{3}}$};
\node[xshift = -.2cm, yshift= -.3cm] at (-2.7735,-4.16025)
{\small $b_{\textcolor{violet}{4}}$};
\node[xshift = -.3cm]                at (-5,0)
{\small $b_{\textcolor{violet}{5}}$};
\node[xshift = -.2cm, yshift= .3cm]  at (-2.7735,4.16025)
{\small $b_{\textcolor{violet}{6}}$};

\node[xshift = 0.2cm] at (2.13675,3.205125)
{\scriptsize \textcolor{magenta}{$1$}};
\node[yshift = 0.2cm] at (3.75,0)
{\scriptsize \textcolor{magenta}{$4$}};
\node[xshift = 0.2cm] at (2.13675,-3.205125)
{\scriptsize \textcolor{magenta}{$2$}};
\node[xshift = -0.2cm] at (-2.13675,-3.205125)
{\scriptsize \textcolor{magenta}{$2$}};
\node[yshift = 0.2cm] at (-3.75,0)
{\scriptsize \textcolor{magenta}{$4$}};
\node[xshift = -0.2cm] at (-2.13675,3.205125)
{\scriptsize \textcolor{magenta}{$1$}};

\node[xshift = 0.2cm]  at (2,1.125)   {\scriptsize \textcolor{gray}{$2$}};
\node[yshift = -0.2cm] at (0,-2.25)   {\scriptsize \textcolor{gray}{$4$}};
\node[xshift = -0.2cm] at (-2,1.125)  {\scriptsize \textcolor{gray}{$2$}};

\node[yshift = 0.22cm] at (0,2.25) {
    \scriptsize \textcolor{gray}{$34$}
};
\draw plot [smooth] coordinates {(-1.5,2.25) (-.75,2.5) (0,2.25)};
\draw plot [smooth] coordinates {(0,2.25) (.75,2.5) (1.5,2.25)};
\draw plot [smooth] coordinates {(-1.5,2.25) (-.75,2) (0,2.25)};
\draw plot [smooth] coordinates {(0,2.25) (.75,2) (1.5,2.25)};

\node[xshift = 0.3cm,yshift=-.2] at (2,-1.125) {
    \scriptsize \textcolor{gray}{$13$}
};
\draw plot [smooth] coordinates {(2.5,0) (2.42,-0.762) (2,-1.25)};
\draw plot [smooth] coordinates {(2,-1.26) (1.98,-1.79) (1.5,-2.25)};
\draw plot [smooth] coordinates {(2.5,0) (2,-.6) (2,-1.25)};
\draw plot [smooth] coordinates {(2,-1.26) (1.6,-1.56) (1.5,-2.25)};

\node[xshift = -0.3cm,yshift=-.2] at (-2,-1.125) {
    \scriptsize \textcolor{gray}{$13$}
};
\draw plot [smooth] coordinates { (-1.5,-2.25) (-1.52,-1.6) (-2,-1.125) };
\draw plot [smooth] coordinates { (-2,-1.125) (-1.999,-.5) (-2.5,0) };
\draw plot [smooth] coordinates { (-1.5,-2.25) (-2,-1.7) (-2,-1.125) };
\draw plot [smooth] coordinates { (-2,-1.125) (-2.5,-.69) (-2.5,0) };

\draw[fill=white] (1.5,2.25)   circle (0.15cm);
\draw[fill=black] (2.5,0)      circle (0.15cm);
\draw[fill=white] (1.5,-2.25)  circle (0.15cm);
\draw[fill=black] (-1.5,-2.25) circle (0.15cm);
\draw[fill=white] (-2.5,0)     circle (0.15cm);
\draw[fill=black] (-1.5,2.25)  circle (0.15cm);

\draw[fill=black] (2.7735,4.16025)   circle (0.15cm);
\draw[fill=white] (5,0)              circle (0.15cm);
\draw[fill=black] (2.7735,-4.16025)  circle (0.15cm);
\draw[fill=white] (-2.7735,-4.16025) circle (0.15cm);
\draw[fill=black] (-5,0)             circle (0.15cm);
\draw[fill=white] (-2.7735,4.16025)  circle (0.15cm);


\node[xshift = 5.5cm] at (0,0) {
        $\begin{matrix}
        ~
        & \textcolor{violet}{1}
        & \textcolor{violet}{2}
        & \textcolor{violet}{3}
        & \textcolor{violet}{4}
        & \textcolor{violet}{5}
        & \textcolor{violet}{6} \\
        \omega =
        & \textcolor{magenta}{1}
        & \textcolor{magenta}{\overline{4}}
        & \textcolor{magenta}{2}
        & \textcolor{magenta}{\overline{2}}
        & \textcolor{magenta}{4}
        & \textcolor{magenta}{\overline{1}}
    \end{matrix}$
};
\begin{scope}[xshift=35cm, scale=3.5, transform shape, inner sep=0in,outer sep=0in,xshift=-1in]
    \node (n) {
    \begin{ytableau}
        \none & \textcolor{violet}{\scriptstyle 1}\,\textcolor{violet}{\scriptstyle \overline{6}} & \none  \\
        \none & \textcolor{violet}{\scriptstyle 3}\,\textcolor{violet}{\scriptstyle\overline{4}} & \none \\
        \none & \none & \none \\
        \textcolor{violet}{\scriptstyle \overline{2}}\,\textcolor{violet}{\scriptstyle 5}
        & \none & \none \\
    \end{ytableau}
    };
    \draw[line width=.7mm,black] ([xshift=.65cm]n.south west)--([xshift=.65cm]n.north west);
\end{scope}
\end{tikzpicture}
    \caption{On the left we have an hourglass plabic graph corresponding to the plane partition with a single box. The edge labelings are the separation labels (Definition \ref{def: separation-labelings}). The separation labels on the simple edges incident to the boundary vertices $b_1,\ldots,b_6$ give us a lattice word $\omega,$ shown in the center. On the right, we have the oscillating tableau (Definition \ref{def:oscillating-tab}) corresponding to $\omega.$ The correspondence between the web on the left and the tableau on the right is an important component of \cite[Theorem~B]{gaetz2023rotation}.}
    \label{fig:sep-labels-lattice-word-tableau}
\end{figure}

For the sake of clarity, we outline algorithmically how to construct the diagram of an oscillating tableau from a Yamanouchi lattice word following the definitions in \cite{gaetz2023promotion}. An example can be seen in Figure \ref{fig:sep-labels-lattice-word-tableau}. We call $\mathcal{L}_n$ the set of Yamanouchi lattice words of length $n$.
\begin{algorithm}[Lattice word to oscillating tableau]~
\label{algo:lattice-fluct-tab}
    \begin{enumerate}
        \item Let $\omega\in\mathcal{L}_n.$ Then $\omega\in\{\pm 1,\pm 2,\pm 3,\pm 4\}^n.$
        \item Write a multiline array:
        \[
        \binom{1,\ldots, n}{\omega}.
        \]
        This gives pairs $(i,j)$ where $j$ is the $i$-th letter in $\omega.$
        \item Begin with $\lambda^0 = \varnothing$ as in Definition \ref{def:oscillating-tab}.
        \item For $i=1,\ldots,n$ we proceed with the following process:
        \begin{itemize}
            \item Given $i,$ we consider the pair $(i,j)$ from the multiarray, we place $i$ in the $j$-th row of $T.$ If $j$ is negative, denoted $\overline{k},$ then we place $\overline{i}$ in a box in the $k$-th row.
            \item Now we determine the particular box in which to place $i$ or $\overline{i}.$
            \begin{itemize}
                \item \textbf{Case 1:} If $j$ is negative and the row currently has a box to the right of $\lambda^0,$ place $\overline{i}$ in the furthest right such box (this box is removed). Otherwise construct a box in the first open position to the left of $\lambda^0$ and insert the content $\overline{i}.$
                \item \textbf{Case 2:} If $j$ is positive and the row currently has a box to the left of $\lambda^0,$ place $i$ in the furthest left such box (this box is removed). Otherwise construct a box in the first open position to the right of $\lambda^0$ and insert the content $i.$
            \end{itemize}
        \end{itemize}
    \end{enumerate}
\end{algorithm}
Here the boxes a row \emph{currently} has are the boxes of $\lambda^{i-1},$ so a box can record several entries over the course of the algorithm.

Recall that a boundary word of a web is given by the separation labels of the $n$ edges incident to the $n$ vertices along the boundary of the disk on which the web is embedded. It was shown in \cite[Theorem~4.13]{gaetz2023rotation} that a boundary word of a contracted fully-reduced hourglass plabic graph is a Yamanouchi lattice word (Definition \ref{def:lattice-word}), thereby associating to each boundary word a fluctuating tableau. For $U_q(\sl_4)$-webs, the possible separation labels come from the set $\{1,2,3,4\}.$ Let $e_i$ be an edge with separation label $j$ incident to boundary vertex $b_i.$ If $b_i$ is white, then the corresponding letter $\omega_i$ is negative and is written $\overline{j}.$ If $b_i$ is black, then $\omega_i$ is positive and equals $j$.

Since the lattice word for an oscillating tableau is also the boundary word given by the separation labelings of a web, by classifying webs according to their boundary word, we are associating a particular oscillating tableau to each class of webs. Thus the map from a $U_q(\sl_4)$-web to an oscillating tableau is given by first reading the boundary word of the web and then applying Algorithm \ref{algo:lattice-fluct-tab}. In Figure \ref{fig:sep-labels-lattice-word-tableau} we give an example of how to map a web to an oscillating tableau.
Since the map from webs to oscillating tableaux is dependent only on the boundary word corresponding to a web, there is a noticeable loss of information. Any two distinct webs $W,W'$ with the same boundary word are mapped to the same oscillating tableau.
Thus, the map from $U_q(\sl_4)$-webs to oscillating tableaux is not immediately reversible since we would not know how to construct the internal structure of the web with the information given up till this point. In order to produce a web in the equivalence class of a boundary word, one may use the growth algorithm detailed for $U_q(\sl_4)$-webs given by hourglass plabic graphs in \cite[Algorithm~5.1]{gaetz2023rotation}. Growth rules for $U_q(\sl_3)$-webs were given in \cite{kuperberg,petersen2009promotion}. Finally, we describe the growth rules for $U_q(\sl_2)$-webs in Definition \ref{def:growth-rules-basic}, which will be useful in Section \ref{sec:A combinatorial map between invariant spaces} for mapping certain $U_q(\sl_4)$-webs to $U_q(\sl_2)$-webs.

\begin{defn}
\label{def:growth-rules-basic}
    Let $T$ be a rectangular $2$-row oscillating tableau with lattice word $\omega=\omega_1\cdots\omega_n.$ For each letter $\omega_i$ we place a vertex on a horizontal line with a half-edge going downward, and we do the following:
    \begin{itemize}
        \item if $\omega_i$ is barred, then the corresponding vertex is white. Otherwise, it is black;
        \item if $\omega_i\in\{1,\barone\}$ give the half-edge label $1.$ If $\omega_i\in\{2,\bartwo\},$ give the half-edge label $2.$
    \end{itemize}
    Then use the following growth rules to connect the half-edges. Reading from left to right, the half-edge of each letter $2$ or $\barone$ is connected to the nearest half-edge to its left of a letter $1$ or $\bartwo$ that has not yet been connected, using the rule below that matches their labels and colors. Since every connection is made to the nearest available half-edge, no two edges cross.
    \begin{figure*}[htb]
    \label{fig:growth-rules-basic}
        \centering
\begin{tikzpicture}[scale = 0.5]

\coordinate (1) at (-14,0);
\coordinate (2) at (-10,0);
\coordinate (3) at (-6,0);
\coordinate (4) at (-2,0);
\coordinate (5) at (2,0);
\coordinate (6) at (6,0);
\coordinate (7) at (10,0);
\coordinate (8) at (14,0);

\draw (1) -- (-12,-2.5) -- (2);
\draw[fill=white] (-12,-2.5) circle (0.15cm);
\node[xshift = -0.3cm] at (-13,-1.25)    {$1$};
\node[xshift = 0.3cm]  at (-11,-1.25)    {$2$};
\node[yshift = -1cm]   at (-12,-2.5) {\small Growth Rule 1};

\draw plot [smooth] coordinates {(3) (-4,-2.5) (4)};
\node[xshift = -0.4cm] at (-5,-1.25)    {$2$};
\node[xshift = 0.4cm]  at (-3,-1.25)    {$2$};
\node[yshift = -1cm]   at (-4,-2.5)     {\small Growth Rule 2};

\draw (5) -- (4,-2.5) -- (6);
\draw[fill=black] (4,-2.5) circle (0.15cm);
\node[xshift = -0.3cm] at (3,-1.25)    {$2$};
\node[xshift = 0.3cm]  at (5,-1.25)    {$1$};
\node[yshift = -1cm]   at (4,-2.5)     {\small Growth Rule 3};

\draw plot [smooth] coordinates {(7) (12,-2.5) (8)};
\node[xshift = -0.4cm] at (11,-1.25)    {$1$};
\node[xshift = 0.4cm]  at (13,-1.25)    {$1$};
\node[yshift = -1cm]   at (12,-2.5)     {\small Growth Rule 4};

\foreach \i in {1,2,4,7} {
    \draw[fill=black] (\i) circle (0.15cm);
}

\foreach \i in {3,5,6,8} {
    \draw[fill=white] (\i) circle (0.15cm);
}
\end{tikzpicture}
\end{figure*}
\end{defn}

\begin{remark}
\label{rmk: a2-spiders}
    Definition \ref{def:growth-rules-basic} gives a combinatorial rule for constructing a $U_q(\sl_2)$-web. The representation theories from Rumer, Teller, and Weyl \cite{weyl1932valenztheorie}, and developed by Temperley and Lieb \cite{temperley2004relations}, were shown by Kuperberg to be isomorphic to certain $U_q(\sl_2)$-webs called ``spiders" \cite{kuperberg}. The definition given here, although not known to be explicitly stated elsewhere, can be naturally inferred from the work of Kuperberg.
\end{remark}

\subsection{Symmetry classes of plane partitions}
\label{subsec: symmetry-classes-of-pp}
We now return to the combinatorial object which sparked our interest in the particular class of webs studied in this paper. A \emph{plane partition} $p$ is a set of stacked unit cubes given by the following requirement:
if the outermost corner of a unit cube is at the positive integer lattice point $(n,m,\ell)\in\mathbb{N}^3,$ then there is also a unit cube at all positions $(i,j,k)$ with $1\leq i\leq n,$ $1\leq j\leq m,$ and $1\leq k\leq \ell.$  It is natural to consider plane partitions to be restricted to an $a\times b\times c$ box so that $i\leq a, j\leq b, k\leq c$. We denote the set of plane partitions in an $a\times b\times c$ box as $\pp(a,b,c).$

MacMahon introduced plane partitions in \cite{macmahon1899partitions,macmahon2001combinatory} and both conjectured and proved that the number of plane partitions inside an $a\times b\times c$ box is given by:
\begin{equation*}
    \#\pp(a,b,c) = \prod_{i=1}^a\prod_{j=1}^b\frac{i+j+c-1}{i+j-1}.
\end{equation*}

He went on to characterize six symmetry classes of plane partitions and suggest enumeration formulas for them. Richard Stanley described four more symmetry classes. In total, there are ten symmetry classes of plane partitions. All ten are described in \cite{Stanley-symmetries} where he gives many of their counting formulas.

Of the $10$ distinct symmetry classes, we will be concerned with the four symmetry classes that contain the symmetry class with the most restrictions. We now list them from least restrictive to most restrictive:

\begin{enumerate}
    \item Symmetric (SPP)
    \item Cyclically symmetric (CSPP)
    \item Totally symmetric (TSPP)
    \item Totally symmetric and self-complementary (TSSCPP)
\end{enumerate}

The enumeration formula for the number of symmetric partitions was conjectured by MacMahon \cite{macmahon1899partitions} and proven by Andrews \cite{andrews1978plane} and Macdonald \cite{macdonald1979symmetric}. The enumeration of cyclically symmetric plane partitions was originally conjectured by Macdonald \cite{macdonald1979symmetric} and proven by Andrews \cite{andrews1979plane}. The enumeration of totally symmetric plane partitions was originally conjectured by Macdonald \cite{macdonald1979symmetric} and proven by Stembridge \cite{stembridge1995enumeration}. Finally, the enumeration of totally symmetric self-complementary plane partitions was first observed by Robbins and proven by Andrews \cite{andrews1994plane}. For the exact formulas see \cite{kuperberg2002symmetry}.

It is standard to denote the number of
cubes with first two coordinates $(i,j)$
by $p_{ij}.$ This is sometimes called \emph{matrix notation.} We give an example of a plane partition of each symmetry class written in matrix notation in Figure~\ref{fig:symmetry-classes-matrix-form}.

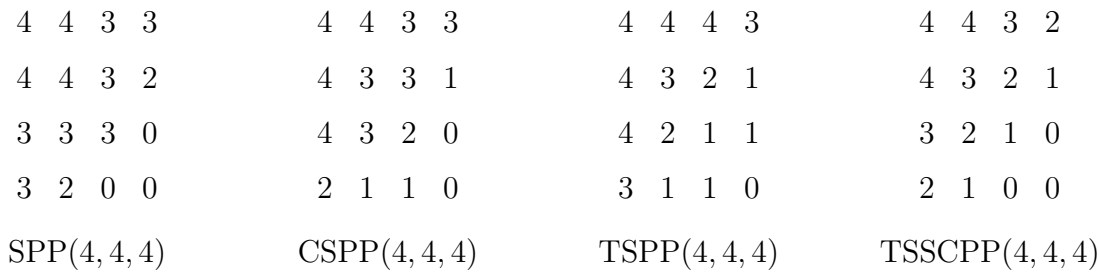
\begin{figure}[htbp]
    \centering
    \begin{tikzpicture}
        \node at (-6,0) {
        $
        \begin{matrix}
            4 & 4 & 3 & 3 \\
            4 & 4 & 3 & 2 \\
            3 & 3 & 3 & 0 \\
            3 & 2 & 0 & 0
        \end{matrix}
        $
        };
        \node at (-6,-2) { $\spp(4,4,4)$};
        \node at (-2,0) {
        $
        \begin{matrix}
        4 & 4 & 3 & 3 \\
        4 & 3 & 3 & 1 \\
        4 & 3 & 2 & 0 \\
        2 & 1 & 1 & 0
        \end{matrix}
        $
        };
        \node at (-2,-2) { $\cspp(4,4,4)$};
        \node at (2,0) {
        $
        \begin{matrix}
        4 & 4 & 4 & 3 \\
        4 & 3 & 2 & 1 \\
        4 & 2 & 1 & 1 \\
        3 & 1 & 1 & 0
        \end{matrix}
        $
        };
        \node at (2,-2) { $\tspp(4,4,4)$};
        \node at (6,0) {
        $
        \begin{matrix}
        4 & 4 & 3 & 2 \\
        4 & 3 & 2 & 1 \\
        3 & 2 & 1 & 0 \\
        2 & 1 & 0 & 0
        \end{matrix}
        $
        };
        \node at (6,-2) { $\tsscpp(4,4,4)$};
    \end{tikzpicture}

    \caption{
        The matrix notation for the plane partitions in Figure \ref{fig:symmetry-classes-example}.
    }
    \label{fig:symmetry-classes-matrix-form}
\end{figure}

\begin{example}
     The definitions of the symmetries in Figure \ref{fig:symm-class-def} can be compared to the examples in Figure \ref{fig:symmetry-classes-matrix-form}. For the case of $\spp,$ if we draw a diagonal along $p_{i,i},$ then $p_{i,j}=p_{j,i}$ for all $i,j$ ($p_{1,3}=p_{3,1}=3$), thus it is symmetric along the diagonal. In the case of $\cspp,$ the $i$-th row is conjugate (in the sense of integer partitions) to the $i$th column for all $i.$ For example, in row $1$ we have $(4,4,3,3)$ which is conjugate to column $1,$ given by $(4,4,4,2)$. In the case of $\tspp,$ $p$ is both symmetric across the diagonal and the $i$th row is conjugate to the $i$th column (row $1$ is $(4,4,4,3)$ which is equal to column $1$). Finally, in the case of $\tsscpp,$ $p$ is symmetric along the diagonal, it is cyclically symmetric, and it is self-complementary since $p_{i,j}+p_{a-i+1,b-j+1} = c$ for all $1\leq i\leq a,1\leq j\leq b$ ($p_{1,2}+p_{4-1+1,4-2+1} = p_{1,2}+p_{4,3}=4+0=4).$ Thus, it indeed belongs in $\tsscpp(4,4,4).$
\end{example}

Each symmetry class has requirements on the shape of the $a\times b\times c$ box. These requirements, along with the definitions for each symmetry class, are given in the chart of Figure \ref{fig:symm-class-def}.

\begin{figure}[ht]
\begin{center}
\begin{tabular}{|l|l| }
\hline
\rowcolor[HTML]{E6E6E6} Symmetry Class & Definition \\
\hline
$\spp(a,a,c)$       & $p_{i,j}=p_{j,i}$ for all $i,j$     \\
$\cspp(a,a,a)$      & The $i$th row of $p$ is conjugate to the $i$th column for all $i$\\
$\tspp(a,a,a)$      & Invariant under all permutations of the coordinates  \\
                    & (equivalently, both symmetric and cyclically symmetric)  \\
$\scpp(a,b,c)$      & $p_{i,j}+p_{a-i+1,b-j+1} = c$ for all $1\leq i\leq a,1\leq j\leq b.$ \\
$\tsscpp(2d,2d,2d)$ & Both totally symmetric and self-complementary\\

\hline
\end{tabular}
\end{center}
    \caption{Definitions of select symmetry classes of plane partitions}
    \label{fig:symm-class-def}
\end{figure}

The symmetry classes can be described as follows:
\begin{itemize}
    \item A $\spp$ is invariant under reflection along the diagonal;
    \item A $\cspp$ is invariant under $120^\circ$ rotations;
    \item A $\tspp$ is invariant under $120^\circ$ and $240^\circ$ rotations and under reflection along each of the three diagonals;
    \item A $\tsscpp$ is invariant under reflections and $120^\circ$ rotations, so it has dihedral symmetries.
\end{itemize}

Instead of considering the entire plane partition of each symmetry class, it is often useful to restrict to a minimal portion of the plane partition from which the full object can be generated by applying the appropriate symmetry operation.
We require that the restriction be such that it avoids redundancy by representing each symmetry equivalence class exactly once.
This restricted portion is called a \emph{fundamental domain}.

In Figure \ref{fig:symm-class-&-fundamental-domain}, we have described each symmetry class of interest in this paper, its symmetry group, and its fundamental domain in terms of the ambient box, which can be viewed as a hexagon.

\begin{figure}[ht]
\begin{center}
\begin{tabular}{|l|l|l| }
\hline
\rowcolor[HTML]{E6E6E6} Symmetry Class & Symmetry Group & Fundamental Domain  \\
\hline
$\spp(a,a,c)$       &  $\Z_2$     & $1/2$  of the hexagon\\
$\cspp(a,a,a)$      &  $\Z_3$      & $1/3$  of the hexagon\\
$\tspp(a,a,a)$      &  $S_3$      & $1/6$  of the hexagon\\
$\tsscpp(2d,2d,2d)$ &  $D_{12}$   & $1/12$ of the hexagon\\

\hline
\end{tabular}
\end{center}
    \caption{Symmetry classes of plane partitions, their symmetry group, and their fundamental domain}
    \label{fig:symm-class-&-fundamental-domain}
\end{figure}

We have given an example of these fundamental domains in Figure \ref{fig:symmetry-classes-example}.

\section{Plane Partitions as webs}
\label{sec: Lattice words of symmetry classes}
In this section we explore the bijection between plane partitions and certain $U_q(\sl_4)$-webs. The bijection is explicitly stated in \cite{gaetz2023rotation} but was known in terms of dimer covers (perfect matchings) of graphs whose internal vertices are trivalent. In the perfect matching, the hourglass edges correspond to the included edge of the dimer cover \cite{kenyon2005}.
In Subsection \ref{subsec:hourglass-plabic-graphs and -pp} we give labels to the faces and boundaries of plane partitions. We also describe the paths of the trips for any web given by a plane partition. In Subsection \ref{subsec: A description of fundamental domains} we describe the fundamental domains of certain symmetry classes of plane partitions and label the boundaries corresponding to these domains. We then give a decomposition of the lattice words corresponding to each boundary. This section will be crucial in the proof of Theorem \ref{thm: lattice-words-symmetry-classes}. Finally, in Subsection \ref{subsec: A description of hourglass edges along boundaries} we describe how to handle hourglass edges that are intersected by a boundary of a fundamental domain.

\subsection{A description of hourglass plabic graphs on plane partitions}
\label{subsec:hourglass-plabic-graphs and -pp}

Given a unit cube in an arbitrary plane partition $p\in\pp(a,b,c),$ the bijection to $U_q(\sl_4)$-webs is well-understood in the sense that, for each visible face and for each visible edge of the unit cube inside $p,$ we can write down the precise edge type in the corresponding web. The unit cube has at most $3$ visible faces which we will call (1) the north face, (2) the east face, and (3) the west face, as shown in Figure \ref{fig: cube-web-bijection}.  Each of these faces corresponds to an hourglass edge with the white vertex to the right while each edge in the unit cube corresponds to a simple edge in the web which connects the hourglass edges on the faces.

\begin{figure}[ht]
    \centering
    \includegraphics[scale=0.3]{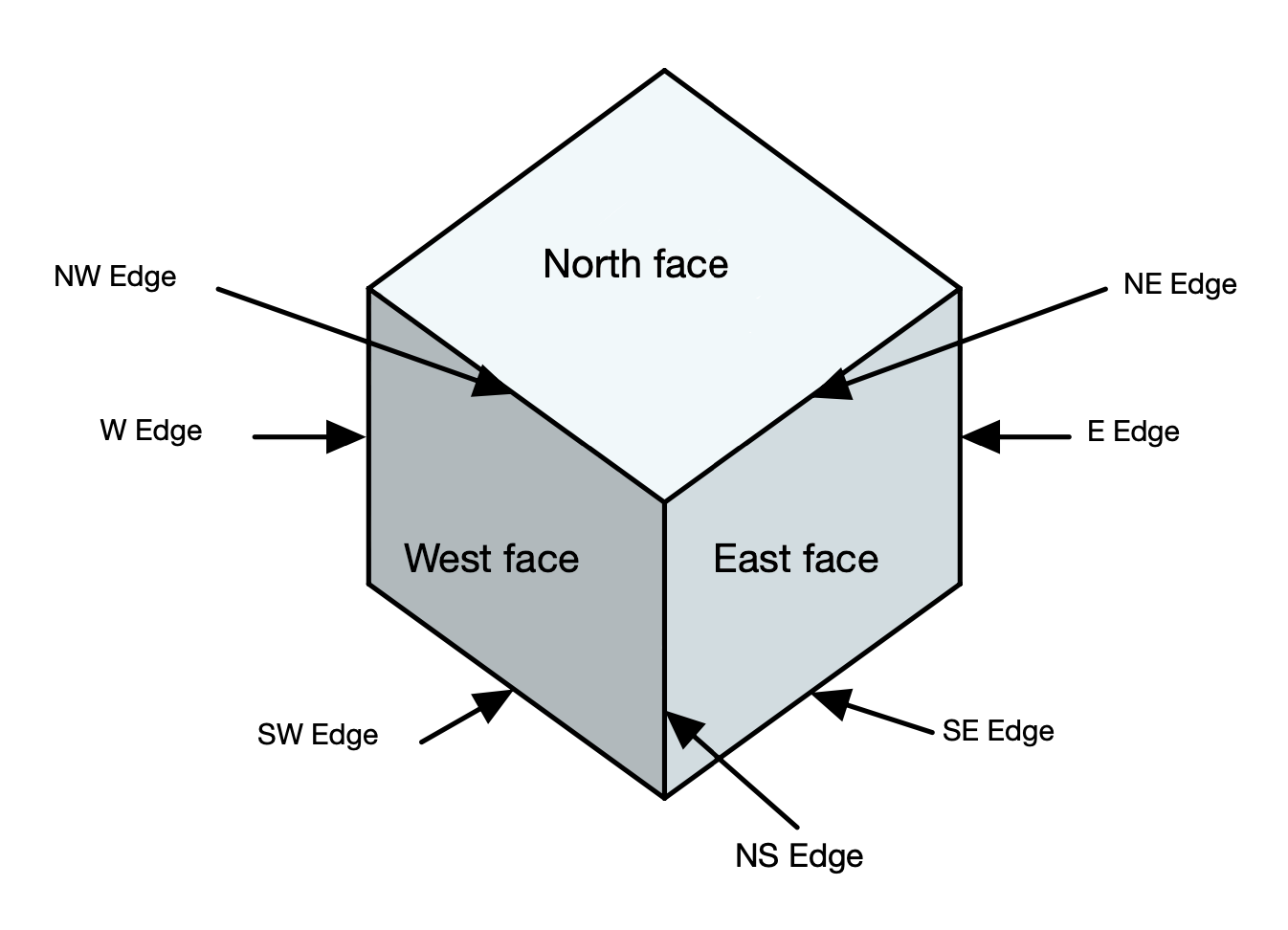}
    \includegraphics[scale=0.3]{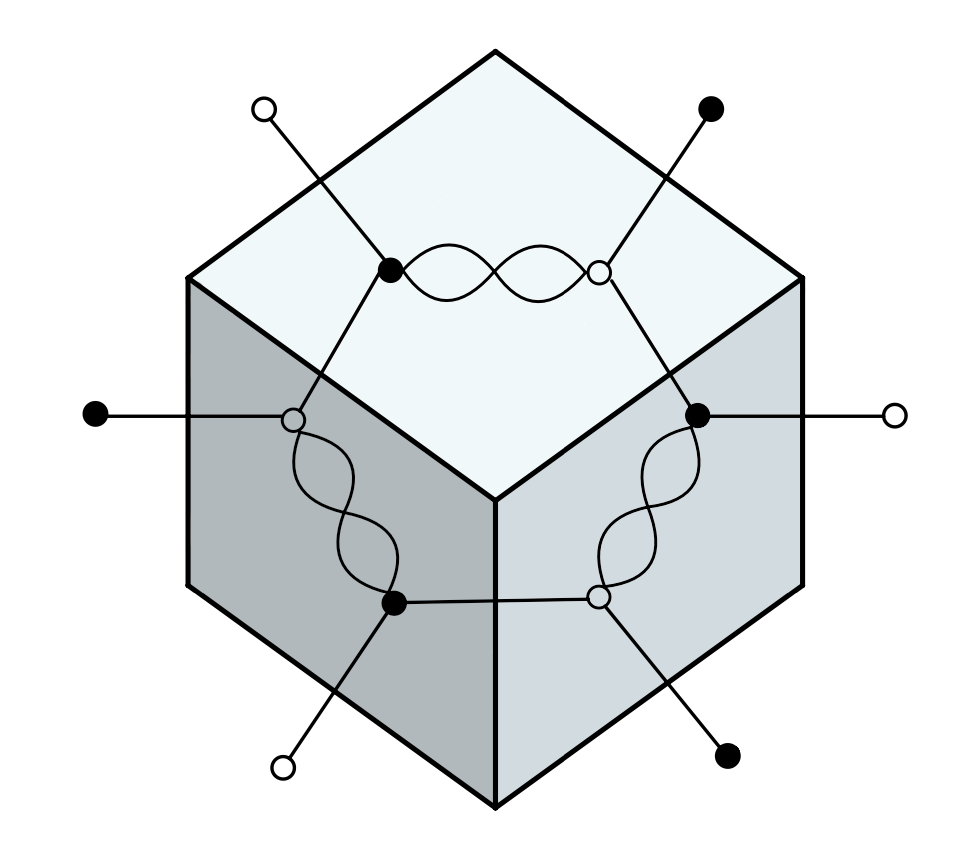}
    \caption{Face labelings of unit cube}
    \label{fig: cube-web-bijection}
\end{figure}

We use these labels to expand on the proof of \cite[Prop.~5.5]{gaetz2023rotation}, which was restated in this paper in Theorem \ref{thm:general-lattice-words}. The trips lie on very predictable paths, which is due to our understanding of how a benzene move affects a trip.

\begin{lemma}
\label{lemma: benzene-moves-and-trips}
Given a hexagonal face of a $U_q(\sl_4)$-web on which a benzene move can be applied (that is, the edges of the face alternate between hourglass edges and simple edges) and a spoke edge $e,$ (i.e., an edge $e=\{v_1,v_2\}$ such that $v_1$ is incident to the hexagon face and $v_2$ is not) the application of a benzene move on the face alters the $\trip_i(e)$ for $i=1,2,3$ in the following manners:
\begin{itemize}
    \item $\trip_1(e)$ and $\trip_3(e)$ are invariant under the benzene move on the hexagon face.
    \item The entry and exit edges for $\trip_2$ are unchanged under a benzene move. However, the path around the hexagon face alternates with each benzene move applied to the face.
\end{itemize}
\end{lemma}

\begin{proof}
     See Figure \ref{fig: benzene-trip-routes}.
\end{proof}

\begin{figure}[ht]
    \centering
        \includegraphics[scale=0.25]{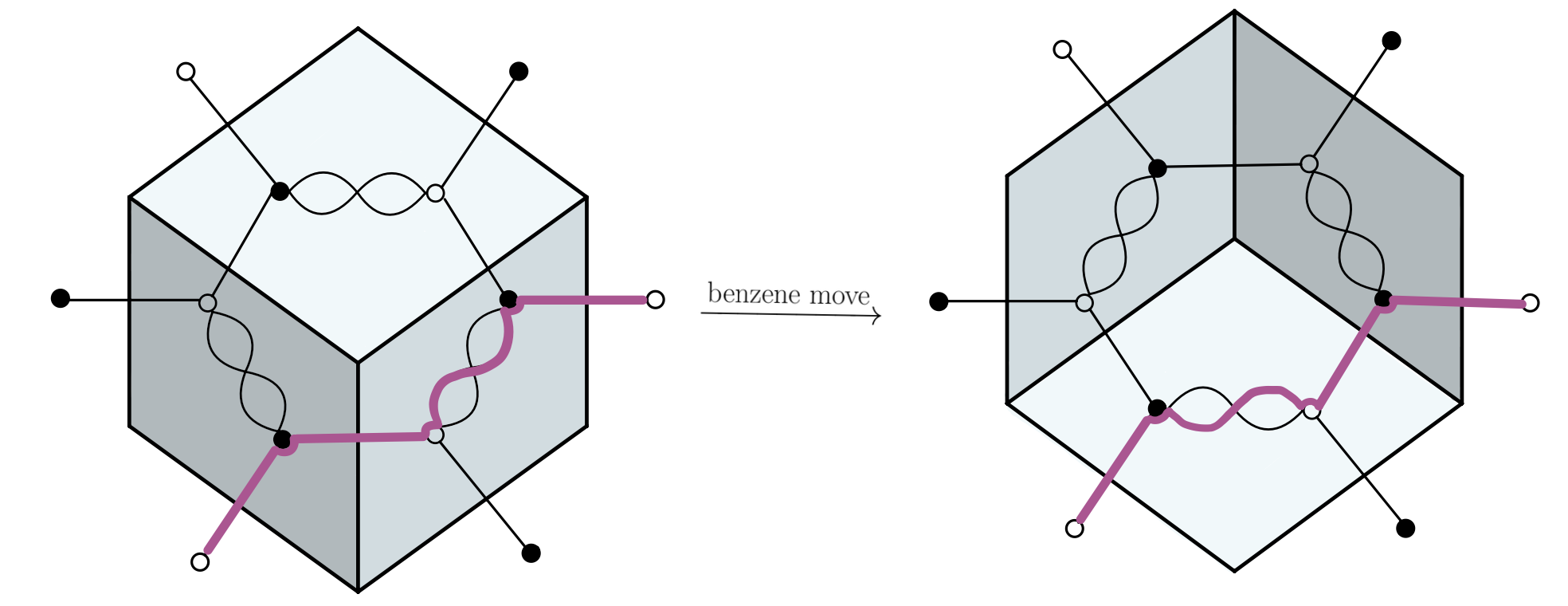}
        \includegraphics[scale=0.2]{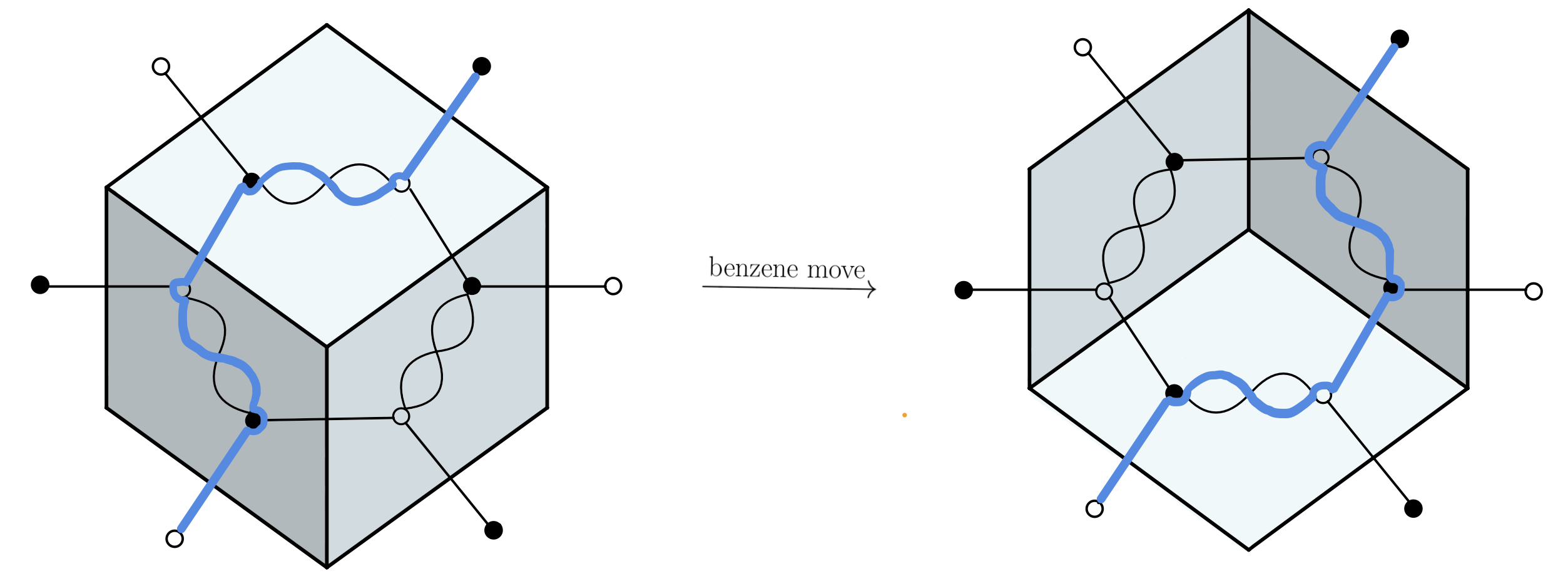}
    \caption{$\trip$ routes through a benzene move: $\trip_1$ and $\trip_3$ are pictured above in magenta and $\trip_2$ is pictured below in blue.}
    \label{fig: benzene-trip-routes}
\end{figure}

The boundary $\b_{\pp}$ of a $U_q(\sl_4)$-web in bijection with a plane partition can be read by setting the base face of the web to be the northern most face (see Figure \ref{fig:fundamental-domain-example}) of the web and reading clockwise around each edge of the ambient hexagon (representing the projection of the ambient $a\times b\times c$ box) of the underlying plane partition. The boundaries are given by the box edges NE, E, SE, SW, W, and NW. Thus, the boundary can be subdivided into the following boundaries: $\b_{\pp}=\b_{NE}\b_{E}\b_{SE}\b_{SW}\b_{W}\b_{NW}.$ We say that a boundary $\b_\boxdot$ is $k$ boundaries from $\b_\bullet$ when reading clockwise if it is reached from $\b_\bullet$ by moving $k$ steps in this cyclic order. For example, $\b_{SE}$ is two boundaries from $\b_{NE}.$

Now that we have described the local behaviour of trips under a benzene move, we can describe the general paths the trips of webs in bijection with plane partitions take. We give an explicit description of Lemma \ref{lem: global-trip-routes} in Figure \ref{fig: global-trip-routes}.

\begin{lemma}
\label{lem: global-trip-routes}
    Given a $U_q(\sl_4)$-web $W$ that is in bijection with a plane partition $p\in\pp(a,b,c),$ then
    \begin{itemize}
        \item $\trip_2$ beginning on an edge incident to a vertex on boundary $\b_\bullet$ ends on an edge incident to a vertex on the boundary $\b_\boxdot$ opposite to $\b_\bullet,$ which is three boundaries from $\b_\bullet$ when reading clockwise.
        \item $\trip_1$ beginning on an edge incident to a vertex on boundary $\b_\bullet$ ends on an edge incident to a vertex on $\b_\boxdot,$ where $\b_\boxdot$ is two boundaries from $\b_\bullet$ when reading clockwise.
        \item $\trip_3$ is the reversal of $\trip_1$. So $\trip_3$ beginning on $\b_\bullet$ ends on the boundary two boundaries from $\b_\bullet$ when reading counterclockwise.
    \end{itemize}
\end{lemma}

\begin{proof}
    Say $\tripstrand_1$ begins on the northeast boundary walking from a black vertex to a white vertex. The trip immediately takes a left and continues south until the trip encounters a north face of a cube. It then takes a left until it reaches another north face where it will take a right. For each north face $\tripstrand_1$ encounters (whether or not they share an edge or are separated by a west face) it alternates between taking a left or taking a right. Thus, $\trip_1$ will always terminate at the southeast boundary when it begins on the northeast boundary.

    If $\trip_1$ begins on the eastern boundary, then it will walk along the west faces until they encounter an east face at which time they will alternate between taking rights and lefts at each east face $\trip_1$ encounters, thereby terminating at the southwest boundary.

    If $\trip_1$ begins on the southeast boundary, it will walk straight across any north face until it reaches an east face, at which point, it will alternate between taking left and right at each east face it encounters until reaching the western boundary. For $\trip_1$s beginning on the northwest, west, or southwest boundaries, we simply reflect the paths just examined. Furthermore, since $\trip_1$ and $\trip_3$ are inverses, we also understand $\trip_3$ by reversing the steps and reading counterclockwise instead of clockwise. These paths have been drawn in pink on the right example of an empty plane partition in Figure \ref{fig: 3x3x3-empty-pp trip descriptions}.

    Finally, $\trip_2$ is special in that it walks directly along the faces and terminates on the opposite boundary. For example, if it begins on the northeast boundary, then it will terminate on the southwest boundary as demonstrated in purple lines on the left example of an empty plane partition in Figure \ref{fig: 3x3x3-empty-pp trip descriptions}. We will use this understanding of the trips in the proof of Theorem \ref{thm: lattice-words-symmetry-classes}.
\end{proof}

\begin{figure}[ht]
\begin{center}
\begin{tabular}{|l|l|l| }
\hline
\rowcolor[HTML]{E6E6E6} $\trip_i$ & (Beginning $\b_\bullet$, Culminating $\b_\boxdot$)  \\
\hline
$i = 1$    &   $(\b_{NE},\b_{SE}),(\b_E,\b_{SW}),(\b_{SE},\b_W),(\b_{SW},\b_{NW}),(\b_W,\b_{NE}),(\b_{NW},\b_{E})$     \\
$i = 2$    &  $(\b_{NE},\b_{SW}),(\b_{E},\b_W),(\b_{SE},\b_{NW}),(\b_{SW},\b_{NE}),(\b_W,\b_E),(\b_{NW},\b_{SE})$     \\
$i = 3$    &  Reverse the pairs in the $\trip_1$ case. \\

\hline
\end{tabular}
\end{center}
    \caption{The pairs of boundaries $(\b_\bullet,\b_\boxdot)$ where if an edge is incident to a vertex on $\b_\bullet,$ then $\trip_i$ ends on $\b_\boxdot.$ }
    \label{fig: global-trip-routes}
\end{figure}

In Figure \ref{fig: 3x3x3-empty-pp trip descriptions} we draw these paths without their underlying webs.
By Lemma \ref{lemma: benzene-moves-and-trips}, a benzene move ultimately does not disturb the overall trajectory of the trip, which is why it is sufficient to demonstrate the trip paths in an empty $3\times 3\times 3$ box.

\begin{figure}[ht]
    \centering
    \includegraphics[scale=0.45]{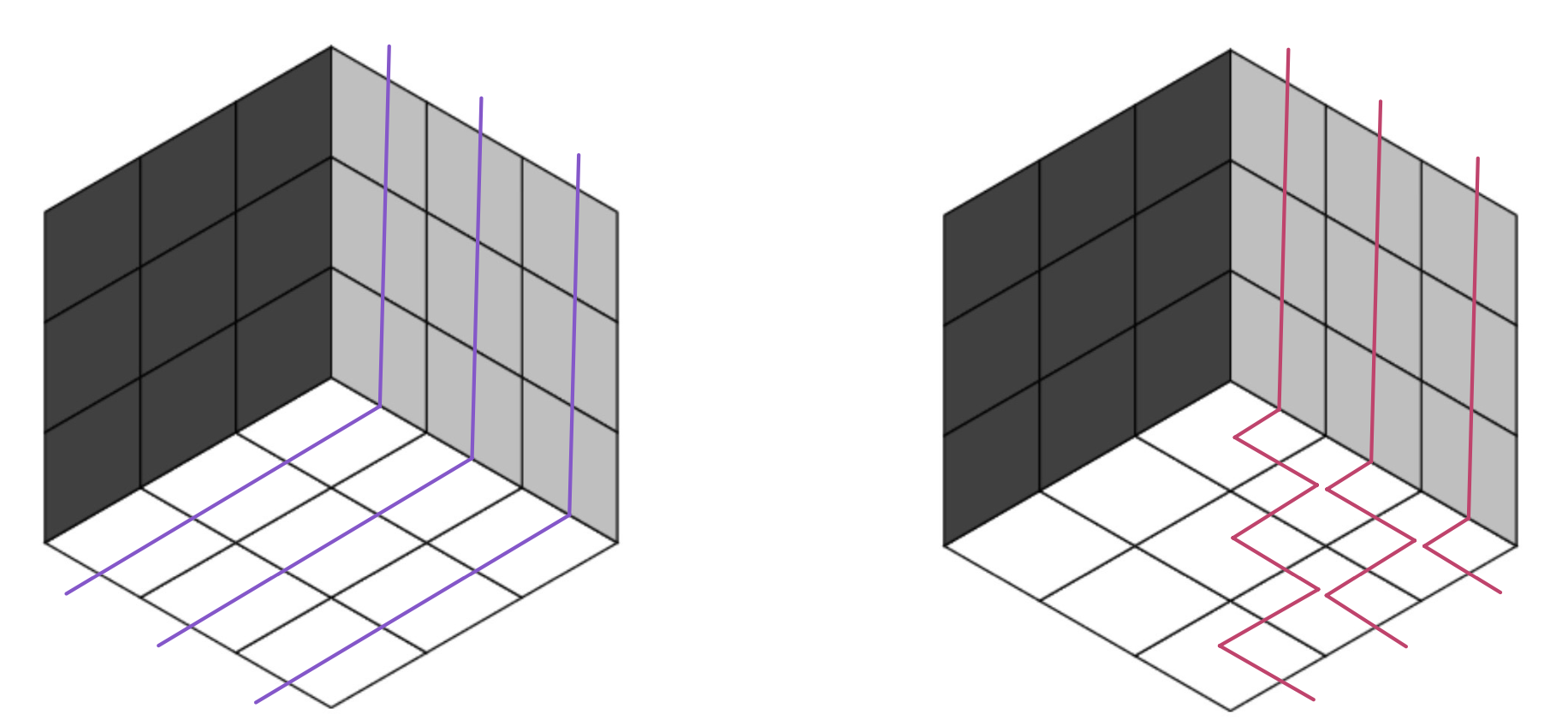}
    \caption{Here we have the empty plane partition in a $3\times 3\times 3$ box. On the left we demonstrate the path $\trip_2$ in purple with the webs removed. It begins on the northeast boundary and ends at the southwest boundary. On the right, in pink, we have the trajectory of $\trip_1,$ also with the webs removed. It begins on the northeast boundary and ends on the southeast boundary. Since $\trip_1$ and $\trip_3$ are inverses of each other, we also have $\trip_3$ beginning on the southeast boundary and ending on the northeast boundary.}
    \label{fig: 3x3x3-empty-pp trip descriptions}
\end{figure}

Throughout this paper, when we restrict a web $W$ in bijection with a plane partition to a subweb containing the base face $F_0$ (for example, to a fundamental domain), we label the edges of the subweb, including the two edges of each split edge (Section \ref{subsec: A description of hourglass edges along boundaries}), by their separation labels computed in $W$ with the trips of $W.$ The boundary word of the restricted web is read from these labels. The following proposition, which will be important to the proof of Theorem \ref{thm: lattice-words-symmetry-classes}, follows immediately from the description of trips in this section.

\begin{prop}
\label{prop: restriction preserves words}
    Given a web $W$ that is in bijection with a plane partition $p$, the separation labels for edges in $W$ remain unchanged when the web is restricted to a subweb containing the fixed base face.
\end{prop}

\begin{proof}
    Combining Lemmas \ref{lemma: benzene-moves-and-trips} and \ref{lem: global-trip-routes}, we have completely determined the behaviour of the trips on $W$ and thus, the separation labels for each edge $e\in W.$ With the convention above, the separation labels of the restricted web are computed with the trips of $W$ and the same base face $F_0,$ so restricting the web does not change them.
\end{proof}

\subsection{A description of fundamental domains of plane partitions and the corresponding lattice words}
\label{subsec: A description of fundamental domains}
Recall the boundary $\b_{\pp}=\b_{NE}\b_{E}\b_{SE}\b_{SW}\b_{W}\b_{NW}$ of a web in bijection with a plane partition from Section \ref{subsec:hourglass-plabic-graphs and -pp}.

Recall from Section \ref{subsec: symmetry-classes-of-pp} that in order to understand the symmetry classes of plane partitions, we restrict the plane partition to a fundamental domain.
In this paper, we are choosing to restrict ourselves to the eastern half of the $a\times a\times c$ box for the fundamental domain for a $\spp.$
Pictorially, we will draw a line through $p_{i,i}$ for all $i.$
We will call the new diagonal boundary $\b_{180}.$
The fundamental domain for any other symmetry class is contained in the fundamental domain for a symmetric plane partition.
We will conflate notation so that given a web $W$ in bijection with a plane partition $p,$ the subweb obtained by restricting to the fundamental domain of its symmetry class will also be called $W.$
Distinguishing between the two webs will never be necessary in this paper.
We demonstrate these boundaries in Figure \ref{fig:fundamental-domain-example}.

\begin{figure}[ht]
    \centering
    \includegraphics[width=0.5\linewidth]{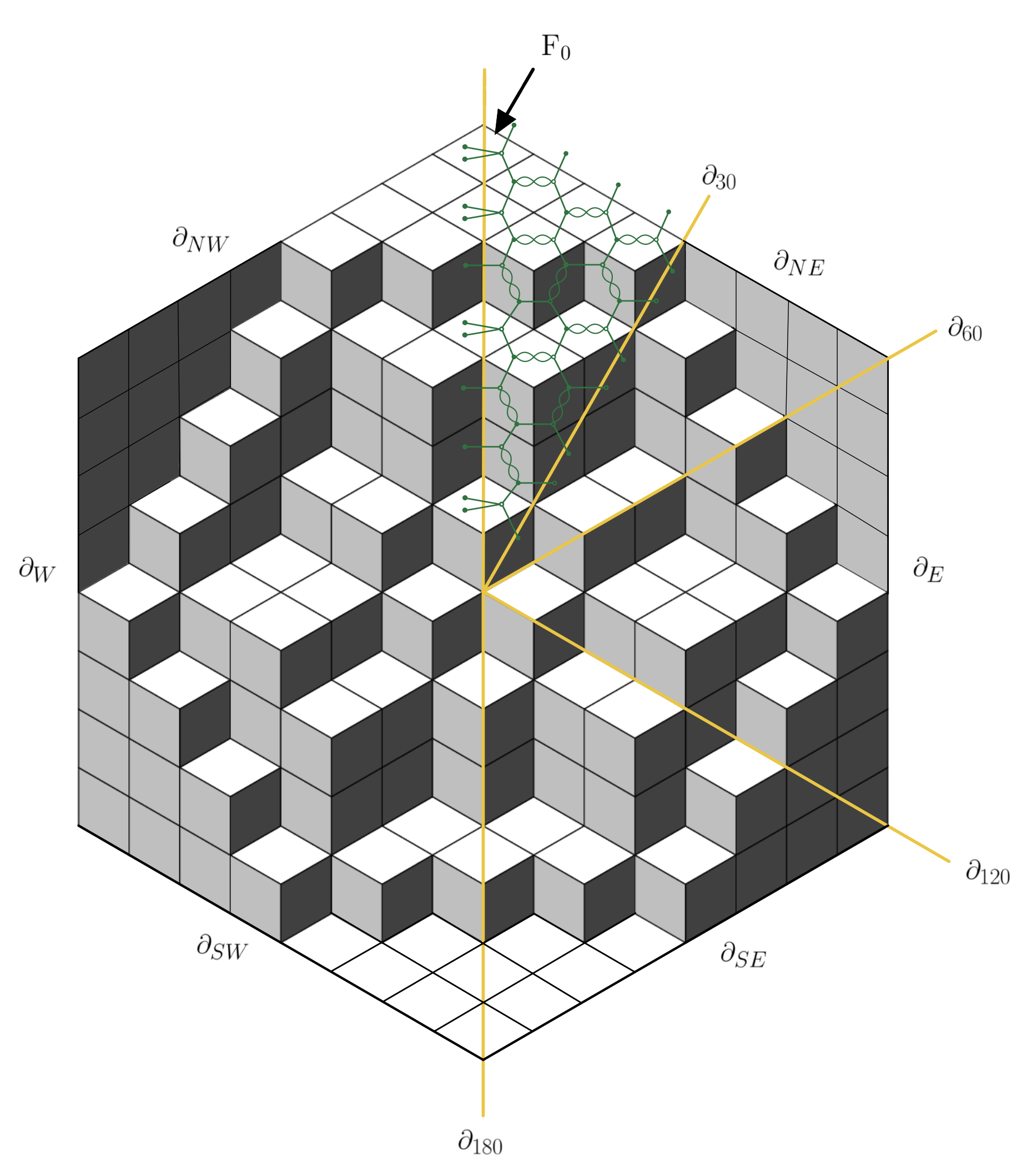}
    \caption{This is an example plane partition $p\in\tsscpp(8,8,8)$ from Figure \ref{fig:TSSCPP-web-boundary-word}. Since any  $p\in\tsscpp(2d,2d,2d)$ is also totally symmetric, cyclically symmetric, and symmetric,  we may draw the fundamental domains for each symmetry class; the corresponding boundary is labeled by $\b_{30},\b_{60},\b_{120},$ and $\b_{180}$, respectively. We have also labeled the base face $F_0$ necessary for separation labeling computations.}
    \label{fig:fundamental-domain-example}
\end{figure}

We can similarly break down the boundaries of the fundamental domains of symmetry classes of plane partitions into boundary segments. Since the boundaries can be subdivided, it follows that any boundary word can also be broken down into subwords by reading from our chosen base face $F_0$ clockwise and labeling the subwords $\omega_{B},$ where $B$ denotes the boundary segment given by either a cardinal direction or degrees for where the plane partition has been restricted. If the fundamental domain only uses half of a boundary $\b_\bullet,$ we call that boundary $\b_{\widehat{\bullet}}$ and the corresponding subword $\omega_{\widehat{\bullet}}.$ We give a full description of the boundaries and their corresponding lattice words now.
\begin{itemize}
    \item For a general plane partition,  the boundary can be partitioned into the following subboundaries:
    \[
    \b_{\pp} =
    \b_{\mathrm{NE}}
    \b_{\mathrm{E}}
    \b_{\mathrm{SE}}
    \b_{\mathrm{SW}}
    \b_{\mathrm{W}}
    \b_{\mathrm{NW}}.
    \]
    Thus, the corresponding lattice word can be partitioned into subwords corresponding to each subboundary:
    \[
    \omega_{\pp} =
    \omega_{\mathrm{NE}}
    \omega_{\mathrm{E}}
    \omega_{\mathrm{SE}}
    \omega_{\mathrm{SW}}
    \omega_{\mathrm{W}}
    \omega_{\mathrm{NW}}.
    \]
    \item For a symmetric plane partition we have
    \[
    \b_{\spp} =
    \b_{\mathrm{NE}}
    \b_{\mathrm{E}}
    \b_{\mathrm{SE}}
    \b_{180}
    \]
    and
    \[
    \omega_{\spp} =
    \omega_{\mathrm{NE}}
    \omega_{\mathrm{E}}
    \omega_{\mathrm{SE}}
    \omega_{180}.
    \]
     \item
    For a cyclically symmetric plane partition we have
    \[
    \b_{\cspp} =
    \b_{\mathrm{NE}}
    \b_{\mathrm{E}}
    \b_{120}
    \b_{\widehat{180}}
    \]
    and
    \[
    \omega_{\cspp} =
    \omega_{\mathrm{NE}}
    \omega_{\mathrm{E}}
    \omega_{120}
    \omega_{\widehat{180}}.
    \]
    \item For a totally symmetric plane partition we have
    \[
    \b_{\tspp} =
    \b_{\mathrm{NE}}
    \b_{60}
    \b_{\widehat{180}}
    \]
    and
    \[
    \omega_{\tspp} =
    \omega_{\mathrm{NE}}
    \omega_{60}
    \omega_{\widehat{180}}.
    \]
    \item For a totally symmetric self complementary plane partition we have
    \[
    \b_{\tsscpp} =
    \b_{\widehat{\mathrm{NE}}}
    \b_{30}
    \b_{\widehat{180}}
    \]
    and
    \[
    \omega_{\tsscpp} =
    \omega_{\widehat{\mathrm{NE}}}
    \omega_{30}
    \omega_{\widehat{180}}.
    \]
\end{itemize}
Those subwords given by cardinal directions have been previously described in \cite[Prop.~5.5]{gaetz2023rotation} restated in this paper in Theorem \ref{thm:general-lattice-words}. Due to Proposition \ref{prop: restriction preserves words}, when $p\in\pp(a,b,c),$ the subwords are
\begin{align*}
    \omega_{\mathrm{NE}} = 1^a,\hspace{4mm}
    \omega_{\mathrm{E}}  = \barfour^c, \hspace{4mm}
    \omega_{\mathrm{SE}} = 2^b,  \hspace{4mm}
    \omega_{\mathrm{SW}} = \barone^{a-b}\bartwo^b, \hspace{4mm}
    \omega_{\mathrm{W}}  = 4^c,\hspace{2mm}  \text{ and }   \hspace{2mm}
    \omega_{\mathrm{NW}} = \barone^b,
\end{align*}
where we assume $a\geq b$ as in Theorem \ref{thm:general-lattice-words}.

The proof of Theorem \ref{thm: lattice-words-symmetry-classes} then gives the rest of the subwords for each symmetry class studied in this paper.

\subsection{A description of hourglass edges along boundaries of fundamental domains}
\label{subsec: A description of hourglass edges along boundaries}
Restricting a web by any of the boundaries in order to obtain the fundamental domain for a given symmetry class results in the intersection of either a simple edge or an hourglass edge with the new boundary. When an hourglass edge with separation label $\{i,j\}$ crosses the new boundary, it becomes a \emph{split} edge. As such, its separation labeling splits and $i$ is associated to one edge and $j$ the other. Throughout this paper we will refer to the pair of separation labels resulting from a split edge with separation label $\{i,j\}$ as $(ij).$
Any simple edge will not be affected by the restriction to the boundary; as such, the trip for the simple edge remains unaffected causing the separation label for the simple edge to be unchanged. This is demonstrated in Figure \ref{fig:boundary-edges}.

\begin{figure}[ht]
\centering
\includegraphics[scale=0.25]{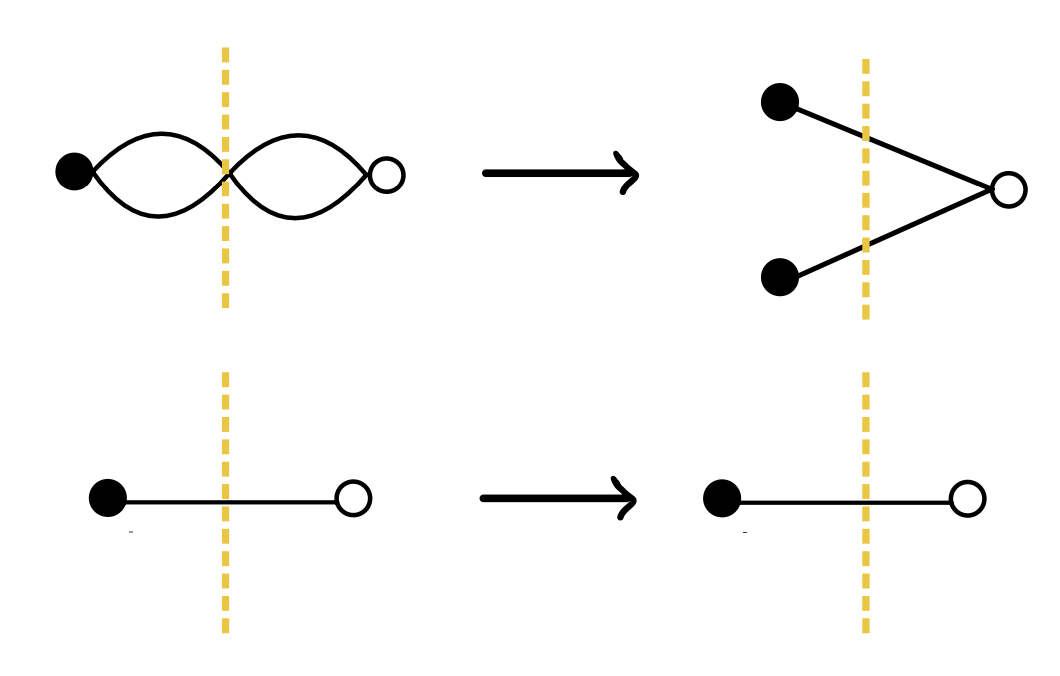}
\caption{The yellow dashed line is the boundary caused by the restriction to the fundamental domain.}
\label{fig:boundary-edges}
\end{figure}

\begin{remark}
    In this paper we have solely focused on oscillating tableaux.
    These correspond to webs that have edges with weight $1$ incident to a boundary vertex.
    However, it is also possible for webs to have hourglass edges incident to a boundary vertex.
    These webs correspond to fluctuating tableaux defined in \cite[Definition 2.1]{gaetz2023promotion} and related to webs in \cite{gaetz2023rotation}.
    The process we describe in this paper of splitting an hourglass edge is \emph{oscillization} of a web (\cite[Definition~3.3]{gaetz2023rotation}), wherein one changes all hourglass edges incident to a boundary vertex to simple edges.
    Before an hourglass edge incident to a boundary vertex is oscillized, it has separation labeling $\{i,j\}$ where order is unknown.
    Once it is oscillized, the separation labels remain the same, but are now ordered. In this paper we denote them $(ij).$
\end{remark}

\section{Proof of lattice words of symmetry classes}
\label{subsec: proof of lattice words of symmetry classes}
We begin this section by restating Theorem \ref{thm: lattice-words-symmetry-classes} to include the specific requirements on the subwords given by the letters in the multisets $\{\cdots\}.$
The proof in this section relies heavily on notation and theory described in Section \ref{sec: Lattice words of symmetry classes}.

\begin{customthm}{1.2}
    The boundary words of $U_q(\sl_4)$-webs corresponding to plane partitions of select symmetry classes are as follows:
\begin{itemize}
    \item $\bm{\spp(a,a,c):}$
        \[
        1^a~\Bar{4}^c~2^a~\{4^c,(34)^a\}
        \]
        \begin{itemize}
            \item $\{4^c,(34)^a\}:$ We may freely order the $c$ $4$s and $a$ $(34)s.$
        \end{itemize}

    \item $\bm{\cspp(a,a,a):}$
        \[
        1^a~\Bar{4}^a~\{(\overline{31})^m,\Bar{1}^{a-m}\}~\{4^{a-m},(34)^m\}
        \]
        \begin{itemize}
            \item $\{(\overline{31})^m,\Bar{1}^{a-m}\}:$ We may freely order the $m$ $(\overline{31})$s and $a-m$ $\Bar{1}$s for $0\leq m\leq a.$
            \item $\{4^{a-m},(34)^m\}:$ If $(\overline{31})$ is in position $i$ in the subword given by $\{(\overline{31})^m,\Bar{1}^{a-m}\}, 0\leq i\leq a-1,$ then $(34)$ is in position $a-i-1$ in the subword given by $\{4^{a-m},(34)^m\}$ (counting positions from $0$).
        \end{itemize}

    \item $\bm{\tspp(a,a,a):}$
        \[
        1^a\{(23)^{a-m},2^m\}~\{4^{a-m},(34)^m\}
        \]
        \begin{itemize}
            \item $\{(23)^{a-m},2^m\}~:$ We may freely order the $a-m$ $(23)$s and the $m$ $2$s for all $0\leq m\leq a.$
            \item $\{4^{a-m},(34)^m\}~:$ For each $1\leq i\leq a,$ the number of $(23)$s among the last $i$ letters of the subword given by $\{(23)^{a-m},2^m\}$ must be less than or equal to the number of $4$s among the first $i$ letters of the subword given by $\{4^{a-m},(34)^m\}.$
        \end{itemize}

    \item $\bm{\tsscpp(a,a,a), a = 2d:}$
        \[
        1^d~\underbrace{2~\Bar{4}~2~\cdots\Bar{4}~2}_{2d-1}~\{4^{d-1},(34)^{d-1}\}~(34)
        \]
        \begin{itemize}
            \item $\{4^{d-1},(34)^{d-1}\}:$ The subword constructed of $d-1$ $4$s and $d-1$ $(34)$s must be Yamanouchi in letters $4$ and $(34)$ when one includes a $4$ at the beginning of the subword and a $(34)$ at the end.
        \end{itemize}
\end{itemize}
\end{customthm}

\begin{proof} A boundary word $\omega$ of an hourglass plabic graph $W$ corresponding to a plane partition $p$  in $\spp(a,a,c),\cspp(a,a,a),\tspp(a,a,a),$ or $\tsscpp(a,a,a)$ is given by certain separation labelings (Definition \ref{def: separation-labelings}) of $W$ constrained to the fundamental domain of a given symmetry class.

In this proof, we begin with the least restricted symmetry class, symmetric. We then proceed to cyclically symmetric, and totally symmetric. Finally, we end with the most restrictive symmetry class, totally symmetric self-complementary.

The proof for each word follows the same form.
By Proposition \ref{prop: restriction preserves words}, the separation labels of edges of a web $W$ remain unchanged when the web is restricted to a subweb that includes the base face of $W.$
As explained in Section \ref{subsec: A description of fundamental domains}, we can decompose the boundary word into subwords corresponding to each boundary of the fundamental domain for the given symmetry class.
We then prove a description for each subword by either using Theorem \ref{thm:general-lattice-words}, which was proven in \cite{gaetz2023rotation}, or we analyze the trips of the edges on the boundary associated to the subword we seek to describe.
For each individual symmetry class, by Proposition \ref{prop: restriction preserves words} (also see a more in depth discussion in Section \ref{subsec: A description of hourglass edges along boundaries}) we may analyze only where the trips begin and end.
Then, for each symmetry class we give an example.
These can be found in Figures \ref{fig:SPP-ex for proof} - \ref{fig:TSSCPP-ex for proof}.

With these details in mind, we now proceed to the proofs for each lattice word for each of the given symmetry classes.

$\bullet$ \textbf{SPP:} (See Figure \ref{fig:SPP-ex for proof}.) For $p\in\spp(a,a,c),$ we consider the web $W,$ which is in bijection with $p,$ restricted to its fundamental domain. By Proposition \ref{prop: restriction preserves words}, the boundary of the web restricted to its fundamental domain can be subdivided into subboundaries:
    \[
    \b_{\spp} =
    \b_{\mathrm{NE}}
    \b_{\mathrm{E}}
    \b_{\mathrm{SE}}
    \b_{180}.
    \]
    Thus, the boundary word for $W$ can be similarly partitioned into subwords corresponding to the division of the boundary of $W$:
    \[
    \omega_{\spp} =
    \omega_{\mathrm{NE}}
    \omega_{\mathrm{E}}
    \omega_{\mathrm{SE}}
    \omega_{180}.
    \]

Following Proposition \ref{prop: restriction preserves words}, Theorem \ref{thm:general-lattice-words} gives
\[
    \omega_{\mathrm{NE}} = 1^a,
    \quad
    \omega_{\mathrm{E}} = \barfour^c,
    \quad
    \text{and}
    \quad
    \omega_{\mathrm{SE}} = 2^a.
\]
So it remains to show that $\omega_{180} = \{4^c, (34)^a\},$ where we may freely choose where to place the $c$ $4$ and the $a$ $(34).$ Since the plane partition $p$ is symmetric, we have $p_{i,j}=p_{j,i}$ for all $i,j.$ This puts no extra restrictions on $p_{i,j}$ such that $i=j$. Therefore, any combination of split edge $(e_1~e_2)$ or simple $g$ along $\partial_{180}$ in $W$ is valid, since any combinations of north faces or NS edges are valid in $p.$ Therefore, it only remains to show the separation labels for $(e_1~e_2)$ is $(34)$ and the separation label for $g$ is $4.$
Note that any edge intersecting $\b_{180}$ is incident to a black boundary vertex. Therefore, for each edge on $\b_{180},$ we seek to know the number of trips that separate the base face of $W$ from the face to the south of that edge.

By Lemma \ref{lem: global-trip-routes}, we know the paths of the trips. In particular, $\trip_1(g)$ ends on the NE boundary, $\trip_i(g)$ end on the W boundary for $i=2,3.$ Therefore, $\sep(g) = 4,$ since all three trips contribute to the separation label. 

Now consider a split edge on $\b_{180}.$ Reading clockwise around the boundary of $W,$ call the first edge $e_1$ and the second edge $e_2.$ Both are incident to black boundary vertices. As in the case of a split edge on $\b_{60}$ in the proof for $\tspp$ below, $\trip_1(e_1)$ follows the path from $e_1$ and ends at $e_2,$ and it does not separate the face to the right of $e_1$ from the base face. The trips $\trip_2(e_1)$ and $\trip_3(e_1)$ end on the W boundary, as for $g,$ and both separate the face to the right of $e_1$ from the base face. Therefore, $\sep(e_1)=3.$ Similarly, $\trip_1(e_2)$ follows the path from $e_2$ to $e_1$ and separates the face to the right of $e_2$ from the base face, and $\trip_2(e_2)$ and $\trip_3(e_2)$ end on the W boundary and also separate. Therefore, $\sep(e_2)=4,$ and the split edge contributes the letter $(34).$ These trips are drawn in the last two pictures of Figure \ref{fig:SPP-ex for proof}.

    \begin{example}
    \label{ex:SPP-ex for proof}
        In Figure \ref{fig:SPP-ex for proof} we have a plane partition in $\spp(4,4,4)$ restricted to its fundamental domain with its corresponding web with base face $F_0.$
        Each image has trips $1,2,3$ given by three different colors for a type of edge incident to each boundary.

        In the first line, we begin with an arbitrary edge $e$ on the boundary $\partial_{NE}$ and the trips that pass through $e.$
        The second figure gives the trips that pass through an arbitrary edge on the boundary $\partial_E$.

        In the second line, the first figure, shows the trips that pass through an arbitrary edge on the boundary $\partial_{SE}.$ The second figure shows the trips for a simple edge on $\partial_{180}.$

        In the third line, the first figure shows the trips that pass through the lower edge of a split edge on $\b_{180}.$ This is the edge with separation label $3$ in the pair $(34).$ The final figure shows the upper edge of a split edge on $\b_{180}.$ This is the edge with separation label $4$ in the pair $(34).$
    \end{example}

    \begin{figure}[!htbp]
    \centering
    \includegraphics[width=0.4\linewidth]{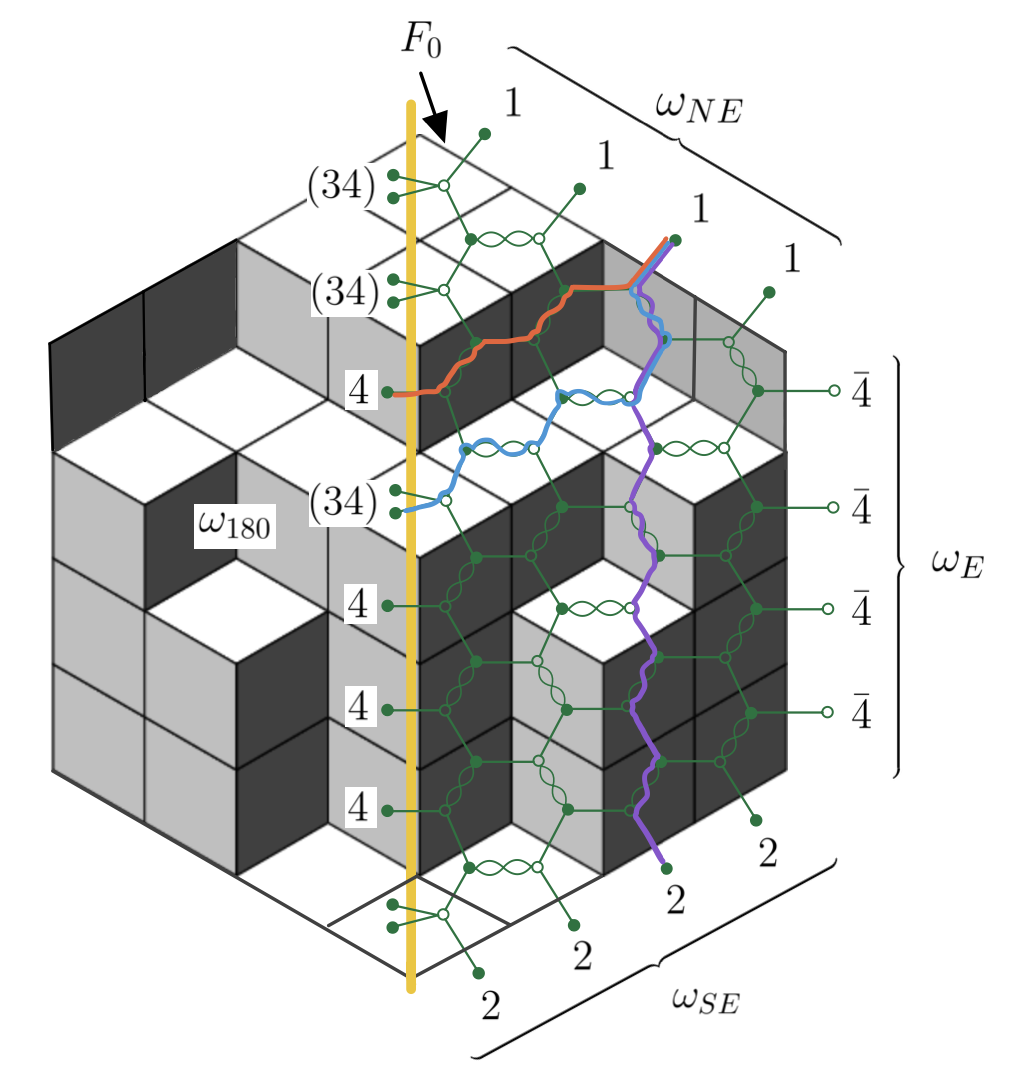}
    \includegraphics[width=0.4\linewidth]{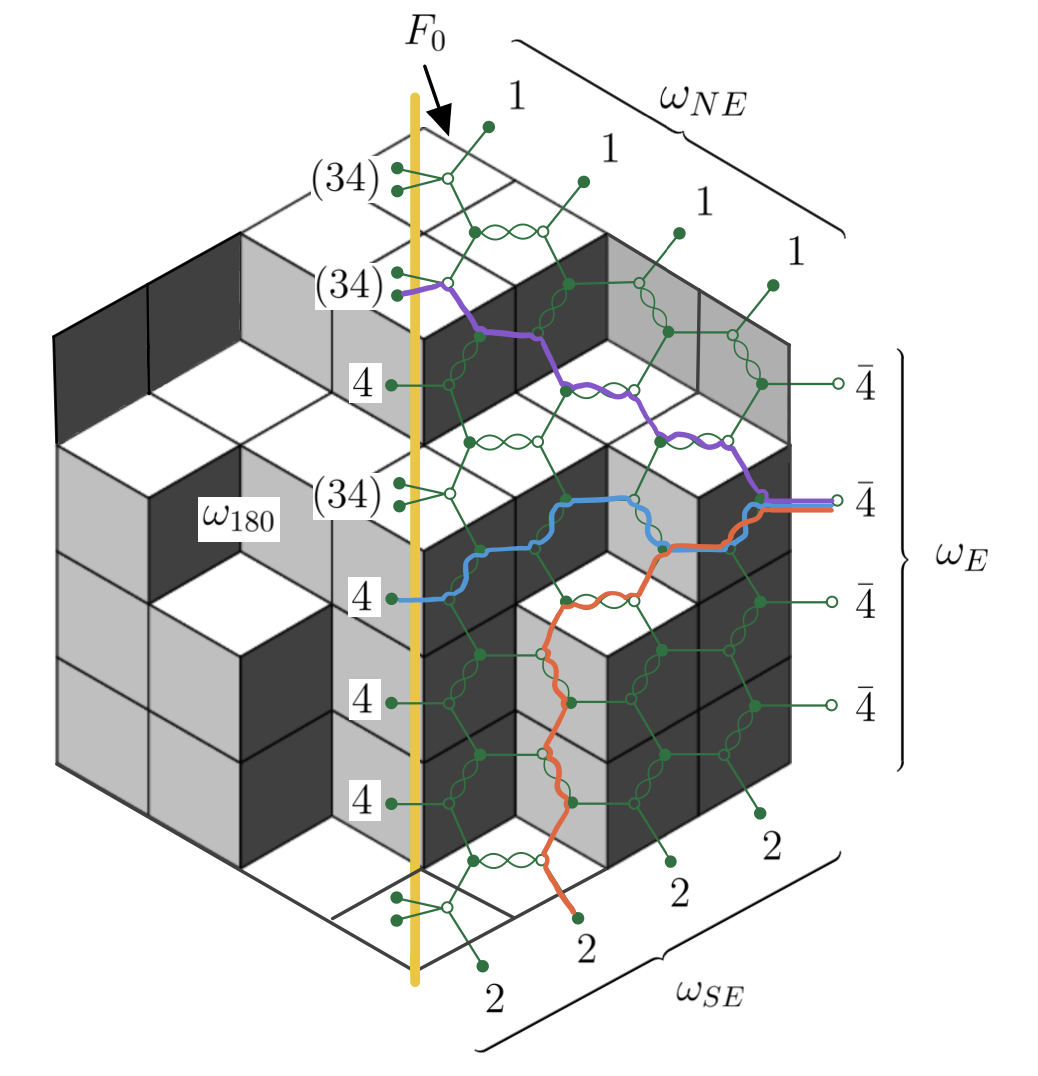}
    \includegraphics[width=0.4\linewidth]{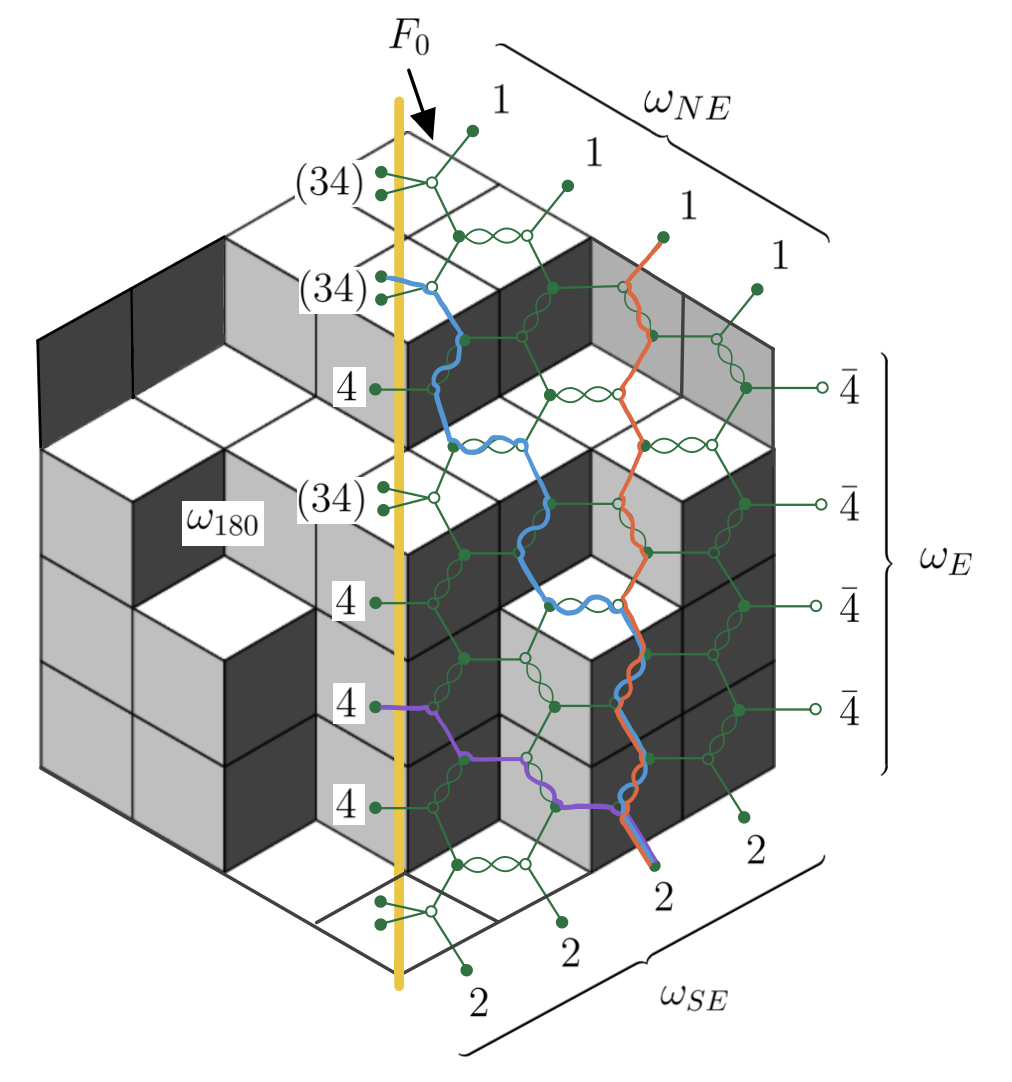}
    \includegraphics[width=0.4\linewidth]{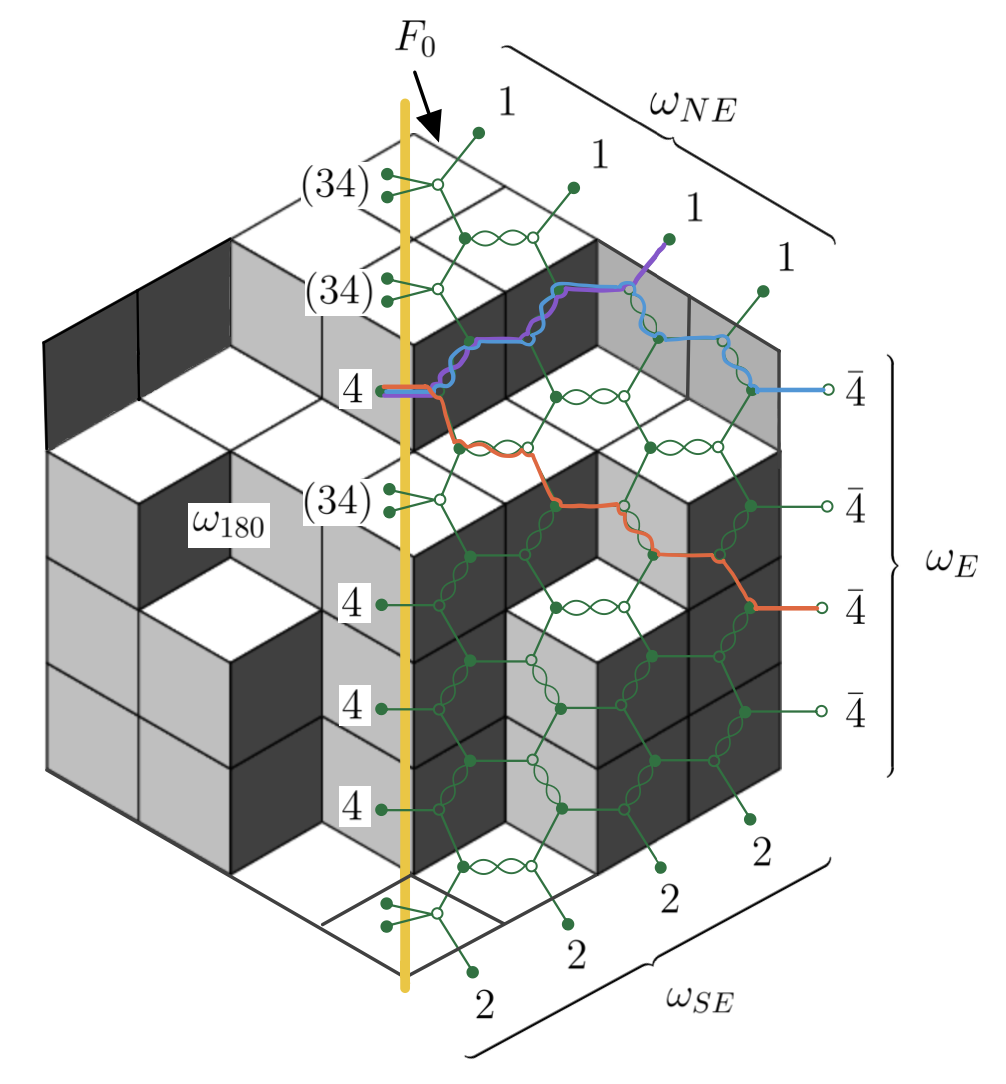}
    \includegraphics[width=0.4\linewidth]{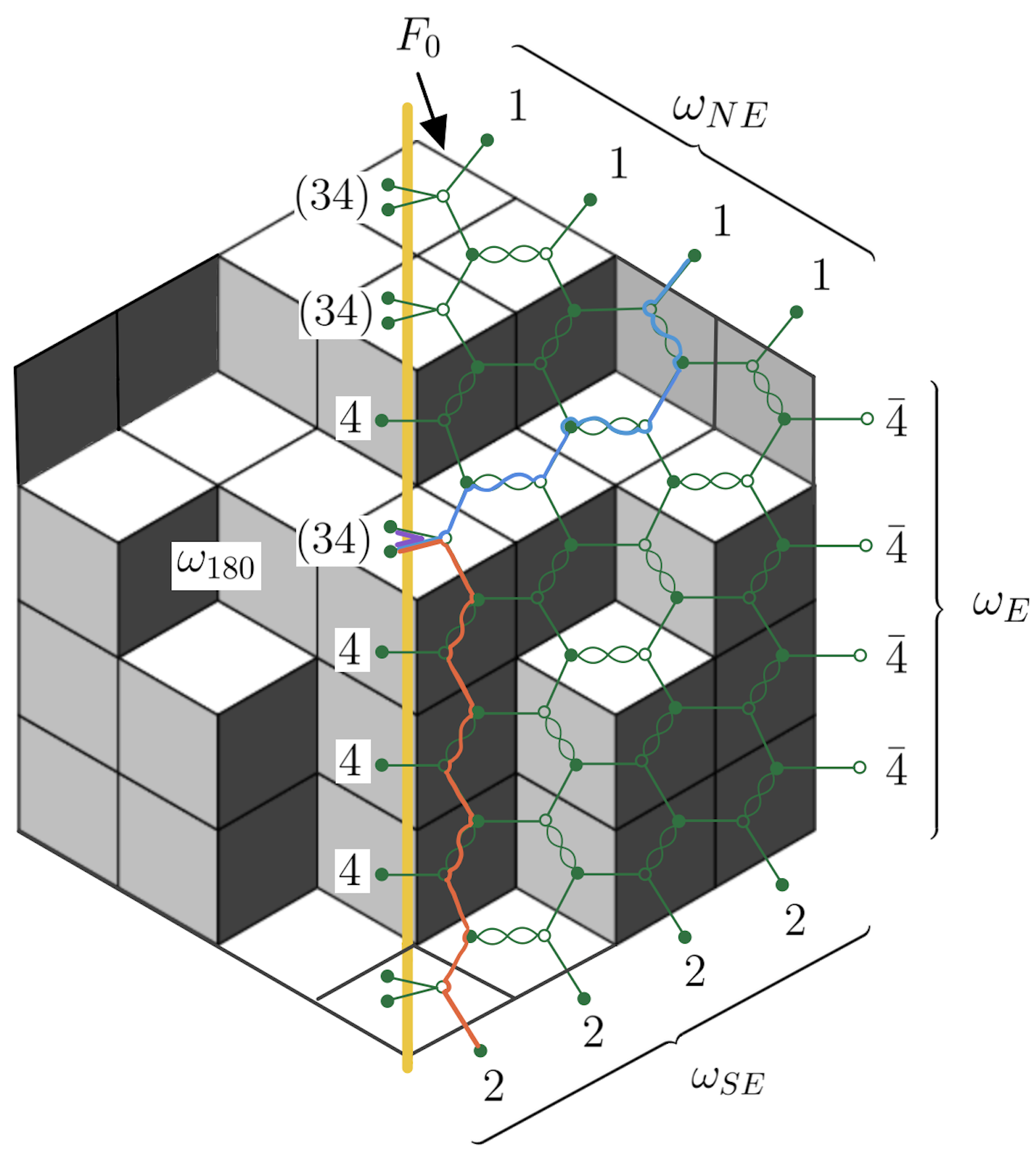}
    \includegraphics[width=0.4\linewidth]{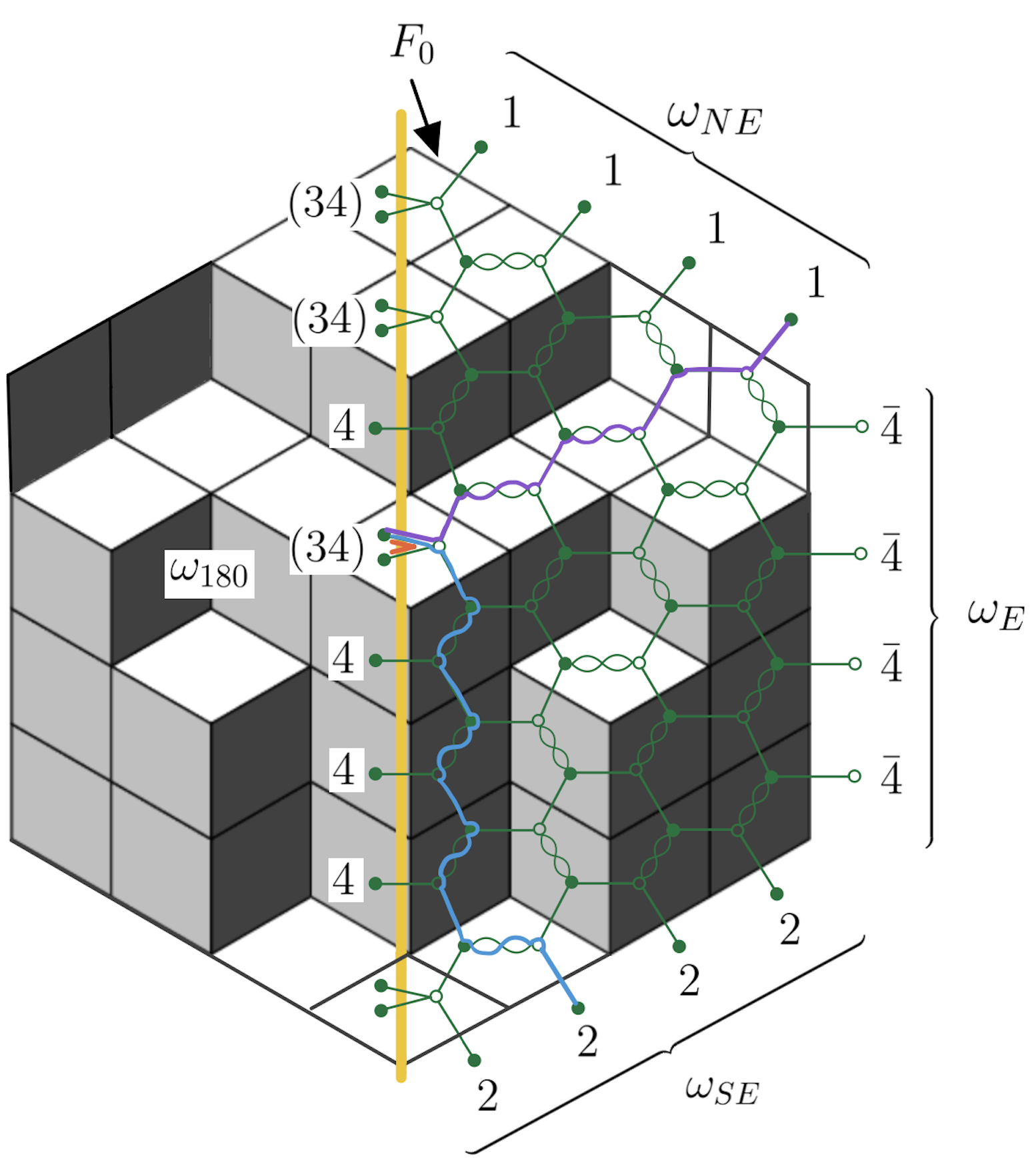}
    \caption{
        A plane partition in $\spp(4,4,4)$ with the trips drawn within its fundamental domain, as discussed in Example \ref{ex:SPP-ex for proof}
    }
    \label{fig:SPP-ex for proof}
    \end{figure}

 \vspace{5mm}
        $\bullet$ \textbf{CSPP:} (See Figure \ref{fig:CSPP-ex for proof}.)
        For $p\in\cspp(a,a,a),$ consider the web $W$ that is in bijection with $p$ restricted to its fundamental domain. By Proposition \ref{prop: restriction preserves words}, the boundary of the fundamental domain of $W$ can be subdivided so that
        \[
        \b_{\cspp} =
        \b_{\mathrm{NE}}
        \b_{\mathrm{E}}
        \b_{120}
        \b_{\widehat{180}}.
        \]
        The corresponding boundary word can be subdivided into the following subwords:
        \[
        \omega_{\cspp} =
        \omega_{\mathrm{NE}}
        \omega_{\mathrm{E}}
        \omega_{120}
        \omega_{\widehat{180}}.
        \]
        From Proposition \ref{prop: restriction preserves words} and Theorem \ref{thm:general-lattice-words},
        \[
            \omega_{\mathrm{NE}} = 1^a
            \quad
            \text{ and }
            \quad
            \omega_{\mathrm{E}} = \barfour^a.
        \]
        So the remaining subwords left to describe are $\omega_{120}$ and $\omega_{\widehat{180}}.$ We begin with showing $\omega_{120} = \{(\overline{31})^m,\barone^{a-m}\}$.

        The boundary $\b_{120}$ intersects either an edge of a unit cube or cuts through an east face of a unit cube. In the first case, $\b_{120}$ intersects a simple edge beginning on a west face of a unit cube in position $(i,j,k)$ crossing its SW edge and ending on either another west face or a north face of a unit cube in position $(i,j,k-1)$ or $(i,j+1,k-1),$ respectively. In either case, we obtain a simple edge with a white boundary vertex. If we call this edge $g,$ then, by the discussion in Lemma \ref{lem: global-trip-routes}, $\trip_1(g)$ and $\trip_2(g)$ end on the NE boundary and $\trip_3(g)$ ends on the E boundary. Thus, none of these contribute to the separation label, giving us $\sep(g)=1,$ but the vertex adjacent to $g$ is white, so we obtain $\barone.$

        In the second case, $\b_{120}$ intersects an east face of a unit cube, which means it intersects an hourglass edge, which results in a split edge with white boundary vertices. Reading clockwise, call the two edges of the split edge $e_1$ and $e_2.$ We once more recall Lemma \ref{lem: global-trip-routes}, which says that $\trip_1(e_1)$ ends on the E boundary, $\trip_2(e_1)$ ends on the NW boundary, and $\trip_3(e_1)$ takes the path along $e_1$ and $e_2.$ Since $\trip_2(e_1)$ and $\trip_3(e_1)$ will separate the face incident to edges $e_1$ and $e_2,$ then $\sep(e_1) = 3.$ Since $e_1$ is incident to a white boundary vertex, then $e_1$ contributes the letter $\barthree.$
        Moving on to $e_2,$ since $e_2$ is a split edge, where $\trip_1(e_2)$ takes a path from a white vertex to black vertex, then the path it takes is simply $e_1$ ending at $e_2.$ Then we observe from Lemma \ref{lem: global-trip-routes}, that $\trip_2(e_2)$ ends on the E boundary and $\trip_3(e_2)$ ends on the NE boundary. Since none of these trips separate the face to the right of $e_2$ as one walks from the black vertex incident to $e_2$ to the white vertex incident to $e_2,$ it follows that $\sep(e_2) = 1.$ Since the boundary vertex incident to $e_2$ is white, then the contributed letter is $\barone.$ Together the split edge $(e_1e_2)$ contributes the pair of letters $(\overline{31}),$ which we will consider from this point forward as a single letter.

        We have now established that $\omega_{120}$ consists of letters $\barone$ and $(\overline{31}).$  The number of letters in $\omega_{120}$ is equal to $a$ since $\b_{120}$ goes from the center to a corner of the $a\times a\times a$ box. If $p\in\cspp(a,a,a)$ is the empty plane partition, then $\omega_{120} = \barone^a.$ However, if $p$ is the full plane partition, $\omega_{120} = (\overline{31})^a.$ By the definition of the symmetry class, any combinations of east faces and NW edges are possible along $\b_{120}.$ Thus, we have $\omega_{120} = \{\barone^{a-m},(\overline{31})^{m}\},$ for $0\leq m\leq a$ for any choice of placement of $\barone$ and $(\overline{31})s$ in $\omega_{120}.$

        We now describe $\omega_{\widehat{180}}.$ Since $\b_{\widehat{180}}$ is half of the $\b_{180}$ boundary, by Proposition \ref{prop: restriction preserves words}, $\omega_{\widehat{180}}$ must consist of only $(34)$s and $4$s. Recall from the definition in Figure \ref{fig:symm-class-def} that for all $i,$ the $i$th row in $p$ is conjugate to the $i$th column, so the only dependency will be between $\omega_{120}$ and $\omega_{\widehat{180}}.$

        In order to address the number of each letter in $\omega_{\widehat{180}}$, we first address the ordering of the letters $(34)$ and $4$ in $\omega_{\widehat{180}}$ through their dependence on the position of the letters in $\omega_{120}.$ The stipulation that  if $(\overline{31})$ is in position $i$ in $\omega_{120},$ then $(34)$ is in position $a-i-1$ follows from the requirement that the $j$th row in $p$ must be conjugate to the $j$th column for all $j.$ Note that the same must also be true for the positions of $\Bar{1}$ in $\omega_{120}$ and $4$ in $\omega_{\widehat{180}}.$ Moreover, this means that for each $(\overline{31})$ in $\omega_{120}$ we obtain a $(34)$ in $\omega_{\widehat{180}}$ and for each $\barone$ in $\omega_{120}$ we have a corresponding $4$ in $\omega_{\widehat{180}}.$ Thus, $\omega_{\widehat{180}} = \{4^{a-m},(34)^m\}$ where the positions of the $4$s and $(34)$s are as described.
    \begin{example}
    \label{ex:CSPP-ex for proof}
        In Figure \ref{fig:CSPP-ex for proof} we have an example of a plane partition in $\cspp(4,4,4)$ restricted to its fundamental domain with its corresponding web with base face $F_0.$ Each image has trips $1,2,3$ given by three different colors for a type of edge incident to each boundary. We have previously shown the trips on boundaries $\partial_{NE},\partial_{E},$ and $\b_{\widehat{180}}.$ Thus we only show the trips on the three types of edges found on the boundary given by a restriction to the fundamental domain of cyclically symmetric plane partitions, that is, $\b_{120}.$
        Beginning with the furthest left figure, we show the trips for a simple edge on $\b_{120}.$ The middle figure gives the trips for the bottom most edge of a split edge on $\b_{120}.$ The furthest right figure gives the trips for the other split edge on the boundary $\b_{120}.$
    \end{example}

    \begin{figure}[htbp]
    \begin{center}
    \includegraphics[width=0.32\linewidth]{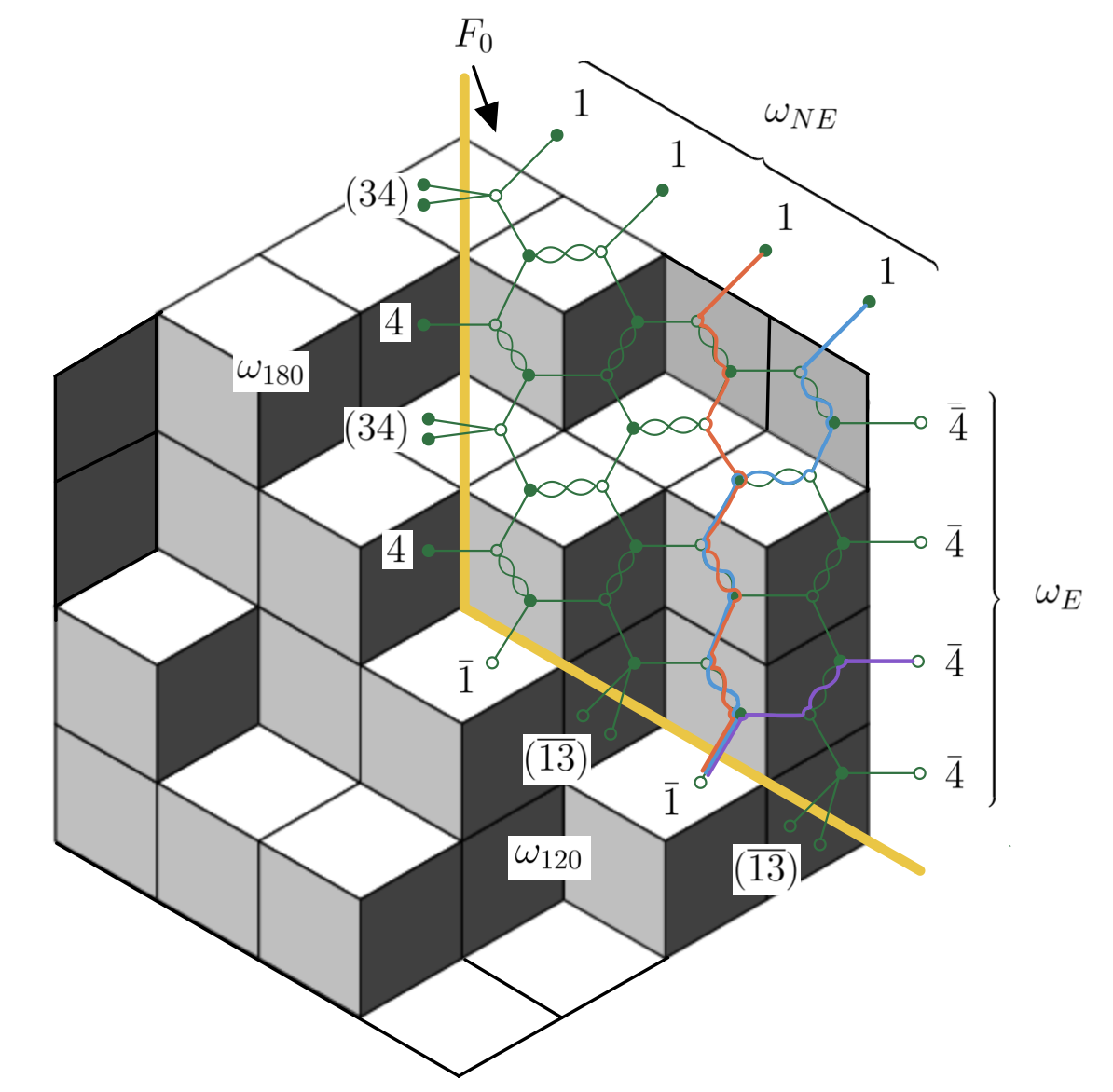}
    \includegraphics[width=0.32\linewidth]{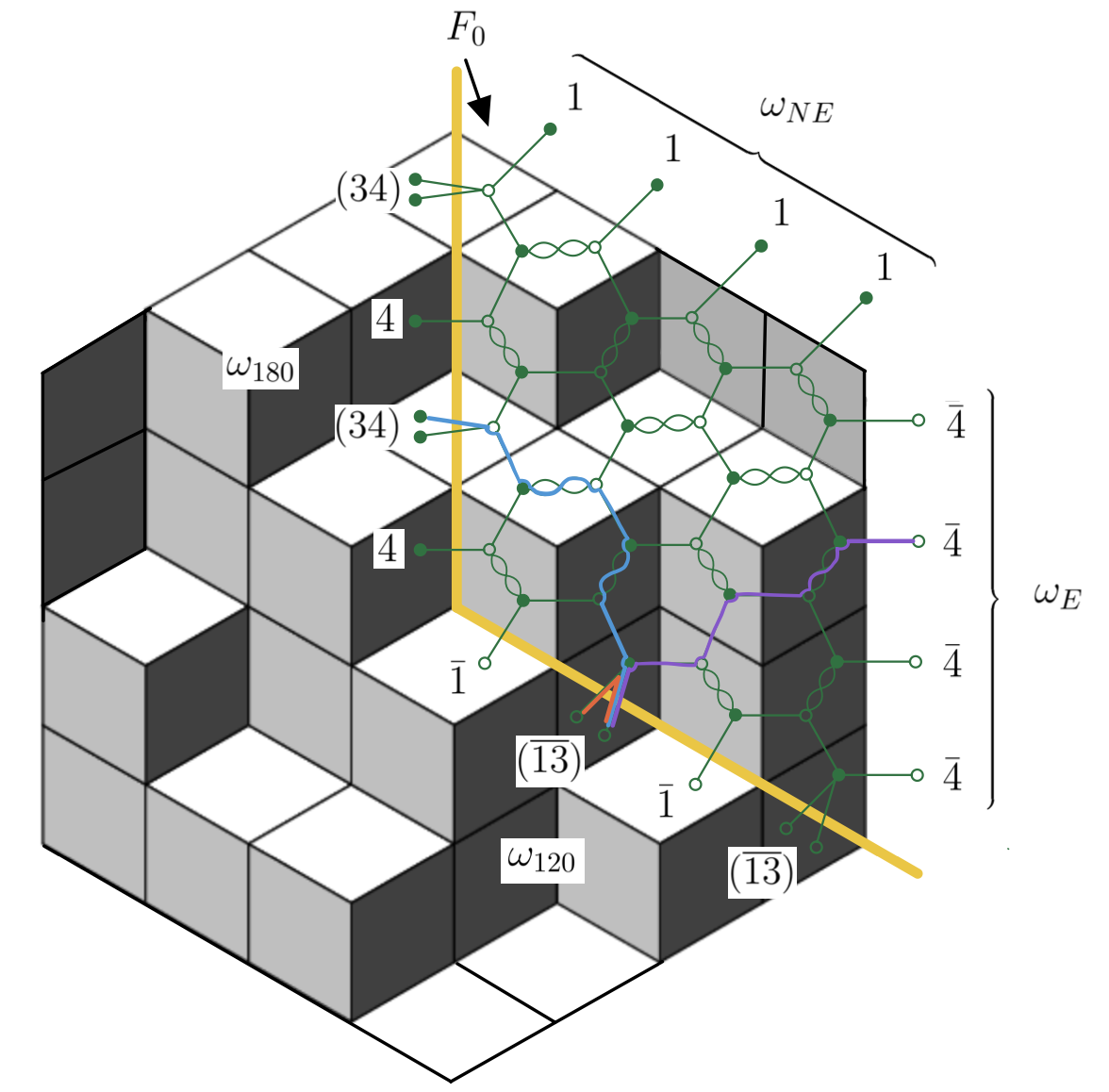}
    \includegraphics[width=0.32\linewidth]{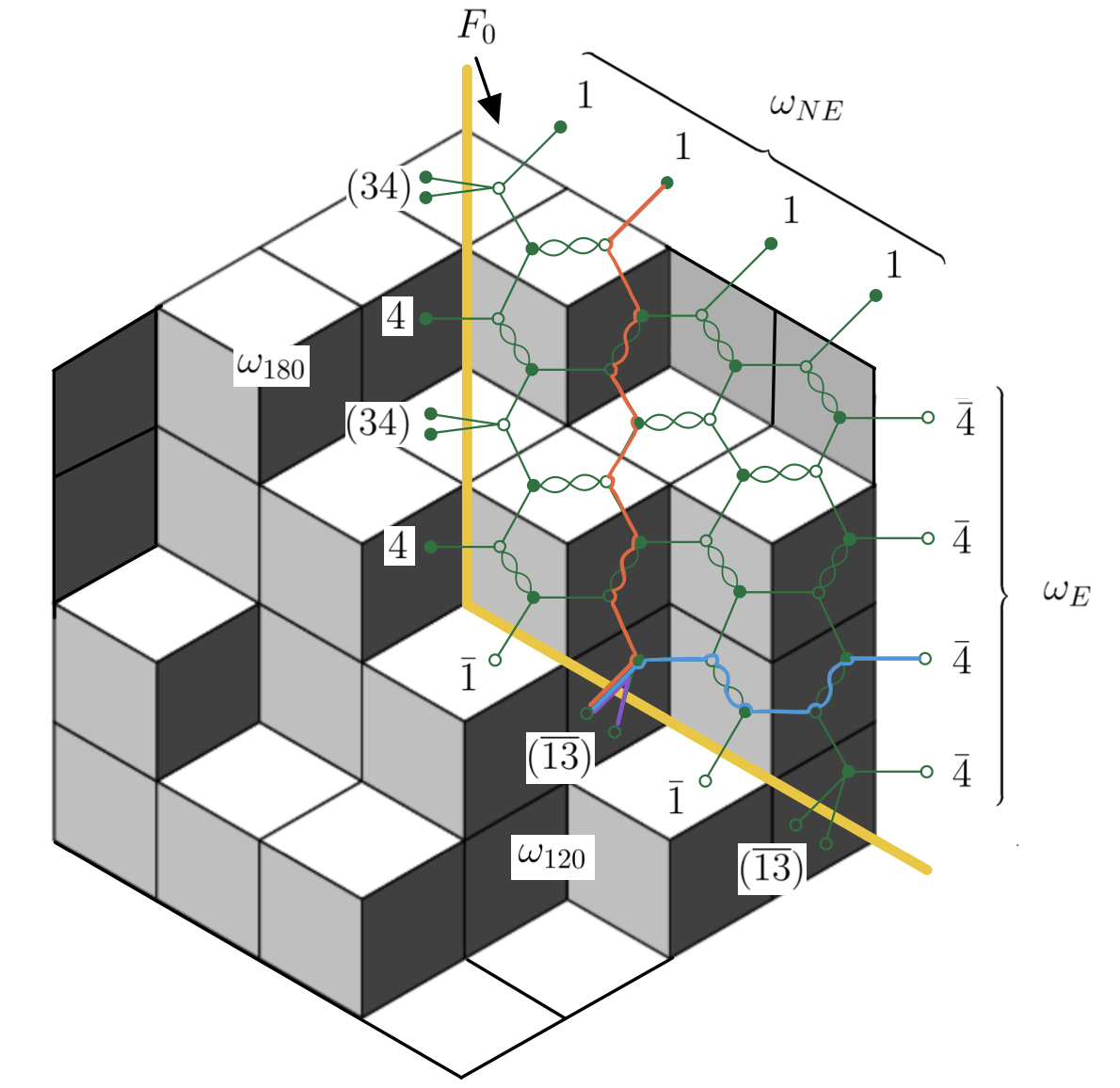}
    \caption{
    A plane partition in $\cspp(4,4,4)$ with the trips drawn within its fundamental domain, as discussed in Example \ref{ex:CSPP-ex for proof}. The lattice word is:
    \[1^4~\barfour^4~(\overline{31})~\barone~(\overline{31})~\barone~4~(34)~4~(34).\]
    }
    \label{fig:CSPP-ex for proof}
    \end{center}
    \end{figure}

\vspace{5mm}
        $\bullet$ \textbf{TSPP:} (See Figure \ref{fig:TSPP-ex for proof}.) For $p\in\tspp(a,a,a),$ consider the web $W$ that is in bijection with $p,$ restricted to its fundamental domain. By Proposition \ref{prop: restriction preserves words}, the boundary of the fundamental domain of $W$ can be partitioned into
        \[
        \b_{\tspp} =
        \b_{\mathrm{NE}}
        \b_{60}
        \b_{\widehat{180}}.
        \]
        The corresponding boundary word, $\omega_{\tspp},$ can be subdivided so that
        \[
        \omega_{\tspp} =
        \omega_{\mathrm{NE}}
        \omega_{60}
        \omega_{\widehat{180}}.
        \]
        From Proposition \ref{prop: restriction preserves words} and Theorem \ref{thm:general-lattice-words},
        \[
        \omega_{\mathrm{NE}} = 1^a.
        \]
        It remains to describe $\omega_{60}$ and $\omega_{\widehat{180}}.$ We will begin with showing that $\omega_{60}$ is a word using letters from the multiset $\{(23)^{a-m},2^m\}$ for $0\leq m\leq a.$

        The boundary $\b_{60}$ cuts through either a west face of a unit cube or an NE edge of a unit cube in $p.$ Thus in $W,$ $\b_{60}$ intersects a simple edge incident to a black boundary vertex or an hourglass edge, which results in a split edge with each edge incident to a black boundary vertex. From Lemma \ref{lem: global-trip-routes}, we can determine exactly where the trips for all of the edges on $\b_{60}$ end.

        We will start with a simple edge $g$ on $\b_{60}.$ We have $\trip_1(g)$ and $\trip_2(g)$ end on the NW boundary, while $\trip_3(g)$ ends on the NE boundary. As we walk on $g$ from the black vertex to the white vertex, the face to the right of $g$ is separated by only $\trip_3.$ Therefore, $\sep(g) = 2$ and contributes the letter $2$ to $\omega_{60}.$

        Moving on to the split edge, reading clockwise around the boundary of $W,$ call the first edge $e_1$ and the second edge $e_2.$ Since $e_1$ comes from a split edge with black boundary vertices, then $\trip_1(e_1)$ simply follows the path from $e_1$ and ends at $e_2.$ Furthermore, $\trip_2(e_1)$ ends on the NW boundary and $\trip_3(e_1)$ ends on the NE boundary. Since only $\trip_3$ will separate the base face from the face to the right of the edge $e_1$ walking from the black vertex to the white vertex, then $\sep(e_1) = 2.$

        We now determine the separation label for the second edge, $e_2$ of the split edge. Since $e_2$ is the second edge in a split edge with black boundary vertices, then $\trip_1(e_2)$ takes the path from $e_2$ to $e_1.$ Furthermore, $\trip_2(e_2)$ ends on the NE boundary and $\trip_3(e_2)$ ends on the NW boundary. Since only $\trip_1(e_2)$ and $\trip_2(e_2)$ separate the base face from the face to the right of $e_2$ as one walks from the black vertex to the white vertex incident to $e_2,$ then $\sep(e_2) = 3.$ Together, the split edge $(e_1~e_2)$ contributes the letter $(23)$ to $\omega_{60}.$

        By definition, if $p\in\tspp(a,a,a),$ then $p$ must be both symmetric and cyclically symmetric. So $p$ must have that $p_{ij} = p_{ji}$ for all $i,j$ and the $i$th row is conjugate to the $i$th column for all $i.$ The symmetric property does not cause there to be any restrictions to $p_{ij}$ along $\b_{60}.$ Then by a similar argument to the case of $p\in\cspp(a,a,a),$ we may freely choose where to place the $m$ $2$s and the $a-m$ $(23)$s for $0\leq m\leq a.$

        We now move on to showing $\omega_{\widehat{180}}$ is comprised of letters $4^{a-m}$ and $(34)^m$ for $0\leq m\leq a.$
        Since $\b_{\widehat{180}}$ is half of $\b_{180},$ then we know the letters in $\omega_{\widehat{180}}$ must be comprised of $4$s and $(34)$s.
        Moreover, since $p\in\cspp(a,a,a)$ and $p\in\spp(a,a,a),$ then if $\omega_{60}$ has $a-m$ $(23)$s, then $\omega_{\widehat{180}}$ has $a-m$ $4$s by the requirements of those symmetry classes.
        This means that there are then $m$ $2$s in $\omega_{60}$ and $m$ $(34)$s in $\omega_{\widehat{180}}.$

        The final requirement follows from the definition of a plane partition in that, the number of unit cubes in each stack of unit cubes along $\b_{60}$ must be less than or equal to the number of unit cubes along $\b_{\widehat{180}}.$ In particular, say we start at the position where $\b_{60}$ and $\b_{\widehat{180}}$ meet, which is between the two subwords $\{(23)^{a-m},2^m\}$ and $\{4^{a-m},(34)^m\}.$ Since each $(23)$ corresponds to a split edge on a west face of a unit cube and each $4$ corresponds to a simple edge crossing an NS edge of a unit cube, if for some $1\leq i\leq a$ the number of $(23)$s among the last $i$ letters of $\omega_{60}$ were greater than the number of $4$s among the first $i$ letters of $\omega_{\widehat{180}},$ then this would mean that the number of unit cubes in a stack on $\b_{60}$ is greater than the number of unit cubes in the corresponding position on $\b_{\widehat{180}},$ thereby breaking the requirements of a plane partition.

    \begin{example}
    \label{ex:TSPP-ex for proof}
        In Figure \ref{fig:TSPP-ex for proof} we have an example of a plane partition in $\tspp(4,4,4)$ restricted to its fundamental domain with its corresponding web with base face $F_0.$ Each image has trips $1,2,3$ given by three different colors for a type of edge incident to each boundary. We have previously shown the trips on boundaries $\partial_{NE}$ and $\b_{\widehat{180}}.$ Thus we only show the trips on the three types of edges found on the boundary given by a restriction to the fundamental domain of totally symmetric plane partitions, that is, $\b_{60}.$
        Beginning with the furthest left figure, we show the trips for a simple edge on $\b_{60}.$ The middle figure gives the trips for an edge of a split edge on $\b_{60}.$ The furthest right figure gives the trips for the other split edge on the boundary $\b_{60}.$
    \end{example}
    \begin{figure}[!htbp]
    \centering
    \includegraphics[width=0.32\linewidth]{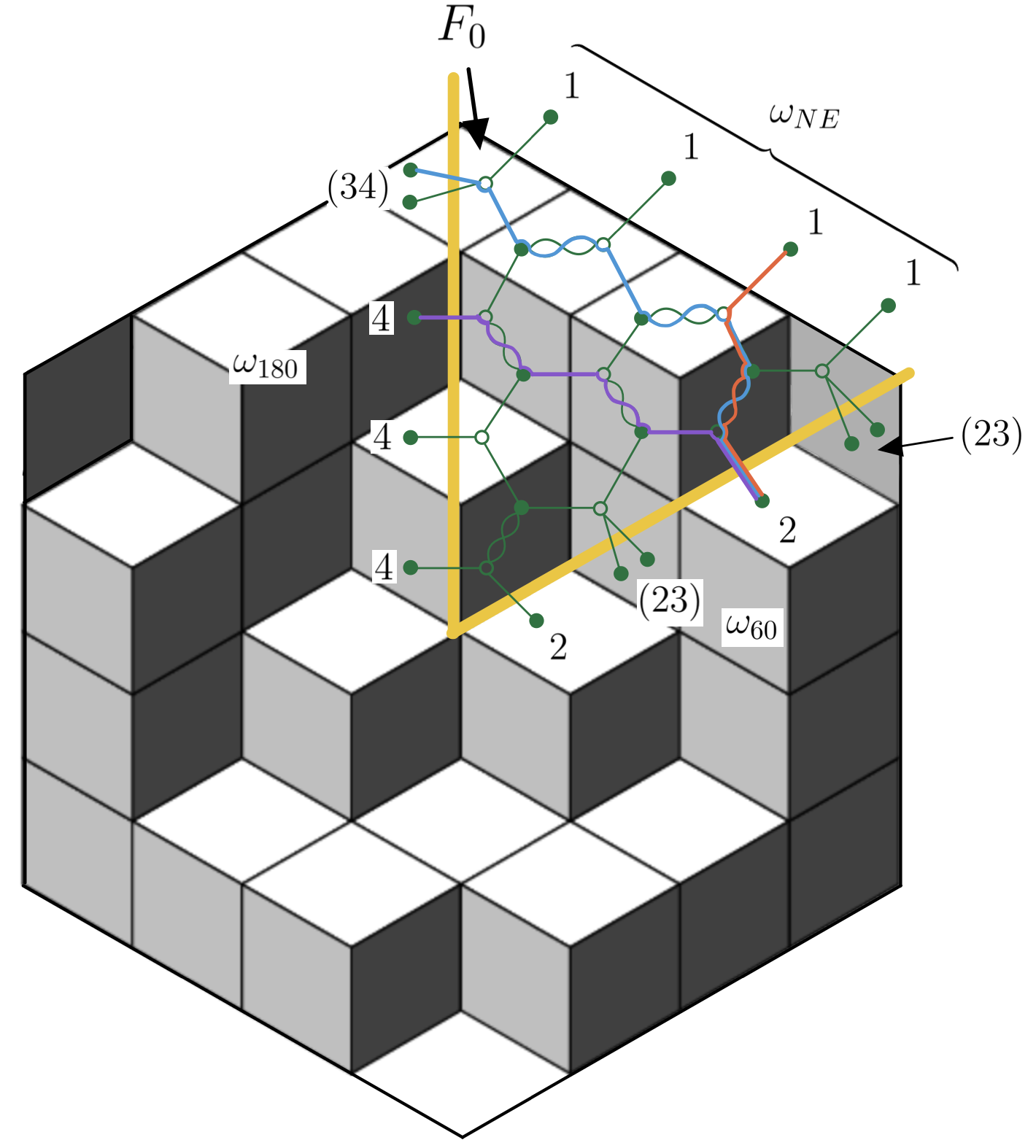}
    \includegraphics[width=0.32\linewidth]{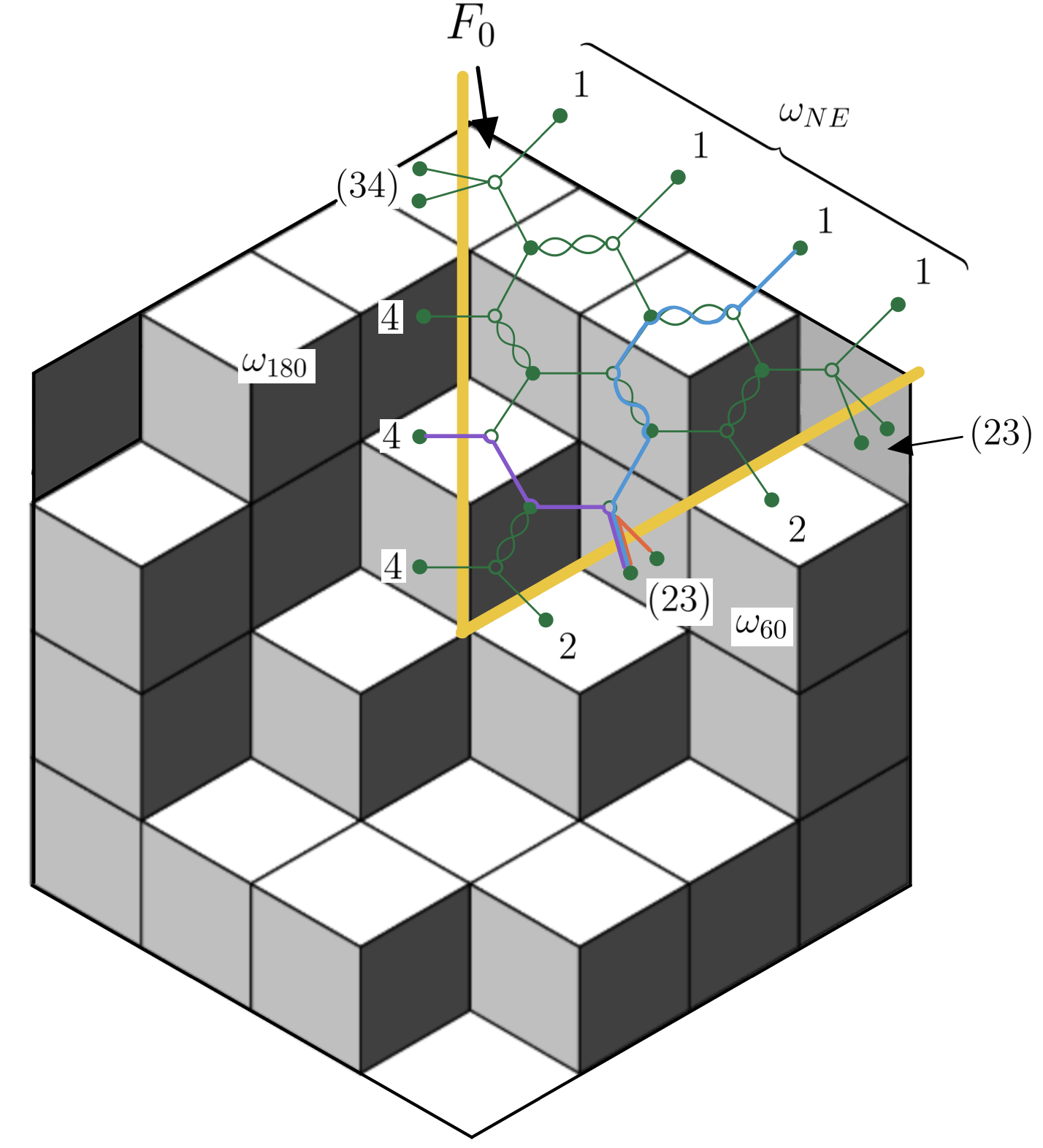}
    \includegraphics[width=0.32\linewidth]{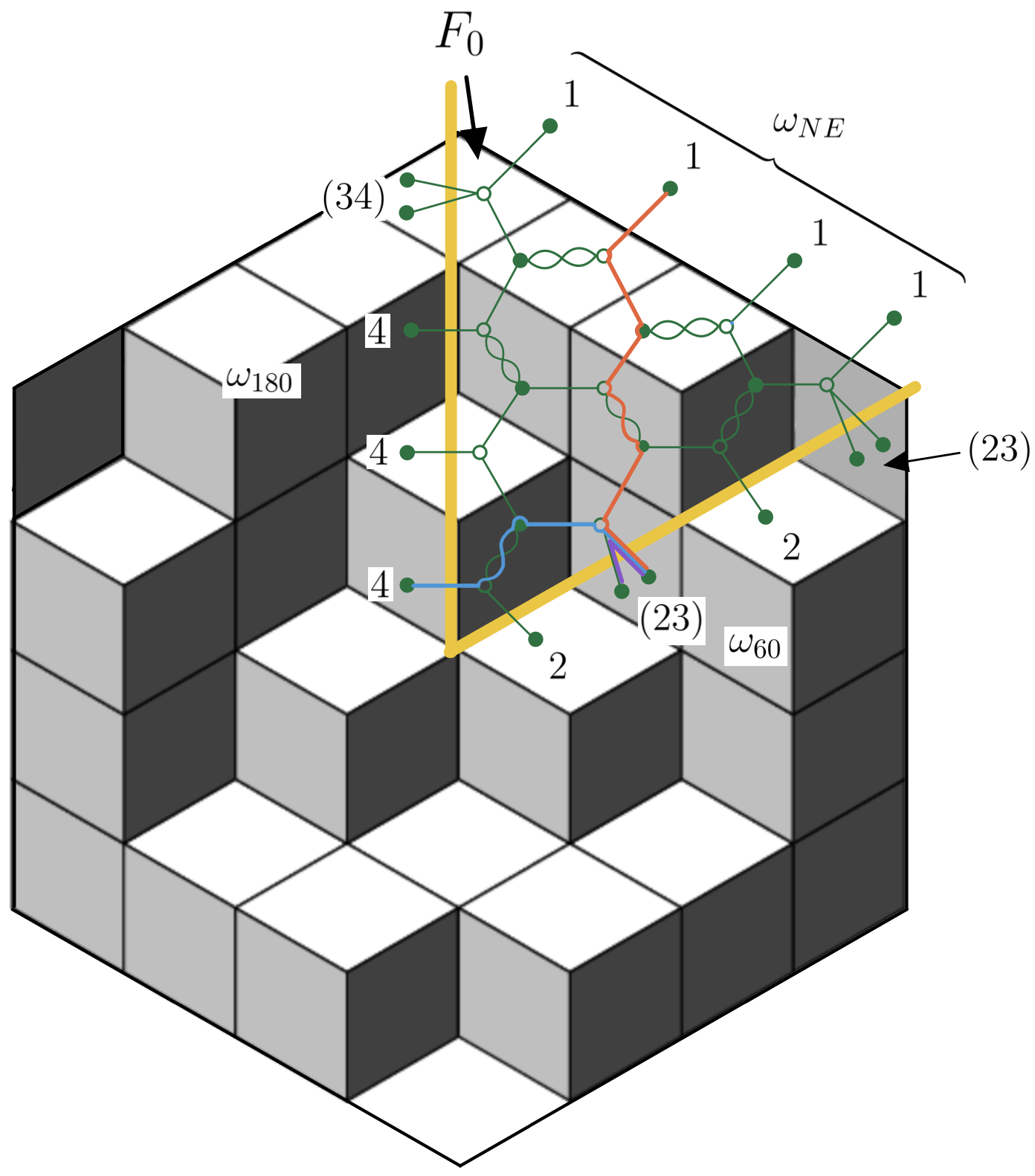}
    \caption{
        A plane partition in $\tspp(4,4,4)$ with the trips drawn within its fundamental domain, as discussed in Example \ref{ex:TSPP-ex for proof}
    }
    \label{fig:TSPP-ex for proof}
    \end{figure}
 \vspace{5mm}
        $\bullet$ \textbf{TSSCPP:} (See Figure \ref{fig:TSSCPP-ex for proof}.) Following the description of the fundamental domain for any $p\in \tsscpp(2d,2d,2d),$ we obtain a web $W$ in bijection with $p$ restricted to this fundamental domain. As described in Section \ref{subsec: A description of fundamental domains}, $W$ has three boundaries each corresponding to a subword that makes up the entire word:
        \[
        \omega_{\tsscpp} = \omega_{\widehat{NE}}\omega_{30}\omega_{\widehat{180}}.
        \]

        The fact that the subword $\omega_{\widehat{NE}} = 1^{d}$ follows from Proposition \ref{prop: restriction preserves words} and by then applying Theorem \ref{thm:general-lattice-words}.
        We now show that $\omega_{30} = \underbrace{2~\barfour~\cdots~\barfour~2}_{2d-1}.$ In any $p\in\tsscpp(2d,2d,2d),$ $\partial_{30}$ intersects only with east faces of cubes in $p.$ Since an hourglass edge lies on an east face, and this web is reduced (in this case, we mean that there are no consecutive hourglass edges and the graph is bipartite), edges on the boundary are only simple edges that alternate between being incident to black and white vertices. The fundamental domain cuts the east face of $p_{d,d}$ such that only a single edge incident to a black vertex lies on the boundary, then we have $2d-1$ letters in $\omega_{30}.$ Now it remains to show the separation labels alternate between $2$ and $4.$ Consider an arbitrary east face on which $\partial_{30}$ lies. Reading clockwise, we have established that there are two edges, $e_1$ and $e_2,$ where $e_1$ is incident to a black boundary vertex and $e_2$ is incident to a white vertex.
        Since the trajectory of $\trip_i$ of $e_1$ is understood for $i=1,2,3$ by the discussion in Proposition \ref{prop: restriction preserves words}, then $\sep(e_1) = 2$ and  $\sep(e_2) = 4.$

        In particular, for both $e_1$ and $e_2,$ $\trip_1$ ends on the NW boundary, $\trip_2$ ends on the E boundary, and $\trip_3$ ends on the NE boundary. Therefore, $\trip_3(e_1)$ is the only trip separating the base face from the face to the right of $e_1$ as you walk from the black boundary vertex of $e_1$ to the white vertex.
        Since $e_2$ is incident to a white boundary vertex, then all of the trips separate the base face to the left of $e_2$ as you walk from the white boundary vertex to the black vertex incident to $e_2.$ This gives us a complete description of $\omega_{30}.$

        We now describe $\omega_{\widehat{180}}.$ As described in the proof for $\omega_{\spp},$ given $p\in\spp(a,a,c)$ the symmetric boundary is given by any permutation of the elements of $\{4^c,(34)^a\}.$ Since any $p\in\tsscpp(2d,2d,2d)$ has reflective symmetry along any diagonal, then by Proposition \ref{prop: restriction preserves words}, the letters for the separation labels are $\{4^{d-1},(34)^d\}.$
        Since $p\in\tsscpp(2d,2d,2d),$ then $p\in\spp(2d,2d,2d).$ By restricting from the fundamental domain of $\spp$ to the fundamental domain of $\tsscpp,$ we are arbitrarily losing an edge which would correspond to prepending a $4$ on to $\omega_{\widehat{180}}.$ From \cite{di2004refined}, we observe that there is a bijection between $\tsscpp$ and non-intersecting lattice paths. The end points of these paths end on $\partial_{\widehat{180}}$ and give us the possible orderings for the split edges (which have separation labels (34)) and simple edges (which have separation label 4). Therefore, the word $4\,\omega_{\widehat{180}}$ must be Yamanouchi in the letters $4$ and $(34)$, that is, every prefix contains at least as many $4$s as $(34)$s. Moreover, this shows that $\omega_{\tsscpp}$ is a Catalan object.

        \begin{example}
        \label{ex:TSSCPP-ex for proof}
        In Figure \ref{fig:TSSCPP-ex for proof} we have an example of a plane partition in $\tsscpp(4,4,4)$ restricted to its fundamental domain with its corresponding web with base face $F_0.$ Each image has trips $1,2,3$ given by three different colors for a type of edge incident to each boundary. We have previously shown the trips on boundaries $\partial_{NE}$ and $\b_{\widehat{180}}.$ Thus we only show the trips on the two types of edges found on the boundary given by a restriction to the fundamental domain of totally symmetric self-complementary plane partitions, that is, $\b_{30}.$
        The left figure shows the trips for an edge on $\b_{30}$ incident to a black boundary vertex, which has separation label $2.$ The right figure shows the trips for an edge on $\b_{30}$ incident to a white boundary vertex, which gives the letter $\barfour.$
        \end{example}
    \begin{figure}[!htbp]
    \centering
    \includegraphics[width=0.25\linewidth]{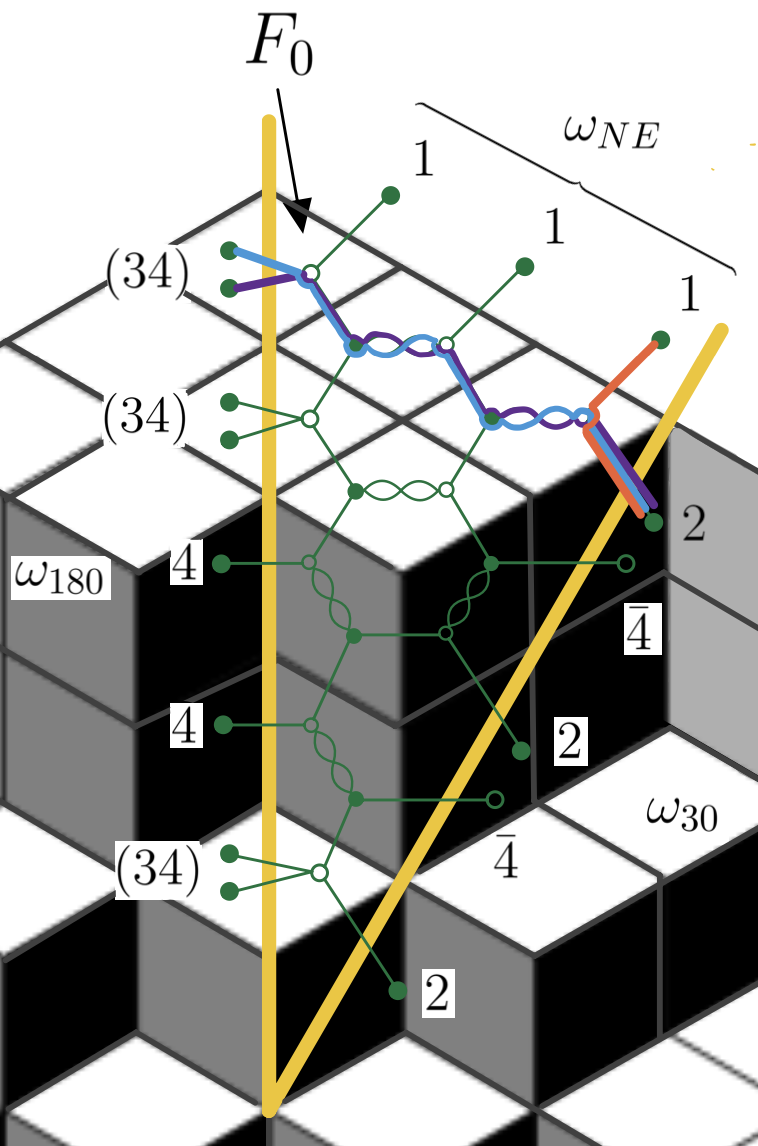}
    \hspace{0.5cm}
    \includegraphics[width=0.25\linewidth]{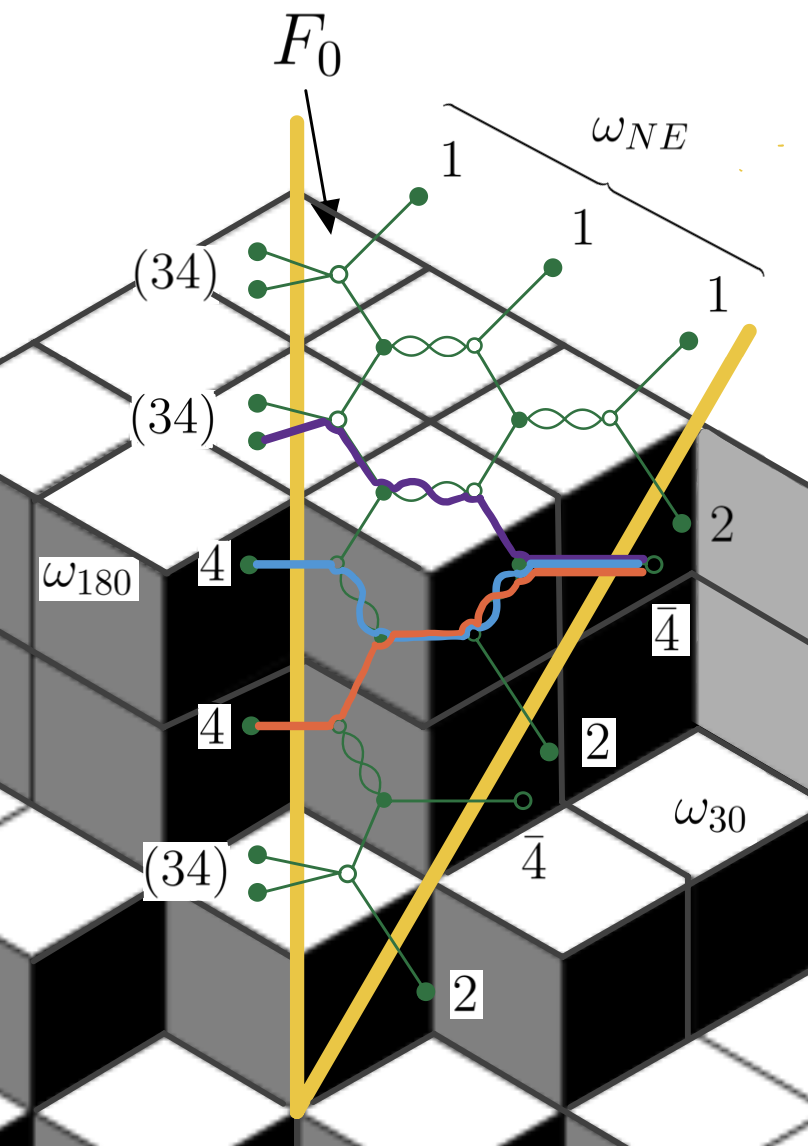}
    \caption{
        A plane partition in $\tsscpp(4,4,4)$ with the trips drawn within its fundamental domain, as discussed in Example \ref{ex:TSSCPP-ex for proof}
    }
    \label{fig:TSSCPP-ex for proof}
    \end{figure}

\end{proof}

\begin{cor}
\label{thm:num-lattice-words-sym}
The following table enumerates the lattice words of each symmetry class from Theorem \ref{thm: lattice-words-symmetry-classes}.
\begin{center}
    \begin{tabular}{|l|l| }
    \hline
    \rowcolor[HTML]{E6E6E6} Symmetry Class & Number of unique lattice words \\
    \hline
    $\spp(a,a,c)$       & $\binom{a+c}{a}$     \\
    $\cspp(a,a,a)$      & $2^a$\\
    $\tspp(a,a,a)$      & $C_{a+1} = \frac{1}{a+2}\binom{2a+2}{a+1}$ \\
    $\tsscpp(2d,2d,2d)$ & $C_d = \frac{1}{d+1}\binom{2d}{d}$ \\
    \hline
    \end{tabular}
\end{center}
\end{cor}

\begin{proof} We prove the number of unique lattice words for an hourglass plabic graph corresponding to plane partitions for each of the symmetry classes listed in Theorem \ref{thm: lattice-words-symmetry-classes}.
    \begin{itemize}
        \item $\spp(a,a,c):$ In $\omega_{\spp}$ the only letters not fixed in the subword are given by any orderings of $c$  $4$s and $a$  $(34)$s. If we consider the $(34)$s as a single letter, then we count the number of ways to organize $c$  $4$s and $a$  $(34)$s in $a+c$  slots. This is given by the counting formula $\binom{a+c}{a}.$

        \item $\cspp(a,a,a):$ In $\omega_{\cspp},$ the only letters not fixed come from the two subwords $\{(\overline{31})^m,\overline{1}^{a-m}\}$ and $\{4^{a-m},(34)^m\}.$ The positions of the letters in the second subword are completely determined by the first subword, so the number of unique lattice words is completely determined by the number of words satisfying the requirements for $\{(\overline{31})^m,\overline{1}^{a-m}\}.$ Since for each $0\leq m \leq a$ we may freely choose the placement of the $m$  $(\overline{31})$s and the $a-m$  $\overline{1}$s then the total number is the simple counting formula
        \[
        \sum_{m=0}^a\binom{a}{m}=2^a.
        \]
        \item $\tspp(a,a,a):$ In $\omega_{\tspp},$ the only letters not fixed come from the two subwords $\{(23)^{a-m},2^m\}$ and $\{4^{a-m},(34)^m\}.$
        Fix $k=a-m.$ Reading $\{(23)^{a-m},2^m\}$ from right to left, let $X\subseteq[a]$ be the set of positions of the $(23)$s, and reading $\{4^{a-m},(34)^m\}$ from left to right, let $Y\subseteq[a]$ be the set of positions of the $4$s. Then $|X|=|Y|=k,$ and the requirement in Theorem \ref{thm: lattice-words-symmetry-classes} says that $|X\cap[i]|\leq|Y\cap[i]|$ for all $1\leq i\leq a.$ Equivalently, if $x_1<\cdots<x_k$ and $y_1<\cdots<y_k$ are the elements of $X$ and $Y,$ then $y_j\leq x_j$ for all $j.$ Viewing $X$ and $Y$ as lattice paths with up-steps in these positions, the Lindstr\"om--Gessel--Viennot lemma \cite{gessel1985binomial} gives that the number of such pairs is 
        \[
        \binom{a}{k}^2-\binom{a}{k-1}\binom{a}{k+1}=\frac{1}{a+1}\binom{a+1}{k}\binom{a+1}{k+1},
        \]
        which is a Narayana number. Summing over $0\leq k\leq a$ gives the Catalan number $C_{a+1}.$
       
        \item $\tsscpp(2d,2d,2d):$  In $\omega_{\tsscpp},$ the only letters not fixed come from the subword $\{4^{d-1},(34)^{d-1}\}.$ So we need only count the number of subwords satisfying the requirement on this word. This requirement, i.e., that the subword is Yamanouchi in $4$s and $(34)$s if we include an extra $4$ at the beginning and $(34)$ at the end, is a well known Catalan object, giving us the $d$-th Catalan number.
    \end{itemize}
\end{proof}

\section{A combinatorial map between invariant spaces}
\label{sec:A combinatorial map between invariant spaces}

In this section we give our second main theorem, Theorem \ref{thm:big invariant to small invariant}.
We first use the lattice words described in Theorem \ref{thm: lattice-words-symmetry-classes} to give a combinatorial algorithm (Algorithm \ref{alg:lattice-word-to-smaller web}) that takes as an input a $U_q(\sl_4)$-web corresponding to a plane partition in a certain symmetry class and outputs a $U_q(\slr)$-web for $r=2$ or $r=3.$

The symmetry classes for which the algorithm can be applied are (1) symmetric, (2) totally symmetric, (3) totally symmetric self-complementary.
In cases (1) and (3), the algorithm results in a $U_q(\sl_2)$-web, which is a decorated non-crossing matching.
A \emph{(perfect) non-crossing matching} is a graph of $n$ points on a plane with edge segments between  vertices so that every vertex has valance one and no two edges cross, i.e., the graph is planar.
In case (2), the algorithm results in a $U_q(\sl_3)$-web. The algorithm relies heavily on our ability to classify the lattice words of symmetry classes as we have done here, which was only possible due to the recent work in \cite{gaetz2023rotation}.

The implications of this algorithm are that we can construct a map from certain invariants in
$\mathrm{Hom}_{U_q(\sl_4)}\left(\bigwedge\nolimits_q^{\underline{c}(\omega)}V_q,\C(q)\right)$
to invariants in
$\mathrm{Hom}_{U_q(\slr)}\left(\bigwedge\nolimits_q^{\underline{c}(\hatomega)}V_q,\C(q)\right)$
for $r=2$ or $r=3,$ where $\underline{c}(\omega)=(c_1,\ldots,c_n)$ with $c_i$ equal to $1$ if $\omega_i$ is positive and $-1$ if $\omega_i$ is negative. We leave further exploration of the representation theory implications to future work.

\begin{algorithm}[Lattice words of symmetry classes to $U_q(\slr)$-webs]~
\label{alg:lattice-word-to-smaller web}

    \textbf{Input:} A plane partition $p$ which is either in $\spp(a,a,c),\tspp(a,a,a),$ \newline
    or $\tsscpp(2d,2d,2d).$
    \begin{enumerate}
        \item Consider the hourglass plabic graph $W$ constructed using the fundamental domain of the symmetry class of $p$. Let $\omega$ denote the boundary word of $W$, as described in Theorem \ref{thm: lattice-words-symmetry-classes}.
        \item Let $r=2$ if $p$ is in $\spp(a,a,c)$ or $\tsscpp(2d,2d,2d),$ and let $r=3$ if $p$ is in $\tspp(a,a,a).$ Remove the \emph{fixed letters} from $\omega,$ by which we mean all letters in $\{\pm 1,\ldots,\pm(4-r)\}$ and, in the case of $\tsscpp,$ also the final letter $(34).$ These letters are the same in any two lattice words $\omega,\omega'$ corresponding to $W,W'$ from the same symmetry class. The letters $\barfour$ are not removed.
        \item Replace each remaining letter $x$ by the letter with the same sign and with absolute value $|x|-(4-r),$ giving a word of $1$s, $2$s, and $\bartwo$s in the case of SPP and TSSCPP. In the case of TSPP, we obtain a word in letters $1,2,$ and $3.$ This gives us a new lattice word, $\hat{\omega}$ that is in bijection with an oscillating tableau of height $r$ for $r=2$ or $r=3,$ depending on the symmetry class of $p$.
        \item When $r=2$ use the growth rules from Definition \ref{def:growth-rules-basic}. When $r=3,$ use the growth rules from \cite{kuperberg,petersen2009promotion}.
    \end{enumerate}
    \textbf{Output:} A $U_q(\slr)$-web, $\widehat{W}.$  In the case of $\spp$ and $\tsscpp,$ $r=2$ and the output is a non-crossing matching with special marked edges. In the case of a totally-symmetric plane partition, $r=3$ and the output is a non-elliptic $U_q(\sl_3)$-web.
\end{algorithm}

We find Algorithm \ref{alg:lattice-word-to-smaller web} to be of particular interest not only due to its simplicity, but also because it cannot be applied to every symmetry class.
For example, this algorithm cannot be applied to a plane partition which is cyclically symmetric.
Algorithm \ref{alg:lattice-word-to-smaller web} is really a story about reducing oscillating tableaux within a symmetry class.
In the cases that $p$ is a $\spp, \tspp,$ or a $\tsscpp,$ the fixed letters in the lattice word corresponding to the web given by $p$ make up the first row (for $\tspp$) or the first two rows (for $\spp$ and $\tsscpp$) of the associated oscillating tableau.
So we can ultimately delete those rows from the tableau.
In the case $p$ is a $\cspp,$ then the lattice word has variance within all of the rows.
So removing any of them in a reduction such as the one given by Algorithm \ref{alg:lattice-word-to-smaller web} would result in a catastrophic loss of information.

In Section \ref{sec:ex-tab-to-NCM} we give the lattice words for the benzene equivalence classes for the webs corresponding to $\tsscpp(6,6,6)$ and demonstrate Algorithm \ref{alg:lattice-word-to-smaller web} for each word.

Before proceeding to the main theorem of this section, we want to give an explanation of why Algorithm \ref{alg:lattice-word-to-smaller web} gives us the desired output. Step $1$ holds easily. Steps $2$ and $3$ we justify now.

The subword $\omega_{NE}$ coming from the boundary $\partial_{NE}$ is $1^a$ in the case of $\spp$ and $\tspp,$ and in the case of $\tsscpp$ it is $1^d.$
By Theorem \ref{thm: lattice-words-symmetry-classes} there are no other $1$s in the lattice words.
Thus, the process of removing the $1$s and adjusting the other letters in $\omega$ by $-1,$ corresponds to removing a row of boxes in the oscillating tableau $T(\omega)$ and adjusting the entries in the tableau accordingly.
Therefore, the modified word remains an oscillating tableau by removing the $1$s from $\omega.$
In the case of $\spp$ and $\tsscpp,$ a similar argument holds for removing the $2$s from the lattice word corresponding to a $\spp$ or $\tsscpp.$
Within the $\tspp,$ the letters $2$ vary in their placement.
So removing them would result in a loss of information from the original oscillating tableau.
Since oscillating tableaux in a rectangle with $r$ rows are in bijection with $U_q(\sl_r)$-webs, it follows that as output of the algorithm we obtain, in the case of $\spp$ and $\tsscpp,$ $U_q(\sl_2)$-webs, whereas in the case of $\tspp,$ we obtain $U_q(\sl_3)$-webs.

\begin{figure}[!htbp]
    \centering
    \begin{center}
    \begin{tabular}{|l|l|l| }
    \hline
    \rowcolor[HTML]{E6E6E6} Symmetry Class & $\omega$ after Step 2 & $\hatomega$   \\
    \hline
    $\spp(a,a,c)$       & $\barfour^c\{4^c,(34)^a\}$                           & $\overline{2}^c\{2^c,(12)^a\}$\\
    $\tspp(a,a,a)$      & $\{(23)^{a-m},2^m\}~\{4^{a-m},(34)^m\}$  & $\{(12)^{a-m},1^m\}~\{3^{a-m},(23)^m\}$ \\
    $\tsscpp(2d,2d,2d)$ & $\barfour^{d-1}\{4^{d-1},(34)^{d-1}\}$                 & $\overline{2}^{d-1}\{2^{d-1},(12)^{d-1}\}$\\
    \hline
    \end{tabular}
\end{center}
    \caption{This chart gives the lattice word, $\hatomega,$ associated with the output of Algorithm \ref{alg:lattice-word-to-smaller web}, according to the given symmetry class. Per Theorem \ref{thm: lattice-words-symmetry-classes}, there are specific requirements on the subwords, which affect the ordering of the letters in $\hatomega.$ We refer the reader to the referenced theorem for those details.}
    \label{fig:lattice-word-to-omega-hat}
\end{figure}

\begin{thm}
\label{thm:big invariant to small invariant}
    Let $S$ be the subspace of
        \begin{equation*}
            \mathrm{Hom}_{U_q(\sl_4)}
            \left(
            \bigwedge\nolimits_q^{\underline{c}(\omega)}V_q,\C(q)
            \right)
        \end{equation*}
    spanned by the invariants $[W]_q,$ where $W$ is the top representative of the move equivalence class with lattice word $\omega$ corresponding to a plane partition in one fixed symmetry class, which is one of $\spp(a,a,c),$ $\tspp(a,a,a),$ or $\tsscpp(2d,2d,2d).$ Then Algorithm \ref{alg:lattice-word-to-smaller web} defines an injective linear map
        \begin{equation*}
            S
            \to
            \mathrm{Hom}_{U_q(\slr)}\left(
            \bigwedge\nolimits_q^{\underline{c}(\hat{\omega})}V_q,\C(q)
            \right)
        \end{equation*}
    given by $[W]_q\mapsto [\widehat{W}]_q,$ where $\widehat{W}$ is the web with lattice word $\hat{\omega}$ resulting from the previous algorithm. When the plane partition is symmetric or totally symmetric self-complementary, then $r=2,$ whereas if the plane partition is totally symmetric, then $r=3.$
\end{thm}

\begin{proof}
    Let $p$ be a plane partition which is either in $\spp(a,a,c),\tspp(a,a,a),$ or \newline $\tsscpp(2d,2d,2d).$
    If $W$ is the hourglass plabic graph in bijection with the fundamental domain of the desired symmetry class of $p$, then $W$ has a lattice word $\omega$ whose form is given by Theorem \ref{thm: lattice-words-symmetry-classes}.
    If $p\in\spp(a,a,c),$ then by removing the fixed letters in $\omega$, the resulting word is $\barfour^c~\{4^c,(34)^a\}.$
    If $p\in\tsscpp(2d,2d,2d),$ then the resulting word is $\barfour^{d-1}~\{4^{d-1},(34)^{d-1}\}.$
    The word also maintains the properties required of it in Theorem \ref{thm: lattice-words-symmetry-classes}.
    By subtracting $2$ from the absolute value of each of these letters, we obtain the word
    $\bartwo^c~\{2^c,(12)^a\}$ in the case of $\spp$ and $\bartwo^{d-1}~\{2^{d-1},(12)^{d-1}\}$ in the case of $\tsscpp.$
    This process is well-defined, as it is the same as removing the first two rows from the oscillating tableau with which $\omega$ is in bijection.
    The information in those rows is the same for any other web in bijection with a plane partition in the desired symmetry class, and is thereby redundant. (These are what we call \emph{fixed letters}.)
    In the case $p\in\tsscpp(2d,2d,2d),$ by removing all fixed letters, we are also removing the furthest box to the right in rows $3$ and $4,$ since that information is also redundant.

    One may think of these fixed $\barfour$s and singleton $4$s in the original words as paired, since in $W$ they are connected by $\trip_2$ (Theorem \ref{thm: lattice-words-symmetry-classes}).

    In the case that $p\in\tspp(a,a,a),$ according to Theorem \ref{thm: lattice-words-symmetry-classes}, the lattice word is \newline $1^a\{(23)^{a-m},2^m\}~\{4^{a-m},(34)^m\},$ with certain requirements upon ordering of the letters in the two subwords $\{\cdots\}$ given in the statement of the theorem.
    Removing the fixed letters and adjusting them results in the word $\{(12)^{a-m},1^m\}~\{3^{a-m},(23)^m\}.$
    This process corresponds to removing the first row, for the same reason as in the previous two cases.

    In all the above cases, we will call the new adjusted words $\hatomega.$
    Since the fixed letters that were removed are exactly the letters in rows $1,\ldots,4-r$ (together with the final $(34)$ in the case of $\tsscpp,$ which removes one box from each of rows $3$ and $4$), the Yamanouchi conditions between the remaining rows are unchanged and the final rows remain equal. Thus the word is still a rectangular lattice word. This holds for both a two row or three row oscillating tableau following Definition \ref{def:oscillating-tab}.
    From here, the map from a lattice word to a $U_q(\sl_2)$- or $U_q(\sl_3)$-web is known, (Definition \ref{def:growth-rules-basic} and \cite{kuperberg}). We call the new web $\widehat{W}.$ In the cases where $p\in\spp(a,a,c)$ or $p\in\tsscpp(2d,2d,2d)$ the growth Rules 1 and 2 of Definition \ref{def:growth-rules-basic}  are the only used rules. Rule $1$ ensures that the pairs $(12)$ coming from split edges in $W$ are connected by an edge marked with a white vertex, while Rule $2$ ensures that the $\bartwo$s and singleton $2$s are connected by a simple edge. Since the $(12)$ pairs are always directly next to each other in $\hat{\omega},$ their edges in $\widehat{W}$ never cross. Furthermore, $\hat{\omega}$ begins with $\bartwo$s and so applying growth rule $2$ causes no edges to cross. This leaves us with a reduced $U_q(\sl_2)$-web which is a marked non-crossing matching.

    We now explain why this algorithm gives a map on invariants.
    Since $\omega$ is a rectangular Yamanouchi word, by \cite[Theorem~B]{gaetz2023rotation} it is the boundary word of a unique move equivalence class of fully reduced hourglass plabic graphs, and we take $W$ to be its top representative.
    The $U_q(\sl_4)$-webs $W$ and $W'$ associated to plane partitions $p$ and $p'$ in the same symmetry class $\spp(a,a,c),\tspp(a,a,a),$ or $\tsscpp(2d,2d,2d),$ but with different lattice words $\omega,\omega'$ give distinct invariants $[W]_q$ and $[W']_q.$ Although $[W]_q$ and $[W']_q$ correspond to different lattice words, they still have the same list of edge weights for those edges incident to boundary vertices. Call this list $\underline{c}(\omega) = (c_1,\ldots,c_n),$ which can be obtained by checking the sign of each value of $\omega.$ So then, $[W]_q,[W']_q$ are elements of
    $
    \mathrm{Hom}_{U_q(\sl_4)}
    \left(
    \bigwedge\nolimits_q^{\underline{c}(\omega)}V_q,\C(q)
    \right).
    $

    Applying the algorithm results in distinct lattice words $\hatomega$ and $\widehat{\omega'}$ which give distinct invariants, $[\widehat{W}]_q$ and $[\widehat{W'}]_q,$ respectively. The algorithm does not change the relative order of the remaining letters, nor does it change the sign. Thus, the list of edge weights $\underline{c}(\hatomega) = \underline{c}(\widehat{\omega}').$ Thus, $[\widehat{W}]_q$ and $[\widehat{W'}]_q$ are in
    $
    \mathrm{Hom}_{U_q(\slr)}\left(
    \bigwedge\nolimits_q^{\underline{c}(\hat{\omega})}V_q,\C(q)
    \right).
    $
    By Theorem \ref{thm:web-basis-GPPSS}, the invariants $[W]_q$ for distinct lattice words $\omega$ are distinct elements of a basis, so they are linearly independent and the assignment $[W]_q\mapsto[\widehat{W}]_q$ extends uniquely to a linear map on their span $S.$ The algorithm is injective on lattice words, since two lattice words in the same symmetry class differ only in letters that are not removed. Since the webs $\widehat{W}$ for distinct $\hatomega$ are distinct elements of Kuperberg's web basis \cite{kuperberg}, this linear map is injective.
    Therefore, this map exists and is well-defined for all three symmetry classes.
 \end{proof}

\begin{remark}
    The map in Theorem \ref{thm:big invariant to small invariant} is defined on the subspace $S$ and depends on the web bases. We do not claim that it extends canonically to all of $\mathrm{Hom}_{U_q(\sl_4)}\left(\bigwedge\nolimits_q^{\underline{c}(\omega)}V_q,\C(q)\right)$ or that it comes from a map between quantum groups.
\end{remark}

\subsection{Example of Theorem \texorpdfstring{\ref{thm:big invariant to small invariant}}{5.2}: \texorpdfstring{$U_q(\sl_4)$}{Uq(sl4)} to \texorpdfstring{$U_q(\sl_2)$}{Uq(sl2)}}
\label{sec:ex-tab-to-NCM}
We demonstrate the process outlined in the proof of Theorem \ref{thm:big invariant to small invariant} for all five unique lattice words obtained from the seven webs in bijection with the seven plane partitions in $\tsscpp(6,6,6).$
A visualization of this can be seen explicitly in Figure \ref{fig:6-ncms from sl4 to sl2 map}. For each word we give the corresponding lattice word $\omega.$
Below $\omega$ is the  oscillating tableau.
The letters in teal are those that actually vary between tableaux; those in black remain the same.
Underneath the  oscillating tableau on $16$ letters is the oscillating tableau obtained from Algorithm \ref{alg:lattice-word-to-smaller web}, that is, by removing the fixed letters and shifting the remaining letters. Underneath this new shifted  oscillating tableau is the $U_q(\sl_2)$-web obtained by using the growth rules. This is a marked perfect non-crossing matching.

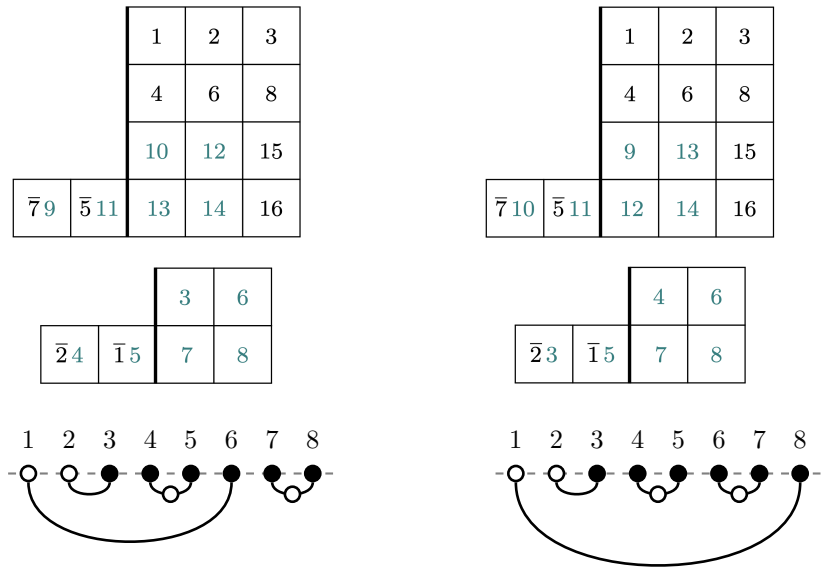
\begin{figure}[htbp]
\begin{center}

\begin{tikzpicture}[scale=0.9]
\node at (-6,2.5) {\tiny $\omega_1~=~1~1~1~2~\overline{4}~2~\overline{4}~2~\textcolor{teal}{4~(34)~4~(34)}~(34)$};
\node (pic) at (-6,0) {\includegraphics[scale=0.2]{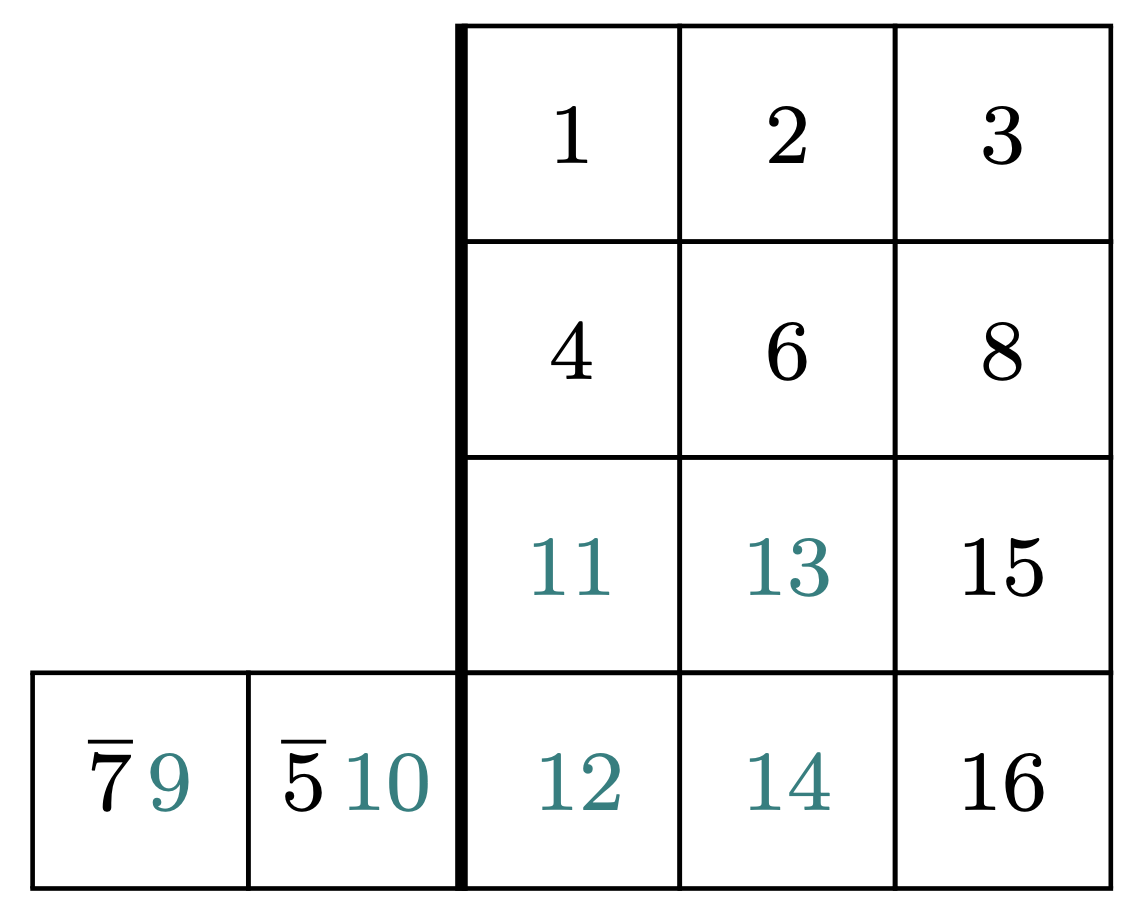}};
\node (pic) at (-6,-3) {\includegraphics[scale=0.2]{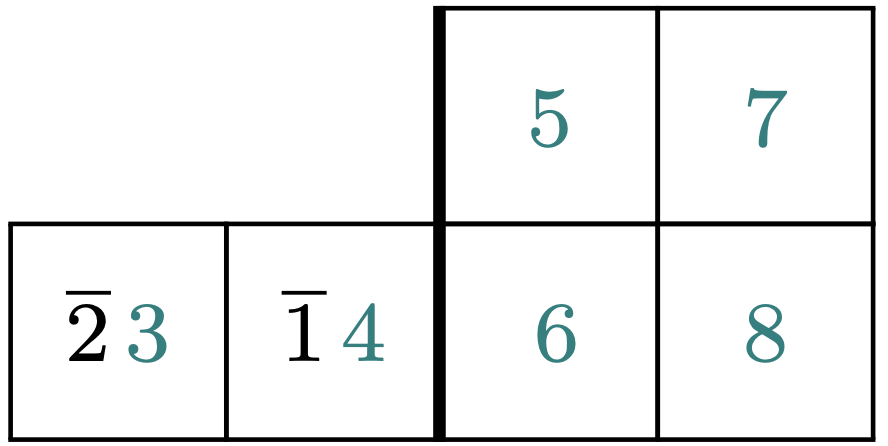}};
\begin{scope}[scale=0.6, line width=1pt,shift={(-14,-9)}]
  \draw[dashed, gray, line width=0.8pt] (0.5,0) -- (8.5,0);

  \foreach \x in {1,...,8} {
    \node[circle, draw, minimum size=2mm, inner sep=0pt, fill=white] (v\x) at (\x,0) {};
    \node[above=2pt, font=\footnotesize] at (v\x.north) {\x};
  }

  \foreach \x in {3,4,5,6,7,8}
    \filldraw[fill=black, draw=black, line width=1pt] (v\x) circle (1.5mm);

  \draw (v1) to[out=270, in=270] (v6);
  \draw (v2) to[out=270, in=270] (v3);

  \draw[postaction={decorate, decoration={
    markings, mark=at position 0.5 with {\node[fill=white, draw=black, circle, minimum size=2mm, inner sep=0pt, line width=1pt] {};}}}]
    (v4) to[out=270, in=270] (v5);
  \draw[postaction={decorate, decoration={
    markings, mark=at position 0.5 with {\node[fill=white, draw=black, circle, minimum size=2mm, inner sep=0pt, line width=1pt] {};}}}]
    (v7) to[out=270, in=270] (v8);
\end{scope}

\node at (0,2.5) {\tiny $\omega_2~=~1~1~1~2~\overline{4}~2~\overline{4}~2~\textcolor{teal}{4~4~(34)~(34)}~(34)$};
\node (pic) at (0,0) {\includegraphics[scale=0.2]{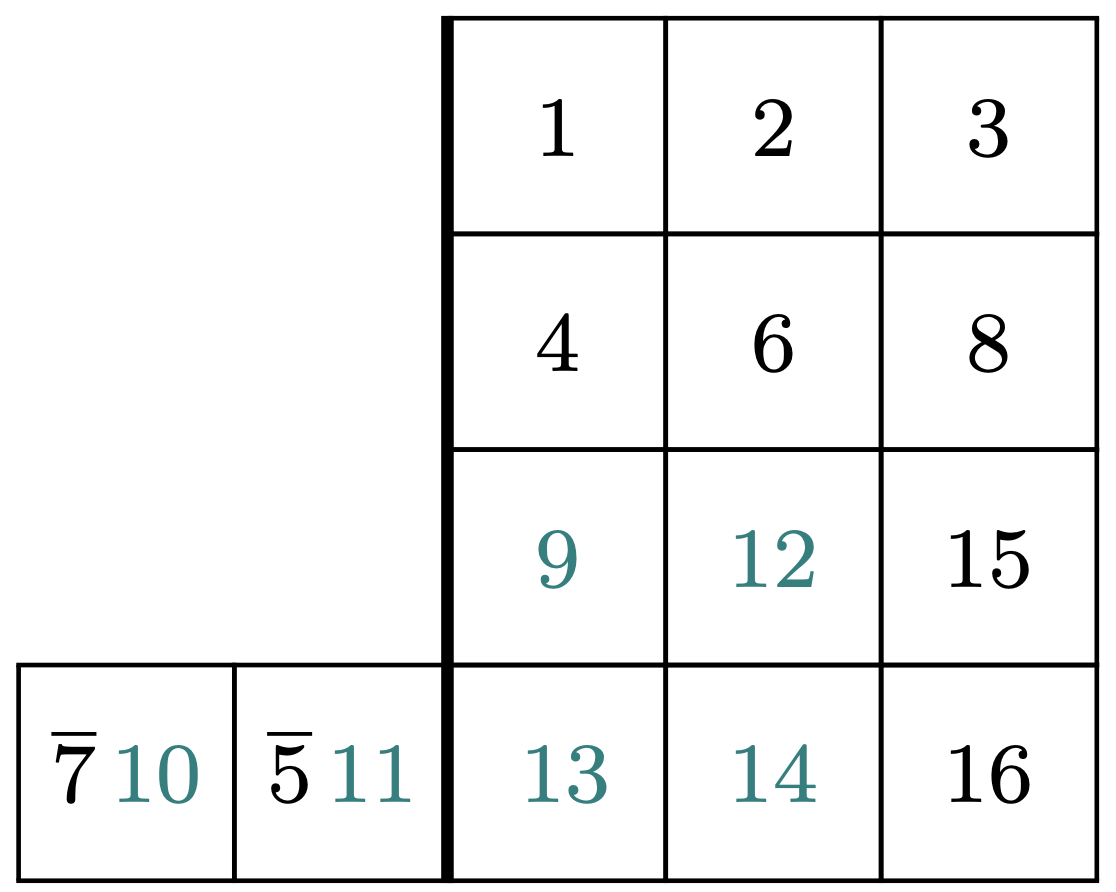}};
\node (pic) at (0,-3) {\includegraphics[scale=0.2]{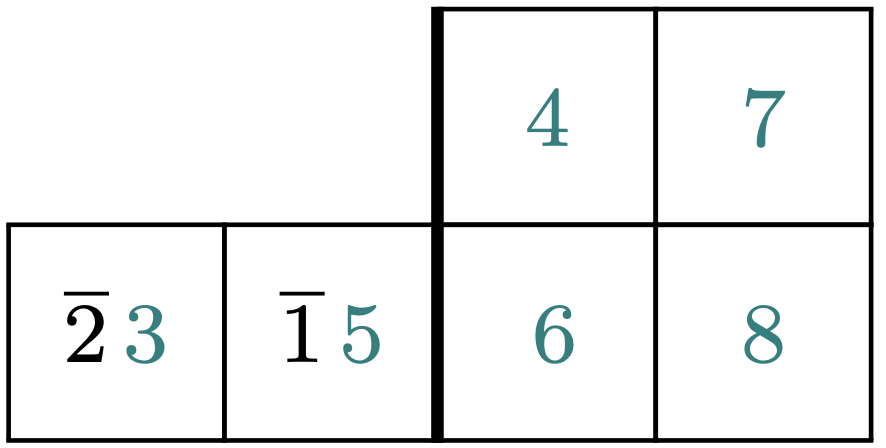}};
\begin{scope}[scale=0.6, line width=1pt, shift={(-4,-9)}]
  \draw[dashed, gray, line width=0.8pt] (0.5,0) -- (8.5,0);

  \foreach \x in {1,...,8} {
    \node[circle, draw, minimum size=2mm, inner sep=0pt, fill=white] (v\x) at (\x,0) {};
    \node[above=2pt, font=\footnotesize] at (v\x.north) {\x};
  }

  \foreach \x in {3,4,5,6,7,8}
    \filldraw[fill=black, draw=black, line width=1pt] (v\x) circle (1.5mm);

  \draw (v1) to[out=270, in=270] (v4);
  \draw (v2) to[out=270, in=270] (v3);

  \draw[postaction={decorate, decoration={
    markings, mark=at position 0.5 with {\node[fill=white, draw=black, circle, minimum size=2mm, inner sep=0pt, line width=1pt] {};}}}]
    (v5) to[out=270, in=270] (v6);

  \draw[postaction={decorate, decoration={
    markings, mark=at position 0.5 with {\node[fill=white, draw=black, circle, minimum size=2mm, inner sep=0pt, line width=1pt] {};}}}]
    (v7) to[out=270, in=270] (v8);
\end{scope}

\node at (6,2.5) {\tiny $\omega_3~=~1~1~1~2~\overline{4}~2~\overline{4}~2~\textcolor{teal}{(34)~4~4~(34)}~(34)$};
\node (pic) at (6,0) {\includegraphics[scale=0.2]{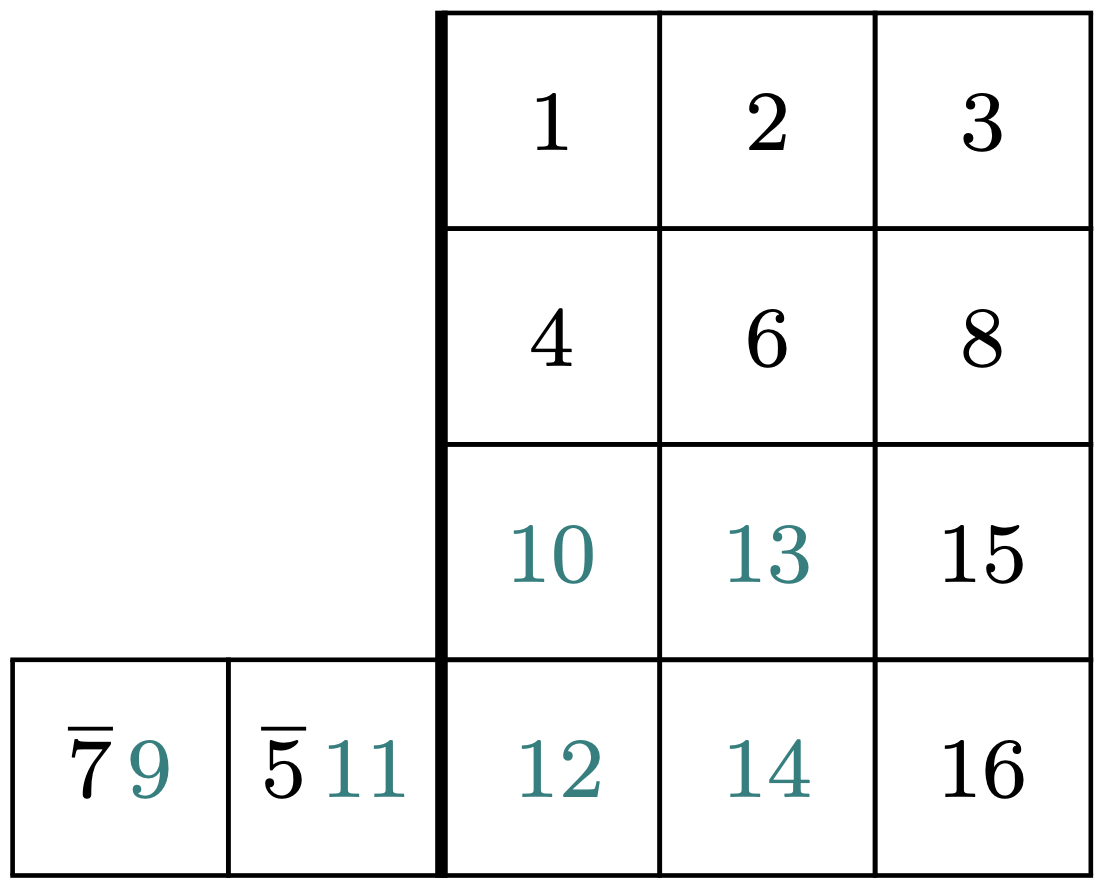}};
\node (pic) at (6,-3) {\includegraphics[scale=0.2]{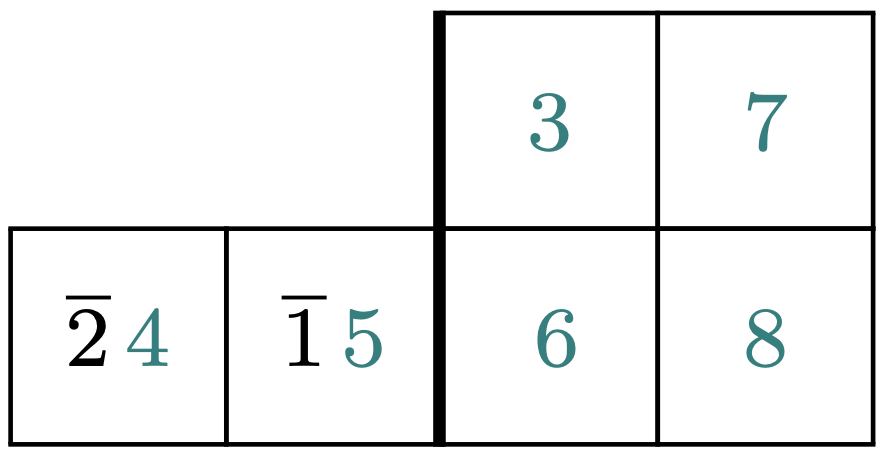}};

\begin{scope}[scale=0.6, line width=1pt, shift={(6,-9)}]
  \draw[dashed, gray, line width=0.8pt] (0.5,0) -- (8.5,0);

  \foreach \x in {1,...,8} {
    \node[circle, draw, minimum size=2mm, inner sep=0pt, fill=white] (v\x) at (\x,0) {};
    \node[above=2pt, font=\footnotesize] at (v\x.north) {\x};
  }

  \foreach \x in {3,4,5,6,7,8}
    \filldraw[fill=black, draw=black, line width=1pt] (v\x) circle (1.5mm);

  \draw (v1) to[out=270, in=270] (v6);
  \draw (v2) to[out=270, in=270] (v5);

  \draw[postaction={decorate, decoration={
    markings, mark=at position 0.5 with {\node[fill=white, draw=black, circle, minimum size=2mm, inner sep=0pt, line width=1pt] {};}}}]
    (v3) to[out=270, in=270] (v4);

  \draw[postaction={decorate, decoration={
    markings, mark=at position 0.5 with {\node[fill=white, draw=black, circle, minimum size=2mm, inner sep=0pt, line width=1pt] {};}}}]
    (v7) to[out=270, in=270] (v8);
\end{scope}

\node at (-3.5,-8.5) {\tiny $\omega_4~=~1~1~1~2~\overline{4}~2~\overline{4}~2~\textcolor{teal}{(34)~4~(34)~4}~(34)$};
\node (pic) at (-3.5,-11) {\includegraphics[scale=0.2]{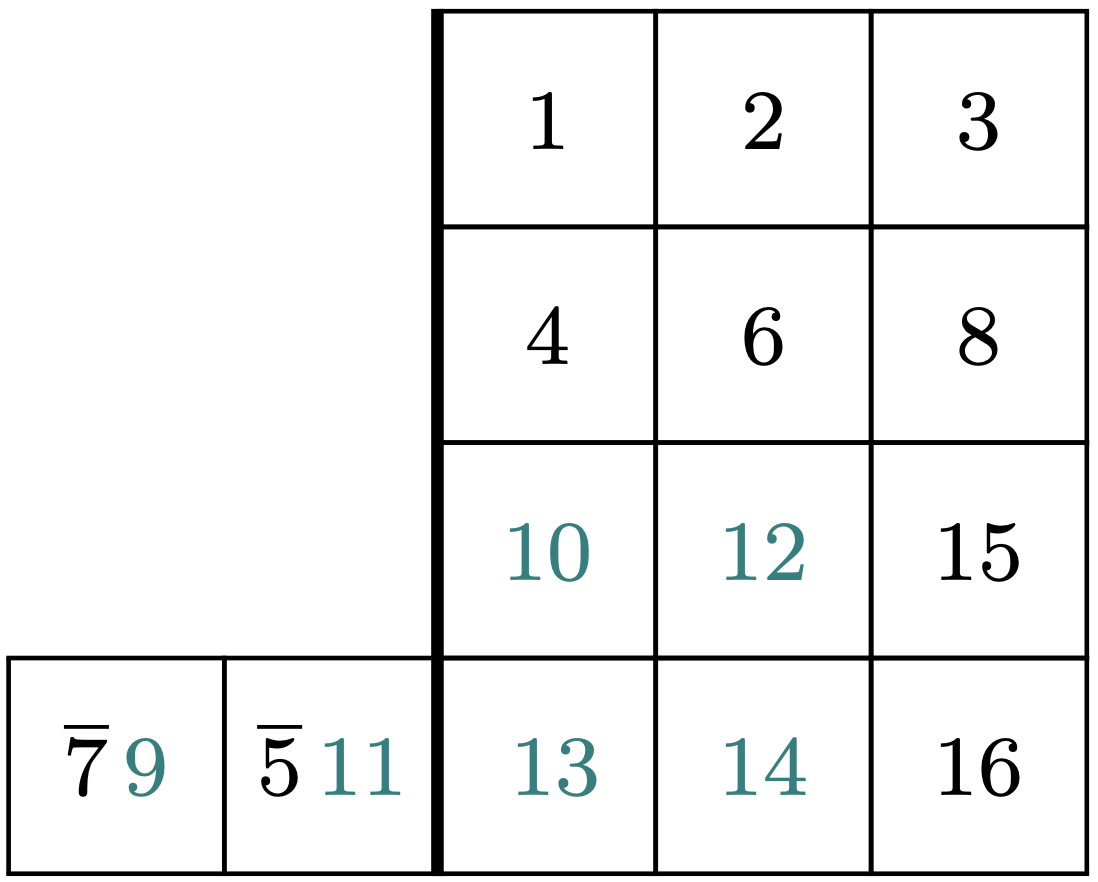}};
\node (pic) at (-3.5,-14) {\includegraphics[scale=0.2]{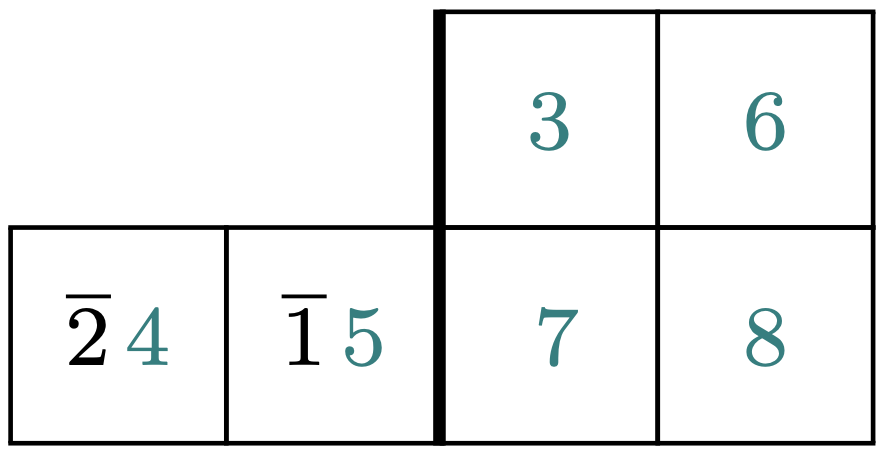}};

\begin{scope}[scale=0.6, line width=1pt, shift={(-10,-27)}]
  \draw[dashed, gray, line width=0.8pt] (0.5,0) -- (8.5,0);

  \foreach \x in {1,...,8} {
    \node[circle, draw, minimum size=2mm, inner sep=0pt, fill=white] (v\x) at (\x,0) {};
    \node[above=2pt, font=\footnotesize] at (v\x.north) {\x};
  }

  \foreach \x in {3,4,5,6,7,8}
    \filldraw[fill=black, draw=black, line width=1pt] (v\x) circle (1.5mm);

  \draw (v1) to[out=270, in=270] (v8);
  \draw (v2) to[out=270, in=270] (v5);

  \draw[postaction={decorate, decoration={
    markings, mark=at position 0.5 with {\node[fill=white, draw=black, circle, minimum size=2mm, inner sep=0pt, line width=1pt] {};}}}]
    (v3) to[out=270, in=270] (v4);

  \draw[postaction={decorate, decoration={
    markings, mark=at position 0.5 with {\node[fill=white, draw=black, circle, minimum size=2mm, inner sep=0pt, line width=1pt] {};}}}]
    (v6) to[out=270, in=270] (v7);
\end{scope}

\node at (3.5,-8.5) {\tiny $\omega_5~=~1~1~1~2~\overline{4}~2~\overline{4}~2~\textcolor{teal}{4~(34)~(34)~4}~(34)$};
\node (pic) at (3.5,-11) {\includegraphics[scale=0.2]{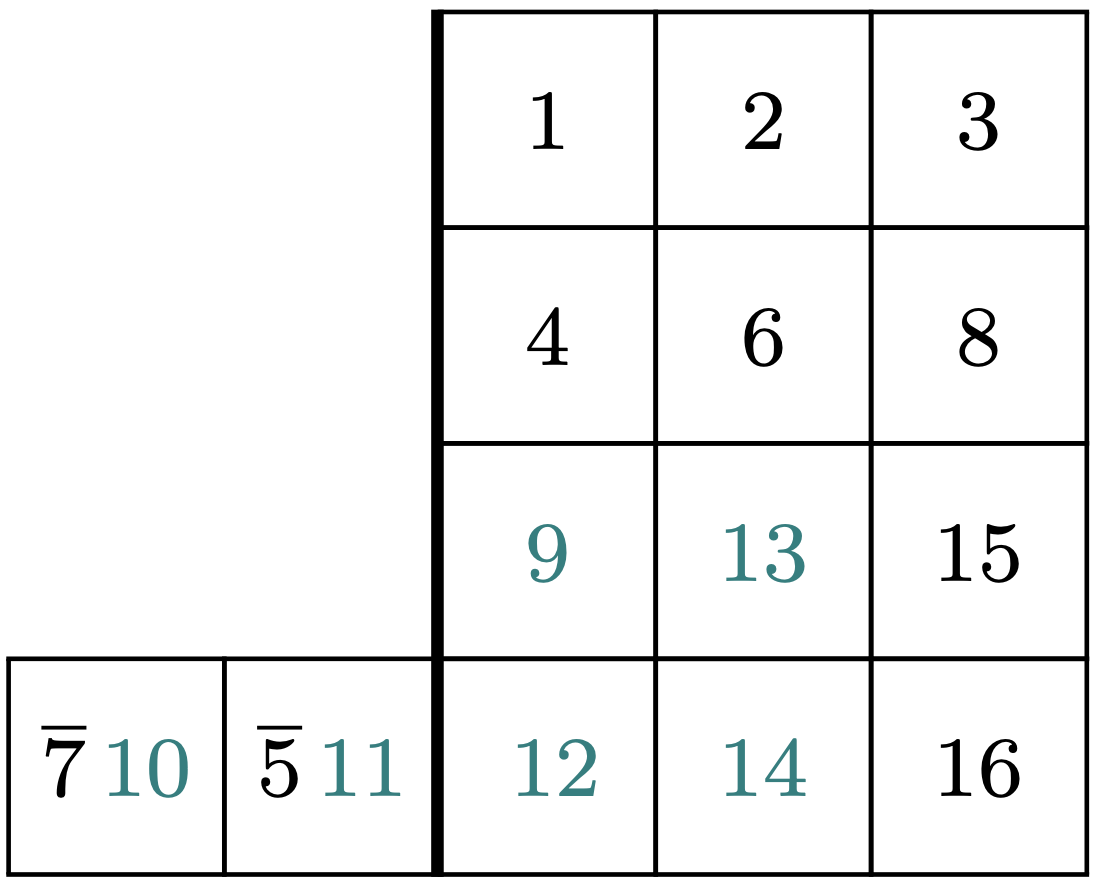}};
\node (pic) at (3.5,-14) {\includegraphics[scale=0.2]{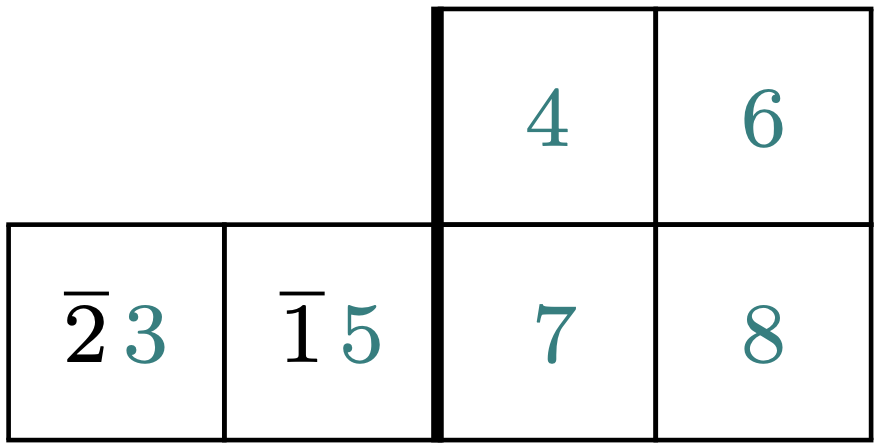}};
\begin{scope}[scale=0.6, line width=1pt, shift={(2,-27)}]
  \draw[dashed, gray, line width=0.8pt] (0.5,0) -- (8.5,0);

  \foreach \x in {1,...,8} {
    \node[circle, draw, minimum size=2mm, inner sep=0pt, fill=white] (v\x) at (\x,0) {};
    \node[above=2pt, font=\footnotesize] at (v\x.north) {\x};
  }

  \foreach \x in {3,4,5,6,7,8}
    \filldraw[fill=black, draw=black, line width=1pt] (v\x) circle (1.5mm);

  \draw (v1) to[out=270, in=270] (v8);
  \draw (v2) to[out=270, in=270] (v3);

  \draw[postaction={decorate, decoration={
    markings, mark=at position 0.5 with {\node[fill=white, draw=black, circle, minimum size=2mm, inner sep=0pt, line width=1pt] {};}}}]
    (v4) to[out=270, in=270] (v5);

  \draw[postaction={decorate, decoration={
    markings, mark=at position 0.5 with {\node[fill=white, draw=black, circle, minimum size=2mm, inner sep=0pt, line width=1pt] {};}}}]
    (v6) to[out=270, in=270] (v7);
\end{scope}

\end{tikzpicture}
\end{center}
\caption{Algorithm \ref{alg:lattice-word-to-smaller web} applied to the 5 unique lattice words resulting from plane partitions in $\tsscpp(6,6,6)$}
\label{fig:6-ncms from sl4 to sl2 map}
\end{figure}

We now focus on the web $W$ with the word $\omega = 1~1~1~2~\barfour~2~\barfour~2~(34)~4~4~(34)~(34)$ from the top right of Figure \ref{fig:6-ncms from sl4 to sl2 map} and give an example of a monomial in the invariant $[W]_q.$ We compare it to the web $\widehat{W}$ obtained by applying Algorithm \ref{alg:lattice-word-to-smaller web} to the plane partition in bijection with $W$ and give a monomial in the invariant $[\widehat{W}]_q.$ The algorithm maps the web $W\to\widehat{W}$ as visually demonstrated in Figure \ref{fig:sl4tosl2}.

The map on invariant spaces is given by:
\[
    \Hom_{U_q(\sl_4)}
    \left(
    V_q^4\otimes V_q^*\otimes V_q\otimes V_q^*\otimes V_q^9,\C(q)
    \right)
    \to
    \Hom_{U_q(\sl_2)}
    \left(
    (V_q^*)^2\otimes V_q^6,\C(q)
    \right).
\]
This map sends
\begin{multline*}
    [W]_q = \cdots+(-1)^{42}q^{\mathsf{wgt}_\kappa(W)}
    x_{\textcolor{orange}{1},\textcolor{purple}{1}}
    x_{\textcolor{orange}{1},\textcolor{purple}{2}}
    x_{\textcolor{orange}{1},\textcolor{purple}{3}}
    x_{\textcolor{orange}{2},\textcolor{purple}{4}}
    (
    y_{\textcolor{purple}{5},\textcolor{orange}{4}}
    )
    x_{\textcolor{orange}{2},\textcolor{purple}{6}}
    (
    y_{\textcolor{purple}{7},\textcolor{orange}{4}}
    )
    \\
    (
    x_{\textcolor{orange}{2},\textcolor{purple}{8}}
    x_{\textcolor{orange}{3},\textcolor{purple}{9}}
    x_{\textcolor{orange}{4},\textcolor{purple}{10}}
    x_{\textcolor{orange}{4},\textcolor{purple}{11}}
    x_{\textcolor{orange}{4},\textcolor{purple}{12}}
    x_{\textcolor{orange}{3},\textcolor{purple}{13}}
    x_{\textcolor{orange}{4},\textcolor{purple}{14}}
    x_{\textcolor{orange}{3},\textcolor{purple}{15}}
    x_{\textcolor{orange}{4},\textcolor{purple}{16}}
    )+\cdots
\end{multline*}
to
\[
    [\widehat{W}]_q = \cdots+(-1)^{14}q^{\mathsf{wgt}_\kappa(\widehat{W})}
    (
    y_{\textcolor{purple}{1},\textcolor{orange}{2}}
    y_{\textcolor{purple}{2},\textcolor{orange}{2}}
    )(
    x_{\textcolor{orange}{1},\textcolor{purple}{3}}
    x_{\textcolor{orange}{2},\textcolor{purple}{4}}
    x_{\textcolor{orange}{2},\textcolor{purple}{5}}
    x_{\textcolor{orange}{2},\textcolor{purple}{6}}
    x_{\textcolor{orange}{2},\textcolor{purple}{7}}
    x_{\textcolor{orange}{1},\textcolor{purple}{8}}
    )+\cdots.
\]

\begin{figure}[htbp]
\begin{center}
\begin{tikzpicture}[scale=0.3]
\begin{scope}[shift={(-8,0)}]
\node at (-7,0) {$W = $};
\node (pic) at (5,0) {\includegraphics[scale=0.2]{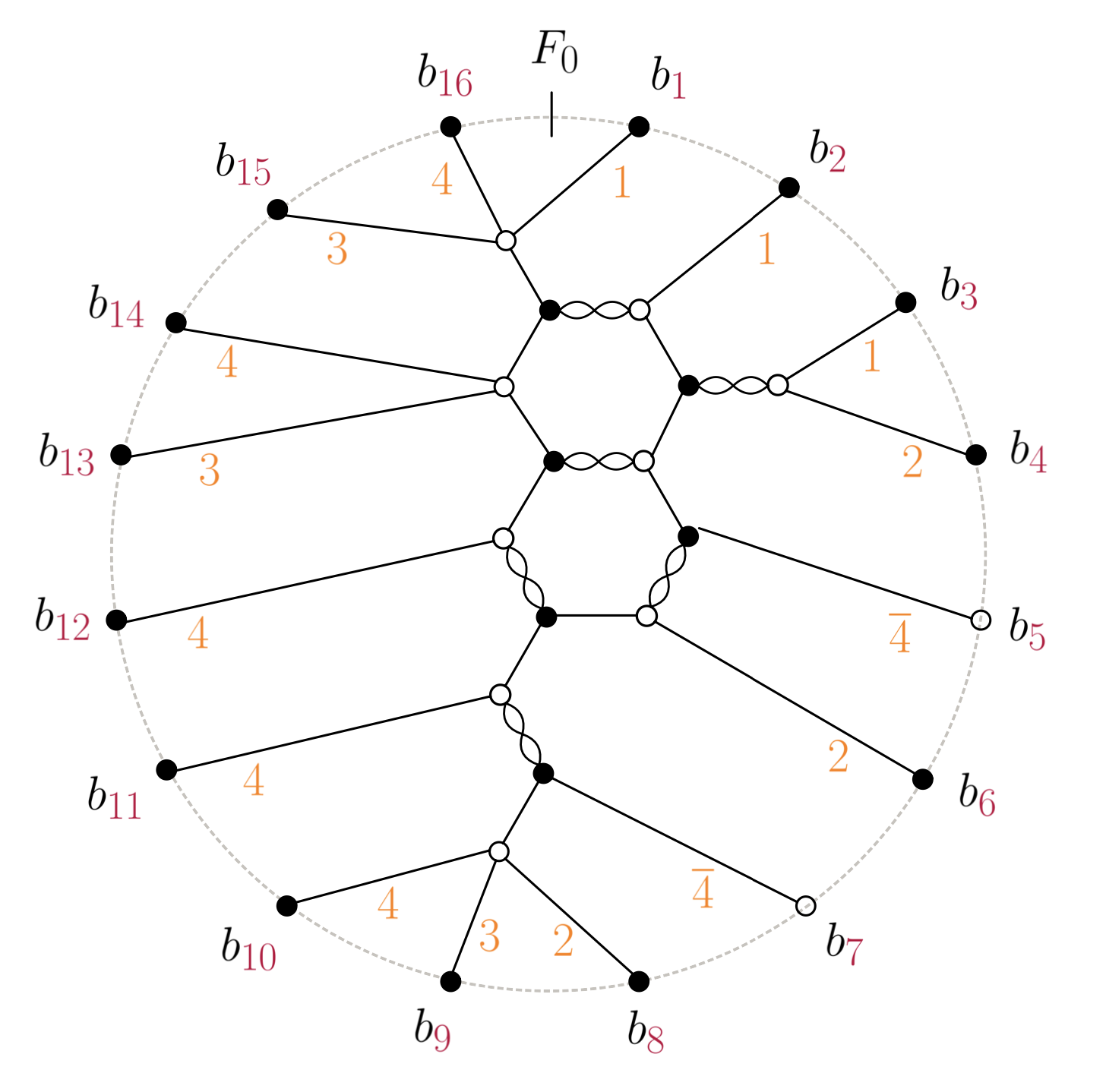}};
\end{scope}

\begin{scope}[shift={(8,0)}]
    \node at (0,0) {$\longrightarrow$};
\end{scope}

\begin{scope}[shift={(21,0)}]
    \node at (-8,0) {$\widehat{W} = $};
    \node (pic) at (0,0) {\includegraphics[scale=0.2]{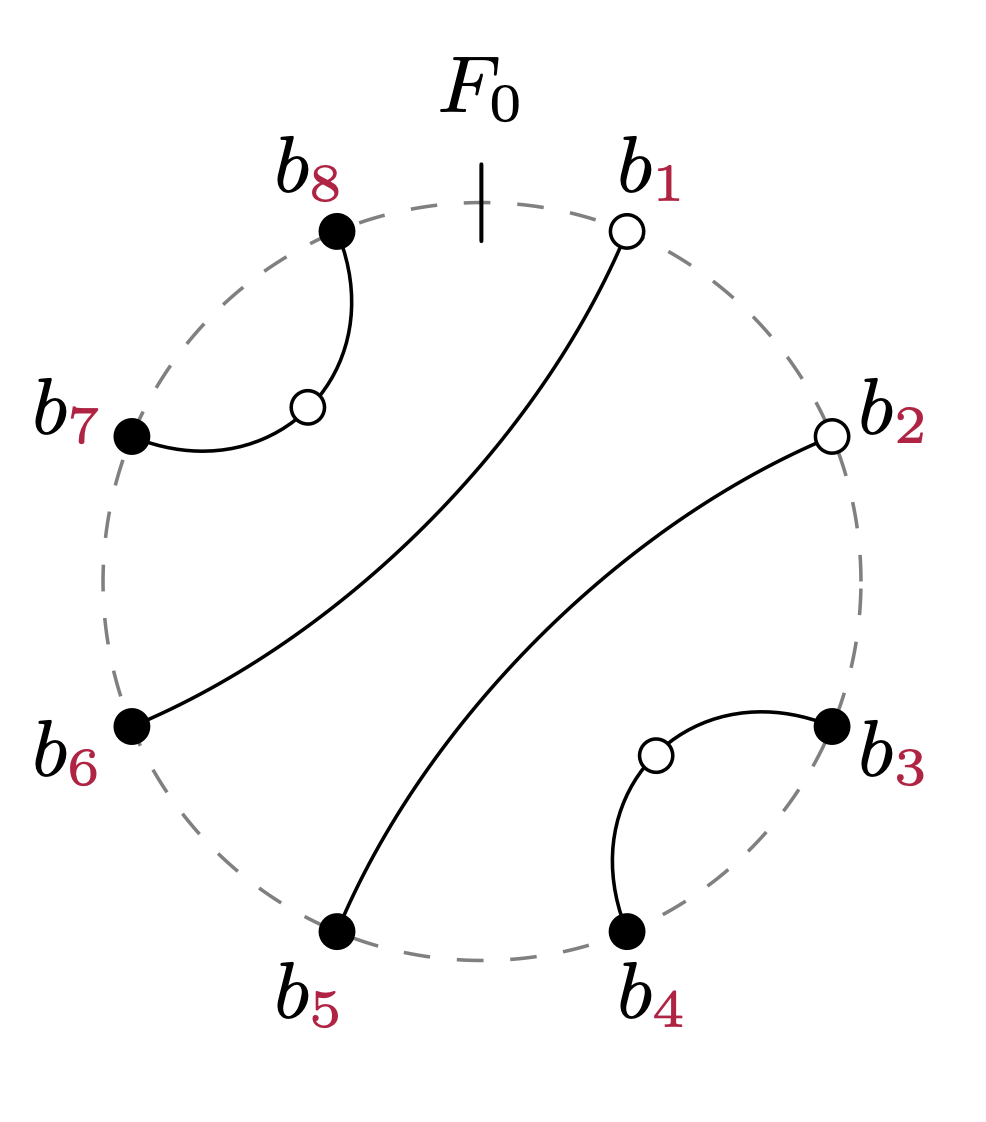}};
\end{scope}
\end{tikzpicture}
\end{center}
\caption{An example of the map from Theorem \ref{thm:big invariant to small invariant} applied to the web with lattice word $\omega_3$ from Figure \ref{fig:6-ncms from sl4 to sl2 map}.}
\label{fig:sl4tosl2}
\end{figure}
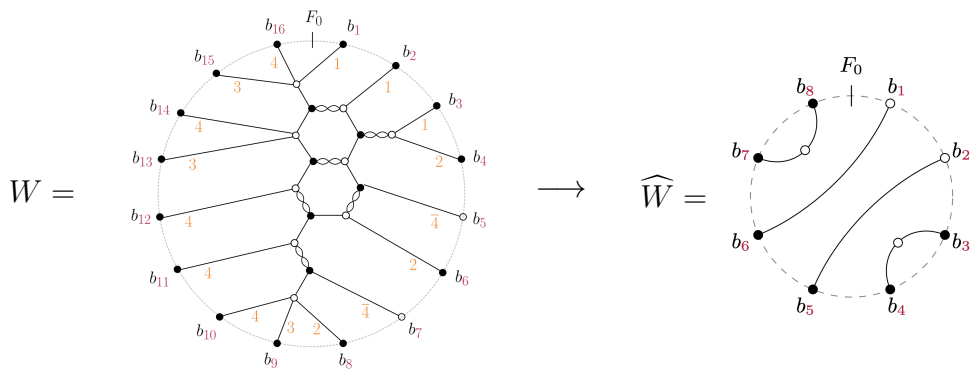

\subsection{Example of Theorem \ref{thm:big invariant to small invariant}: $U_q(\sl_4)$ to $U_q(\sl_3)$}
\label{sec:ex-tab-to-sl3}

Let $p\in\tspp(4,4,4)$ be the plane partition from Figure \ref{fig:TSPP-ex for proof} and denote the web in bijection with the fundamental domain of $p$ as $W.$ The lattice word of $W$ is
\[
\omega = 1~1~1~1~(23)~2~(23)~2~4~(34)~4~(34).
\]

In Algorithm \ref{alg:lattice-word-to-smaller web}, the resulting word is
\[
\widehat{\omega} = (12)~1~(12)~1~3~(23)~3~(23).
\]
Since every letter of $\widehat{\omega}$ is positive, $\widehat{\omega}$ is the lattice word of a standard Young tableau of rectangular shape $(4,4,4),$ and every boundary vertex of $\widehat{W}$ is black. Using the $U_q(\sl_3)$ growth rules (\cite[Fig.~3]{petersen2009promotion}), we can construct the web $\widehat{W}$ as seen in Figure \ref{fig:sl4tosl3}.
The map on invariant spaces is
$
    \Hom_{U_q(\sl_4)}
    \left(
    V_q^{16},\C(q)
    \right)
    \to
    \Hom_{U_q(\sl_3)}
    \left(
    V_q^{12},\C(q)
    \right).
$

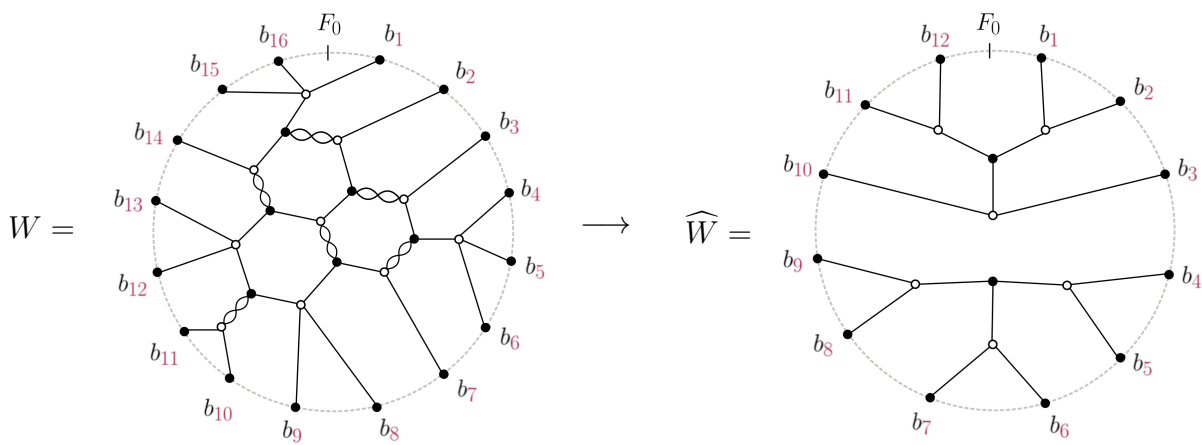
\begin{figure}[htbp]
\begin{center}
\begin{tikzpicture}[scale=0.3]
\begin{scope}[shift={(-8,0)}]
\node at (-9,0) {$W = $};
\node (pic) at (4,0) {\includegraphics[scale=0.25]{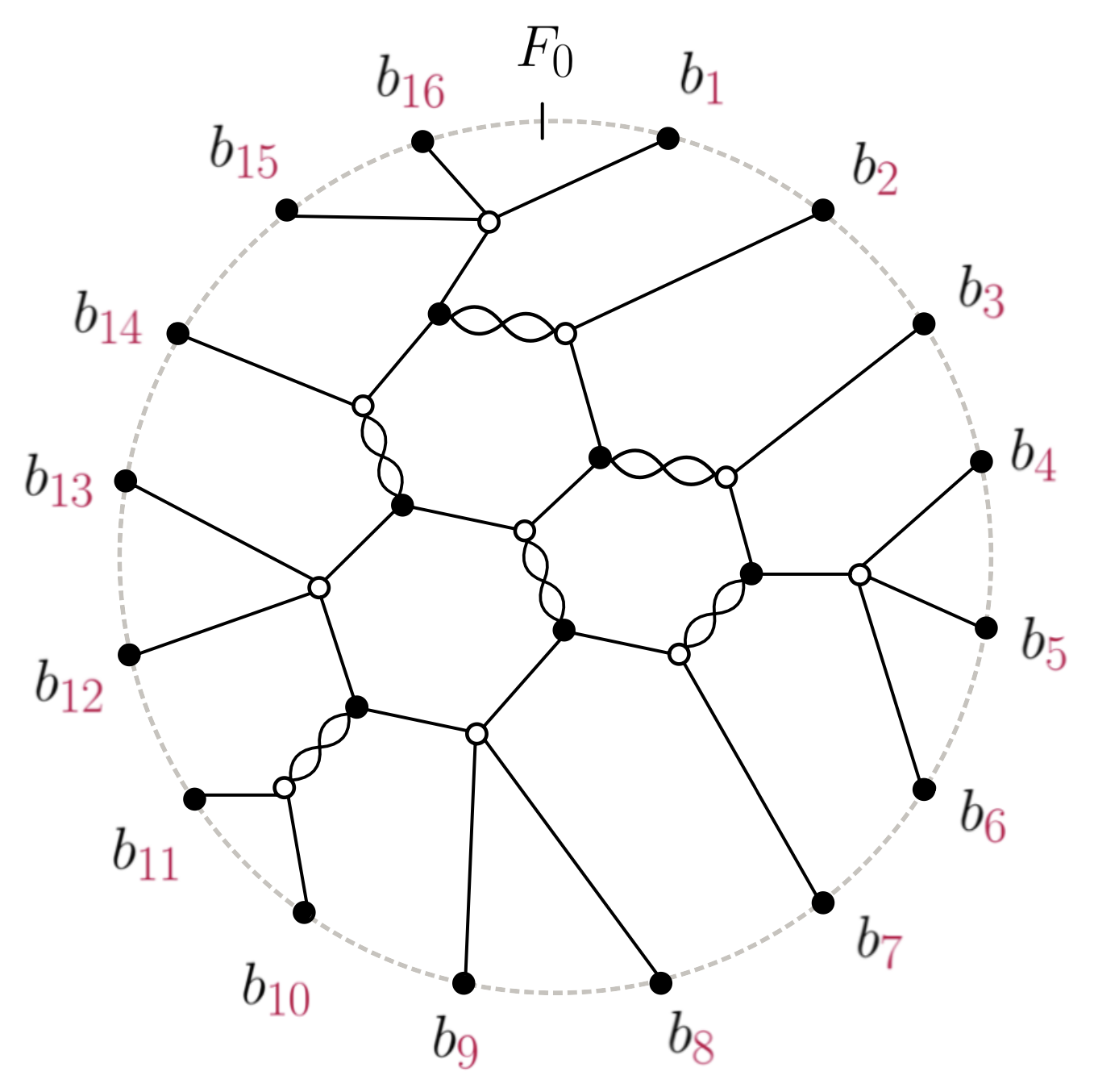}};
\end{scope}

\begin{scope}[shift={(8,0)}]
    \node at (0,0) {$\longrightarrow$};
\end{scope}

\begin{scope}[shift={(22,0)}]
    \node at (-9,0) {$\widehat{W} = $};
    \node (pic) at (3,0) {\includegraphics[scale=0.25]{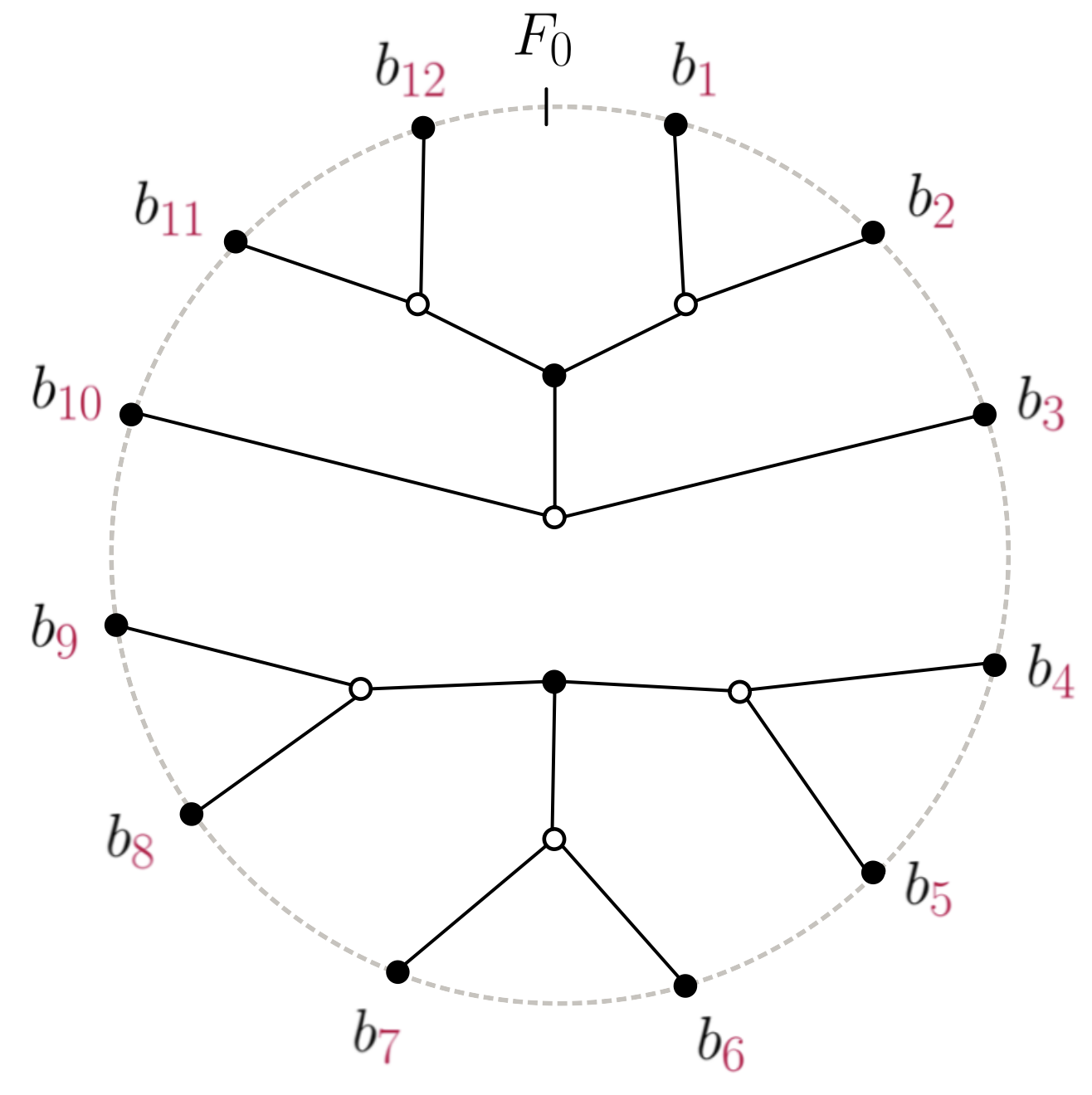}};
\end{scope}
\end{tikzpicture}
\end{center}
\caption{An example of the map from Theorem \ref{thm:big invariant to small invariant} applied to the web from Figure \ref{fig:TSPP-ex for proof}.}
\label{fig:sl4tosl3}
\end{figure}

\newpage
\bibliographystyle{amsalpha}
\bibliography{refs}

@article {Stanley-symmetries,
    AUTHOR = {Stanley, Richard P.},
     TITLE = {Symmetries of plane partitions},
   JOURNAL = {J. Combin. Theory Ser. A},
  FJOURNAL = {Journal of Combinatorial Theory. Series A},
    VOLUME = {43},
      YEAR = {1986},
    NUMBER = {1},
     PAGES = {103--113},
      ISSN = {0097-3165,1096-0899},
   MRCLASS = {05A17 (05A15)},
  MRNUMBER = {859302},
MRREVIEWER = {W.\ H.\ Mills},
       DOI = {10.1016/0097-3165(86)90028-2},
       URL = {https://doi.org/10.1016/0097-3165(86)90028-2},
}

@article {kuperberg,
    AUTHOR = {Kuperberg, Greg},
     TITLE = {Spiders for rank {$2$} {L}ie algebras},
   JOURNAL = {Comm. Math. Phys.},
  FJOURNAL = {Communications in Mathematical Physics},
    VOLUME = {180},
      YEAR = {1996},
    NUMBER = {1},
     PAGES = {109--151},
      ISSN = {0010-3616,1432-0916},
   MRCLASS = {17B10 (22E60 81R05)},
  MRNUMBER = {1403861},
MRREVIEWER = {Stefano\ Capparelli},
       URL = {http://projecteuclid.org/euclid.cmp/1104287237},
}

@article{kenyon2005,
    author = {Kenyon, Richard and Okounkov, Andrei},
    title = {What is... a dimer?},
    journal = {Notices Amer. Math. Soc.},
    volume = {52},
    pages = {342–343},
    number = {3},
    year = {2005}
}

@article{gaetz2023rotation,
author={Gaetz, Christian and Pechenik, Oliver and Pfannerer, Stephan and Striker, Jessica and Swanson, Joshua P.},
title={Rotation-invariant web bases from hourglass plabic graphs},
year={2026},
journal={Invent. Math.},
fjournal = {Inventiones mathematicae},
volume={243},
}

@article{postnikov2006total,
  title={Total positivity, {G}rassmannians, and networks},
  author={Postnikov, Alexander},
  journal={arXiv preprint math/0609764},
  year={2006}
}

@article {kuperberg2002symmetry,
    AUTHOR = {Kuperberg, Greg},
     TITLE = {Symmetry classes of alternating-sign matrices under one roof},
   JOURNAL = {Ann. of Math. (2)},
  FJOURNAL = {Annals of Mathematics. Second Series},
    VOLUME = {156},
      YEAR = {2002},
    NUMBER = {3},
     PAGES = {835--866},
      ISSN = {0003-486X,1939-8980},
   MRCLASS = {05A15 (05A30 05E05 11C20 15A36)},
  MRNUMBER = {1954236},
MRREVIEWER = {Christian\ Krattenthaler},
       DOI = {10.2307/3597283},
       URL = {https://doi.org/10.2307/3597283},
}

@inproceedings{macmahon1899partitions,
  title={Partitions of numbers whose graphs possess symmetry},
  author={MacMahon, Percy Alexander},
  year={1899},
  organization={Proceedings of the Cambridge Philosophical Society}
}

@book{macmahon2001combinatory,
  title={Combinatory analysis, volumes {I} and {II}},
  author={MacMahon, Percy A},
  volume={137},
  year={2001},
  publisher={American Mathematical Society}
}

@article {cautis2014webs,
    AUTHOR = {Cautis, Sabin and Kamnitzer, Joel and Morrison, Scott},
     TITLE = {Webs and quantum skew {H}owe duality},
   JOURNAL = {Math. Ann.},
  FJOURNAL = {Mathematische Annalen},
    VOLUME = {360},
      YEAR = {2014},
    NUMBER = {1-2},
     PAGES = {351--390},
      ISSN = {0025-5831,1432-1807},
   MRCLASS = {17B37 (17B10 17B20)},
  MRNUMBER = {3263166},
MRREVIEWER = {Iwan\ Praton},
       DOI = {10.1007/s00208-013-0984-4},
       URL = {https://doi.org/10.1007/s00208-013-0984-4},
}

@book {morrison2007diagrammatic,
    AUTHOR = {Morrison, Scott Edward},
     TITLE = {A diagrammatic category for the representation theory of
              {$U_q(\mathfrak{sl}_n)$}},
      NOTE = {Thesis (Ph.D.)--University of California, Berkeley},
 PUBLISHER = {ProQuest LLC, Ann Arbor, MI},
      YEAR = {2007},
     PAGES = {89},
      ISBN = {978-0549-17102-7},
   MRCLASS = {99-05},
  MRNUMBER = {2710589},
       URL =
              {http://gateway.proquest.com/openurl?url_ver=Z39.88-2004&rft_val_fmt=info:ofi/fmt:kev:mtx:dissertation&res_dat=xri:pqdiss&rft_dat=xri:pqdiss:3275527},
}

@article {fraser2019dimers,
    AUTHOR = {Fraser, Chris and Lam, Thomas and Le, Ian},
     TITLE = {From dimers to webs},
   JOURNAL = {Trans. Amer. Math. Soc.},
  FJOURNAL = {Transactions of the American Mathematical Society},
    VOLUME = {371},
      YEAR = {2019},
    NUMBER = {9},
     PAGES = {6087--6124},
      ISSN = {0002-9947,1088-6850},
   MRCLASS = {05E10 (05C10 14M15 20C30)},
  MRNUMBER = {3937319},
       DOI = {10.1090/tran/7641},
       URL = {https://doi.org/10.1090/tran/7641},
}

@article {gaetz2023promotion,
    AUTHOR = {Gaetz, Christian and Pechenik, Oliver and Pfannerer, Stephan
              and Striker, Jessica and Swanson, Joshua P.},
     TITLE = {Promotion permutations for tableaux},
   JOURNAL = {Comb. Theory},
  FJOURNAL = {Combinatorial Theory},
    VOLUME = {4},
      YEAR = {2024},
    NUMBER = {2},
     PAGES = {Paper No. 15, 56},
      ISSN = {2766-1334},
   MRCLASS = {05E10 (05E18)},
  MRNUMBER = {4807154},
}

@article {hopkins2022promotion,
    AUTHOR = {Hopkins, Sam and Rubey, Martin},
     TITLE = {Promotion of {K}reweras words},
   JOURNAL = {Selecta Math. (N.S.)},
  FJOURNAL = {Selecta Mathematica. New Series},
    VOLUME = {28},
      YEAR = {2022},
    NUMBER = {1},
     PAGES = {Paper No. 10, 38},
      ISSN = {1022-1824,1420-9020},
   MRCLASS = {05E18 (05E10 06A07)},
  MRNUMBER = {4346507},
MRREVIEWER = {Vedrana\ Mikuli\'c{} Crnkovi\'c},
       DOI = {10.1007/s00029-021-00714-6},
       URL = {https://doi.org/10.1007/s00029-021-00714-6},
}

@article {petersen2009promotion,
    AUTHOR = {Petersen, T. Kyle and Pylyavskyy, Pavlo and Rhoades, Brendon},
     TITLE = {Promotion and cyclic sieving via webs},
   JOURNAL = {J. Algebraic Combin.},
  FJOURNAL = {Journal of Algebraic Combinatorics. An International Journal},
    VOLUME = {30},
      YEAR = {2009},
    NUMBER = {1},
     PAGES = {19--41},
      ISSN = {0925-9899,1572-9192},
   MRCLASS = {05E10},
  MRNUMBER = {2519848},
MRREVIEWER = {Gregory\ S.\ Warrington},
       DOI = {10.1007/s10801-008-0150-3},
       URL = {https://doi.org/10.1007/s10801-008-0150-3},
}

@article{poudel2022comparison,
  title={A comparison between {$\mathrm{SL}_n$} spider categories},
  author={Poudel, Anup},
  journal={Canadian Journal of Mathematics},
  volume={Web},
  year={2025},
  pages={1--30}
}

@book {kim2003graphical,
    AUTHOR = {Kim, Dongseok},
     TITLE = {Graphical calculus on representations of quantum {L}ie
              algebras},
      NOTE = {Thesis (Ph.D.)--University of California, Davis},
 PUBLISHER = {ProQuest LLC, Ann Arbor, MI},
      YEAR = {2003},
     PAGES = {55},
      ISBN = {978-0496-30399-1},
   MRCLASS = {99-05},
  MRNUMBER = {2704398},
       URL =
              {http://gateway.proquest.com/openurl?url_ver=Z39.88-2004&rft_val_fmt=info:ofi/fmt:kev:mtx:dissertation&res_dat=xri:pqdiss&rft_dat=xri:pqdiss:3082548},
}

@article{andrews1978plane,
  title={Plane partitions ({I}): {T}he {M}ac{M}ahon conjecture},
  author={Andrews, George E},
  journal={Studies in Foundations and Combinatorics, Advances in Mathematics Supplementary Studies},
  volume={1},
  pages={131--150},
  year={1978}
}

@article{andrews1979plane,
  title={Plane partitions ({III}): {T}he weak {M}acdonald conjecture},
  author={Andrews, George E},
  journal={Inventiones mathematicae},
  volume={53},
  number={3},
  pages={193--225},
  year={1979},
  publisher={Springer-Verlag Berlin/Heidelberg}
}

@book {macdonald1979symmetric,
    AUTHOR = {Macdonald, I. G.},
     TITLE = {Symmetric functions and {H}all polynomials},
    SERIES = {Oxford Mathematical Monographs},
 PUBLISHER = {The Clarendon Press, Oxford University Press, New York},
      YEAR = {1979},
     PAGES = {viii+180},
      ISBN = {0-19-853530-9},
   MRCLASS = {05-02 (12-02 20C30 20K01)},
  MRNUMBER = {553598},
MRREVIEWER = {Ira\ Gessel},
}

@article {stembridge1995enumeration,
    AUTHOR = {Stembridge, John R.},
     TITLE = {The enumeration of totally symmetric plane partitions},
   JOURNAL = {Adv. Math.},
  FJOURNAL = {Advances in Mathematics},
    VOLUME = {111},
      YEAR = {1995},
    NUMBER = {2},
     PAGES = {227--243},
      ISSN = {0001-8708,1090-2082},
   MRCLASS = {05A15 (05A17)},
  MRNUMBER = {1318529},
MRREVIEWER = {Christian\ Krattenthaler},
       DOI = {10.1006/aima.1995.1023},
       URL = {https://doi.org/10.1006/aima.1995.1023},
}

@article{andrews1994plane,
  title={Plane partitions {V}: {T}he {TSSCPP} conjecture},
  author={Andrews, George E},
  journal={Journal of Combinatorial Theory, Series A},
  volume={66},
  number={1},
  pages={28--39},
  year={1994},
  publisher={Elsevier}
}

@article{weyl1932valenztheorie,
  title={Eine f{\"u}r die {V}alenztheorie geeignete {B}asis der bin{\"a}ren {V}ektorinvarianten},
  author={Weyl, H and Rumer, G and Teller, E},
  journal={Nachrichten von der Gesellschaft der Wissenschaften zu G{\"o}ttingen, Mathematisch-Physikalische Klasse},
  volume={1932},
  pages={499--504},
  year={1932}
}

@article{temperley2004relations,
  title={Relations between the ‘percolation’ and ‘colouring’ problem and other graph-theoretical problems associated with regular planar lattices: some exact results for the ‘percolation’ problem},
  author={Temperley, Harold NV and Lieb, Elliott H},
  journal={Condensed Matter Physics and Exactly Soluble Models: Selecta of Elliott H. Lieb},
  pages={475--504},
  year={2004},
  publisher={Springer}
}

@article {di2004refined,
    AUTHOR = {Di Francesco, P.},
     TITLE = {A refined {R}azumov-{S}troganov conjecture},
   JOURNAL = {J. Stat. Mech. Theory Exp.},
  FJOURNAL = {Journal of Statistical Mechanics: Theory and Experiment},
      YEAR = {2004},
    NUMBER = {8},
     PAGES = {009, 16},
      ISSN = {1742-5468},
   MRCLASS = {82B20 (05A15)},
  MRNUMBER = {2115252},
}

@article{elias2015light,
  title={Light ladders and clasp conjectures},
  author={Elias, Ben},
  journal={arXiv preprint arXiv:1510.06840},
  year={2015}
}

@article {fontaine2013buildings,
    AUTHOR = {Fontaine, Bruce and Kamnitzer, Joel and Kuperberg, Greg},
     TITLE = {Buildings, spiders, and geometric {S}atake},
   JOURNAL = {Compos. Math.},
  FJOURNAL = {Compositio Mathematica},
    VOLUME = {149},
      YEAR = {2013},
    NUMBER = {11},
     PAGES = {1871--1912},
      ISSN = {0010-437X,1570-5846},
   MRCLASS = {17B10 (18D10 51E24 57M27)},
  MRNUMBER = {3133297},
MRREVIEWER = {K.\ Strambach},
       DOI = {10.1112/S0010437X13007136},
       URL = {https://doi.org/10.1112/S0010437X13007136},
}

@article {westbury2012web,
    AUTHOR = {Westbury, Bruce W.},
     TITLE = {Web bases for the general linear groups},
   JOURNAL = {J. Algebraic Combin.},
  FJOURNAL = {Journal of Algebraic Combinatorics. An International Journal},
    VOLUME = {35},
      YEAR = {2012},
    NUMBER = {1},
     PAGES = {93--107},
      ISSN = {0925-9899,1572-9192},
   MRCLASS = {17B37 (05E10 17B10 20G43)},
  MRNUMBER = {2873098},
       DOI = {10.1007/s10801-011-0294-4},
       URL = {https://doi.org/10.1007/s10801-011-0294-4},
}

@book{hagemeyer2018spiders,
  title={Spiders and generalized confluence},
  author={Hagemeyer, Colin Scott},
  year={2018},
  publisher={University of California, Davis}
}

@article{gessel1985binomial,
  title={Binomial determinants, paths, and hook length formulae},
  author={Gessel, Ira and Viennot, G{\'e}rard},
  journal={Advances in Mathematics},
  volume={58},
  number={3},
  pages={300--321},
  year={1985}
}

\end{document}